\documentclass{article}
\usepackage{amsmath, amssymb, amsthm, amscd, amsfonts, graphicx, color, mathrsfs,makeidx,enumerate,caption}
\makeindex
\input epsf.tex
\usepackage{latexsym}
\usepackage{url, multicol, hyperref}
\makeindex

\theoremstyle{plain}
\newtheorem{theorem}{Theorem}[section]
\newtheorem{lemma}[theorem]{Lemma}
\newtheorem{proposition}[theorem]{Proposition}
\newtheorem{corollary}[theorem]{Corollary}

\theoremstyle{definition}
\newtheorem{exercise}[theorem]{Exercise}

\newtheorem{example}[theorem]{Example}
\newtheorem{definition}[theorem]{Definition}
\newtheorem{algorithm}[theorem]{Algorithm}

\newcommand{\wimplies}{\ \ \ \implies \ \ \  }
\newcommand{\shear}{\left[\begin{smallmatrix}1&1 \\ 0 &1 \end{smallmatrix}\right]}
\newcommand{\upshear}{\left[\begin{smallmatrix}1&0 \\ 1 &1 \end{smallmatrix}\right]}
\newcommand{\nupshear}{\left[\begin{smallmatrix}1&0 \\ -1 &1 \end{smallmatrix}\right]}
\newcommand{\stt}[4]{\left[\begin{smallmatrix}#1 & #2 \\  #3 & #4 \end{smallmatrix}\right]}
\newcommand{\lst}{$\mathsf{L}$-shaped table }

\begin{document}

\title{Lines in positive genus: \\ An introduction to flat surfaces}
\author{Diana Davis}




\maketitle

\noindent {\bf \Large Preface}

This text is aimed at undergraduates, or anyone else who enjoys thinking about shapes and numbers. The goal is to encourage the student to think deeply about seemingly simple things. The main objects of study are lines, squares, and the effects of simple geometric motions on them. Much of the beauty of this subject is explained through the text and the figures, and some of it is left for the student to discover in the exercises. We want readers to ``get their hands dirty'' by thinking about examples and working exercises, and  to discover the elegance and richness of this area of mathematics.

\newpage
\tableofcontents

\newpage
\section{Billiards} \label{squaretorus} 


There is a lot of rich mathematics in the study of \emph{billiards}, \index{billiards} a ball bouncing around inside a billiard table. In the game of billiards, the table is rectangular, but we can imagine any shape of table $-$ a triangle, a circle, a $\mathsf{W}$, an infinite sector, or any other shape.\footnote{Sections \ref{squaretorus} and \ref{st} first appeared in \cite{snapshot}.}

Consider the simplest case of a polygonal table: a square. We will assume that the ball is just a point, moving with no friction (so it goes forever), and that when it hits the edge of the table, the angle of reflection is equal to the angle of incidence, as in real life. 

Is it possible to hit the ball so that it repeats its path? Yes: If we hit it vertically or horizontally, it will bounce back and forth between two points on parallel edges (Figure \ref{squaretraj}a). We say that this trajectory is \emph{periodic}, with \emph{period $2$}. Other examples, with period $4$ and $6$, respectively, are in (b) and (c). \\

\begin{figure}[!h]
\begin{center}
\includegraphics[width=50pt]{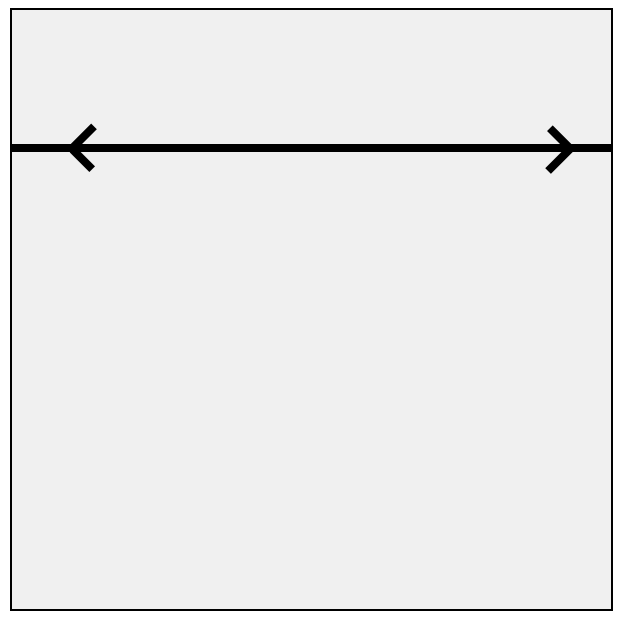} \ \ 
\includegraphics[width=50pt]{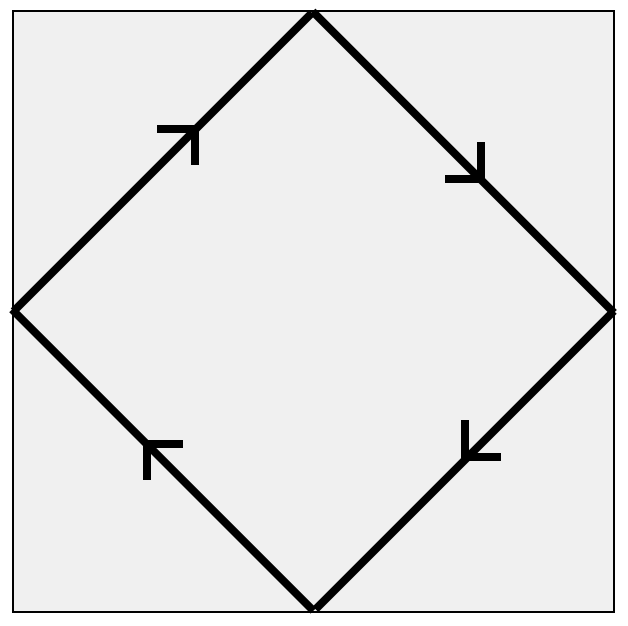} \ \
\includegraphics[width=50pt]{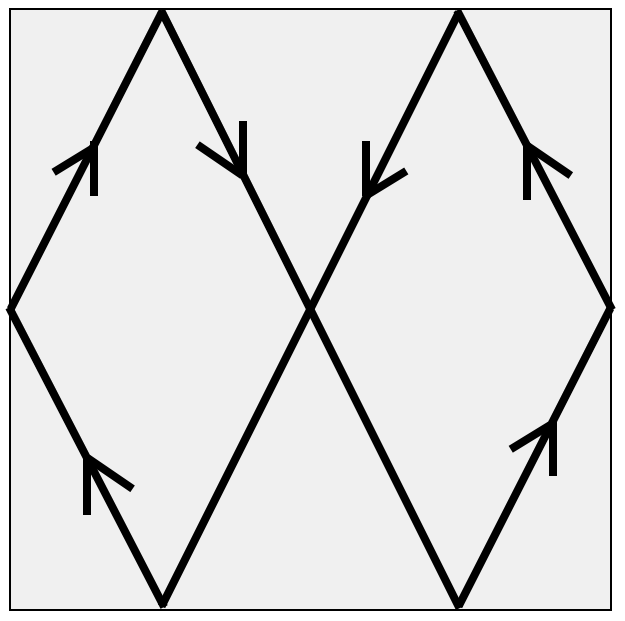} \ \
\includegraphics[width=50pt]{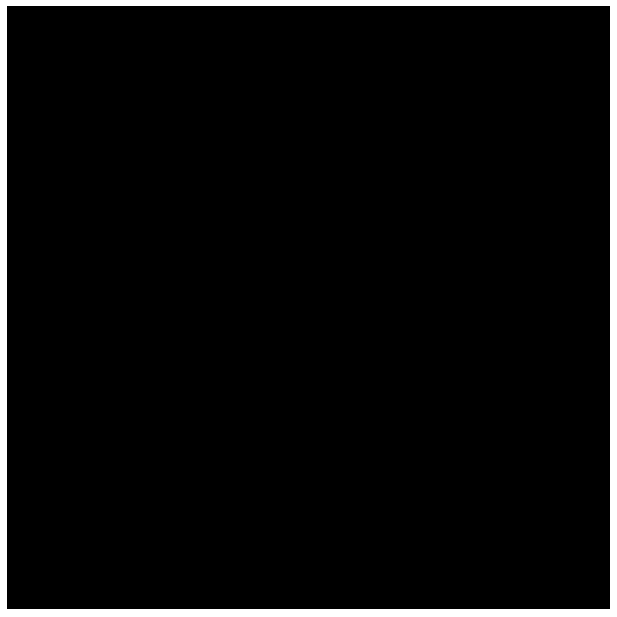} \\
\caption{Trajectories with period $2,4$ and $6$, and a non-periodic trajectory \label{squaretraj}}
\end{center}
\end{figure}

Is it possible to hit the ball so that it \emph{never} repeats its path? It is more difficult to draw a picture of an example of such a \emph{non-periodic} trajectory, because the trajectory never repeats, so it will gradually fill up the table until the picture is a black square (d). However, if we don't restrict ourselves to just the table, we can draw the trajectory, by \emph{unfolding} the table: \index{unfolding}

Consider the simple trajectory in Figure \ref{unfold-line}a below. When the ball hits the top edge, instead of having it bounce and go downwards, we \emph{unfold} the table upward, creating another copy of the table in which the ball can keep going straight (b). In other words, rather than reflecting the \emph{ball} against the top edge, we reflect the whole \emph{table} across the top edge and let the ball go straight.

Now when the trajectory hits the right edge, we do the same thing: we unfold the table to the right, creating another copy of the table in which the ball can keep going straight (c). \index{unfolding}
We can keep doing this, creating a new square every time the trajectory crosses an edge. In this way, a trajectory on the square table is represented as a line on a piece of graph paper. \\

\begin{figure}[!h]
\begin{center}
\includegraphics[width=60pt]{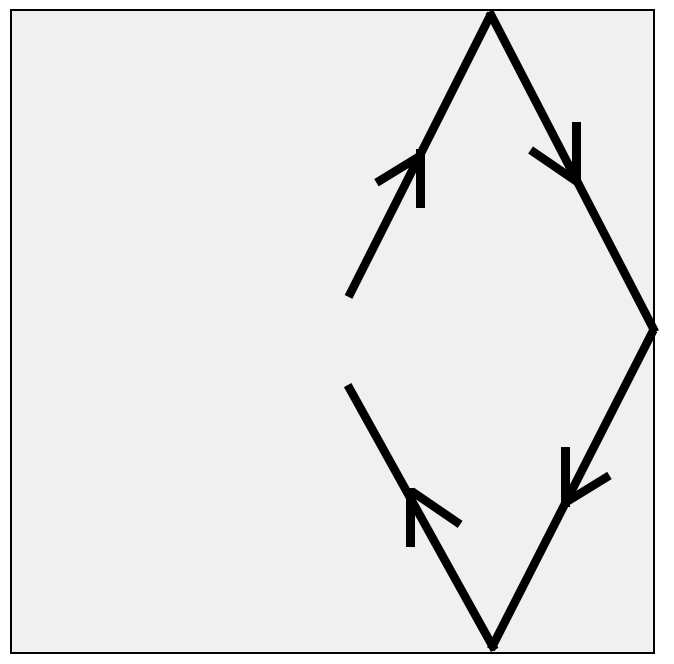} \ \ 
\includegraphics[width=60pt]{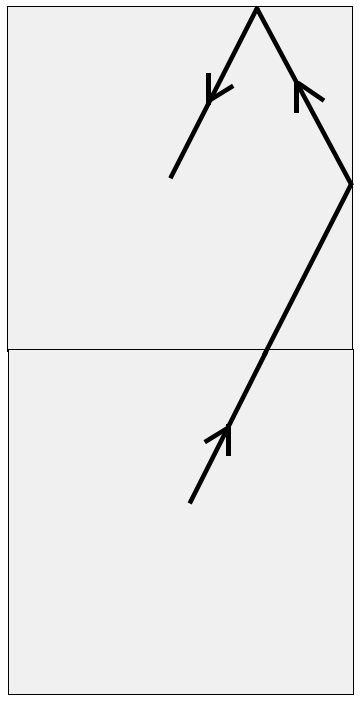} \ \ 
\includegraphics[width=115pt]{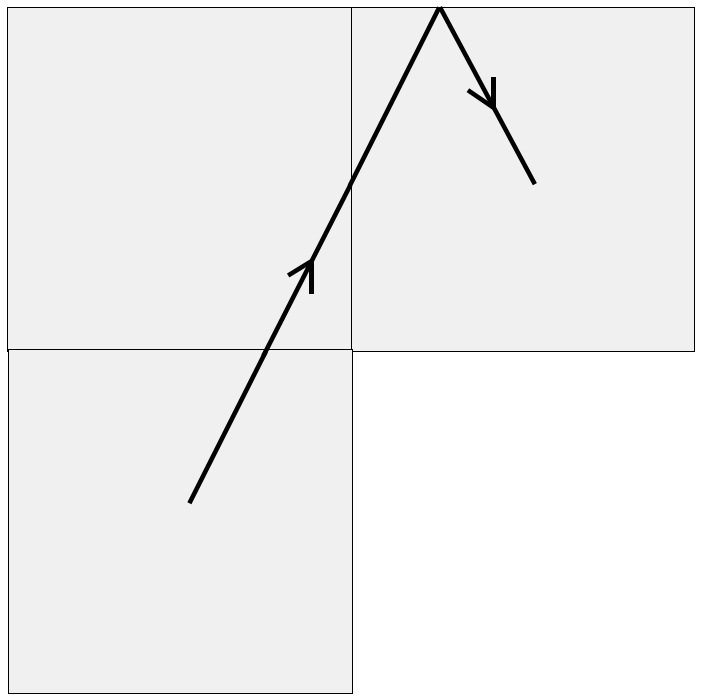} \ \ 
\includegraphics[width=68pt]{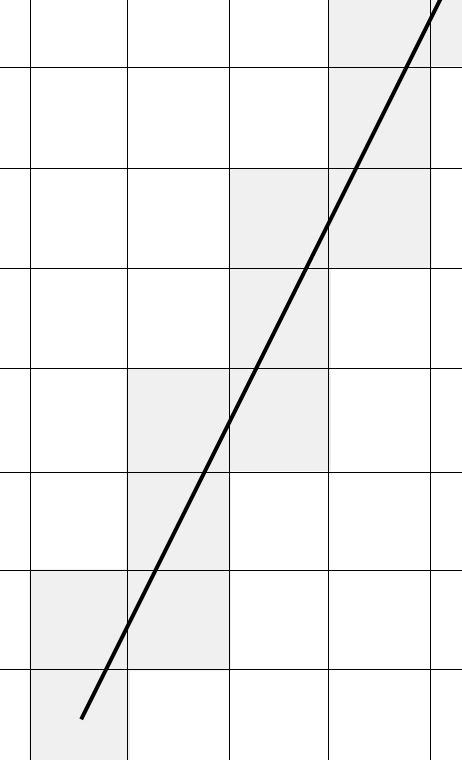} \ \ \\
\caption{Unfolding a trajectory on the square table into a straight line \label{unfold-line}}
\end{center}
\end{figure}

By thinking of the trajectory as a line on graph paper, we can easily find a non-periodic trajectory. Suppose that we draw a line with an irrational slope. Then it will never cross two different horizontal (or vertical) edges at the same point: If it did, then the slope between those corresponding points would be a ratio of two integers, but we chose the slope to be irrational, so this can't happen. So if we hit the ball with \emph{any} irrational slope, its trajectory in the table will be non-periodic. By a similar argument, if we hit the ball with any \emph{rational} slope, its trajectory in the table will be periodic. 

\begin{exercise} \label{circular-table}\index{billiards}
Draw several examples of billiard trajectories on a circular billiard table. Describe the behavior.
\end{exercise}

\begin{exercise} \label{sector-table}\index{billiards}
Consider a billiard ``table'' in the shape of an infinite sector with a small vertex angle, say $10^\circ$. Draw several examples of billiard trajectories in this sector (calculate the angles at each bounce so that your sketch is accurate). Is it possible for the trajectory to go in toward the vertex and get ``stuck''? Find an example of a trajectory that does this, or explain why it cannot happen.
\end{exercise}

\begin{exercise} \label{periodic-table}\index{billiards}
Write down the proof that a trajectory on the square billiard table is periodic if and only if its slope is rational.
\end{exercise}

\begin{exercise}\label{periodeight}\index{billiards}
Construct a periodic billiard path of period $8$ on a square table. How many can you find?
\end{exercise}

\begin{exercise}\label{periodten}\index{billiards}
Construct two different periodic billiard paths of period $10$ on a square table. 
\end{exercise}

\begin{exercise}\label{oddperiod}\index{billiards}
Construct a periodic billiard path on a square table with an odd period, or show that it is not possible to do so.
\end{exercise}

We will explore Exercises $\ref{periodeight}-\ref{oddperiod}$ in greater detail in Section \ref{revisit}.

\newpage
\section{The square torus} \label{st}\index{square torus}

Here's another way that we can unfold the square billiard table.\index{unfolding}
First, unfold across the top edge of the table, creating another copy in which the ball keeps going straight (Figure \ref{unfold-torus}). The new top edge is just a copy of the bottom edge, so we now label them both $A$ to remember that they are the same. Similarly, we can unfold across the right edge of the table, creating another copy of the unfolded table, which gives us four copies of the original table. The new right edge is a copy of the left edge, so we now label them both $B$. 

\begin{figure}[!h]
\begin{center}
\includegraphics[width=66pt]{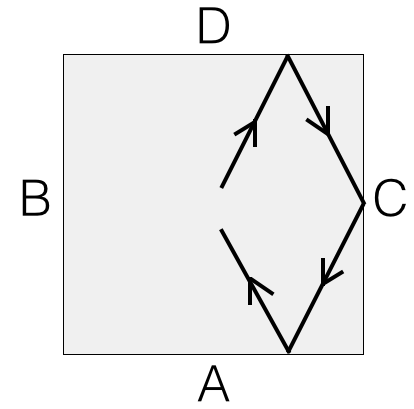} \ \ \ \ \  
\includegraphics[width=66pt]{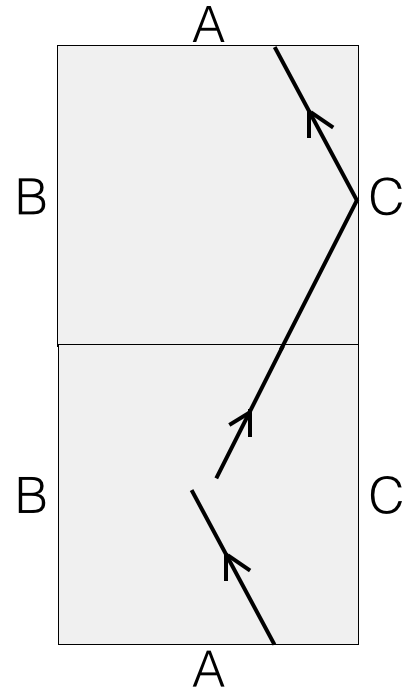} \ \ \ \ \   
\includegraphics[width=115pt]{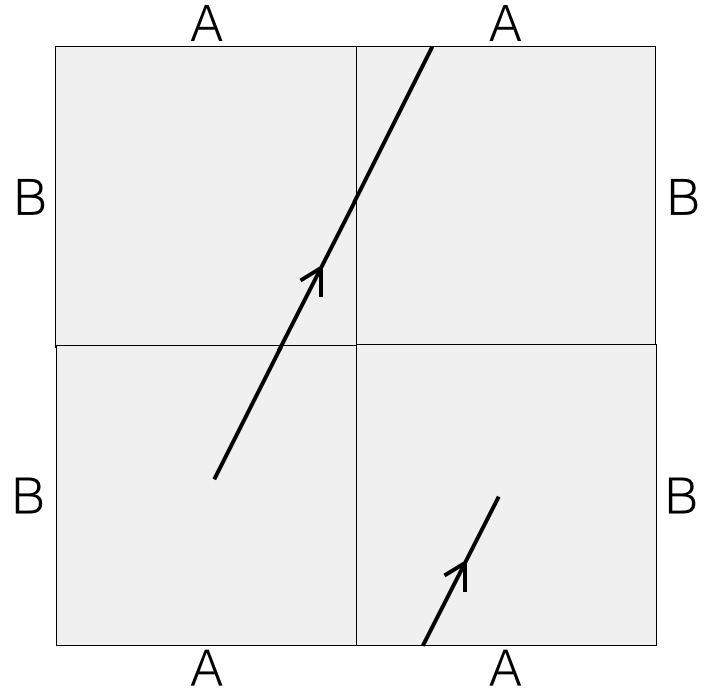} 
\caption{Unfolding the square table into the square torus \label{unfold-torus}}
\end{center}
\end{figure}

When the trajectory hits the top edge $A$, it reappears in the same place on the bottom edge $A$ and keeps going. Similarly, when the trajectory hits the right edge $B$, it reappears on the left edge $B$. This is called \emph{identifying} the top and bottom edges, and \emph{identifying} the left and right edges, of the square.\index{square torus}

You may be familiar with this idea of entering the top wall and re-emerging from the bottom wall from the video games ``Pac-Man,'' ``Snake,'' and ``Portal.'' However, you may not have realized that with these edge identifications, you are no longer on the flat plane, but are on an entirely different flat surface! This surface is actually the surface of a bagel or a donut, which is called a \emph{torus}. 

The square torus has many beautiful properties, and we will explore several of them in depth. First, let's see what the surface looks like: when we glue both copies of edge $A$ to each other, we get a cylinder, whose ends are both edge $B$ (Figure \ref{curved-torus}b). When we wrap the cylinder around to glue both copies of edge $B$ to each other (Figure \ref{curved-torus}c), we get a torus!  \index{torus}\index{square torus}

\begin{figure}[!h]
\begin{center}
\includegraphics[width=60pt]{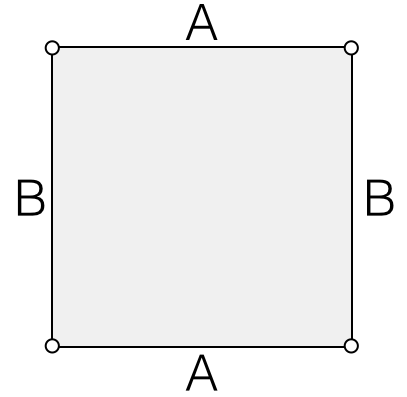} \ \ \ 
\includegraphics[width=150pt]{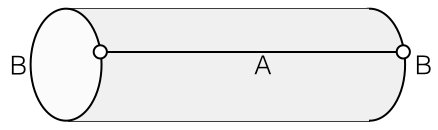} \ \ \ 
\includegraphics[width=120pt]{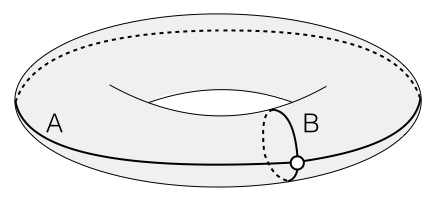} \\
\caption{Gluing the square torus into a 3D torus \label{curved-torus}}
\end{center}
\end{figure}

In practice, we'll just leave the torus as a flat square, and remember that the two pairs of parallel edges are identified. It's much easier to draw paths on a flat square than on the curved, three-dimensional picture. We'll also assume that every non-vertical trajectory on the square torus goes from left to right, so that we can omit arrows on trajectories from now on. \index{square torus}

Notice that when we identify the top and bottom edges, and identify the left and right edges, the four corners of the square all come together at one point. So the square torus actually has just one vertex (the circle in Figure \ref{curved-torus}c) $-$ all four corners of the square in Figure \ref{curved-torus}a come together at the same vertex. The angle at this vertex is $4\cdot\frac{\pi}2=2\pi$, so they fit together to make a flat angle. To a small creature living on the surface, the marked point  looks no different from any other point on the torus. In Section \ref{vert-eul-char}, we will explore surfaces that have vertex angles that are \emph{not} $2\pi$, in which case the vertex point looks different from other points on the surface.\index{vertex}\index{square torus}

It is easy to go from the square torus back to the billiard table: Imagine that we have a path drawn on a square torus made of transparent paper (Figure \ref{fold-table}). Then if we fold it in half twice, like a paper napkin, we can see the corresponding trajectory on the billiard table. So anything we learn about trajectories on the square torus also tells us about trajectories on the square billiard table. After we discover more about trajectories in Section \ref{cuttingseqs}, we will apply our knowledge to billiards in Section \ref{revisit}. 
\index{square torus}\index{billiards}

\begin{figure}[!h]
\begin{center}
\includegraphics[width=270pt]{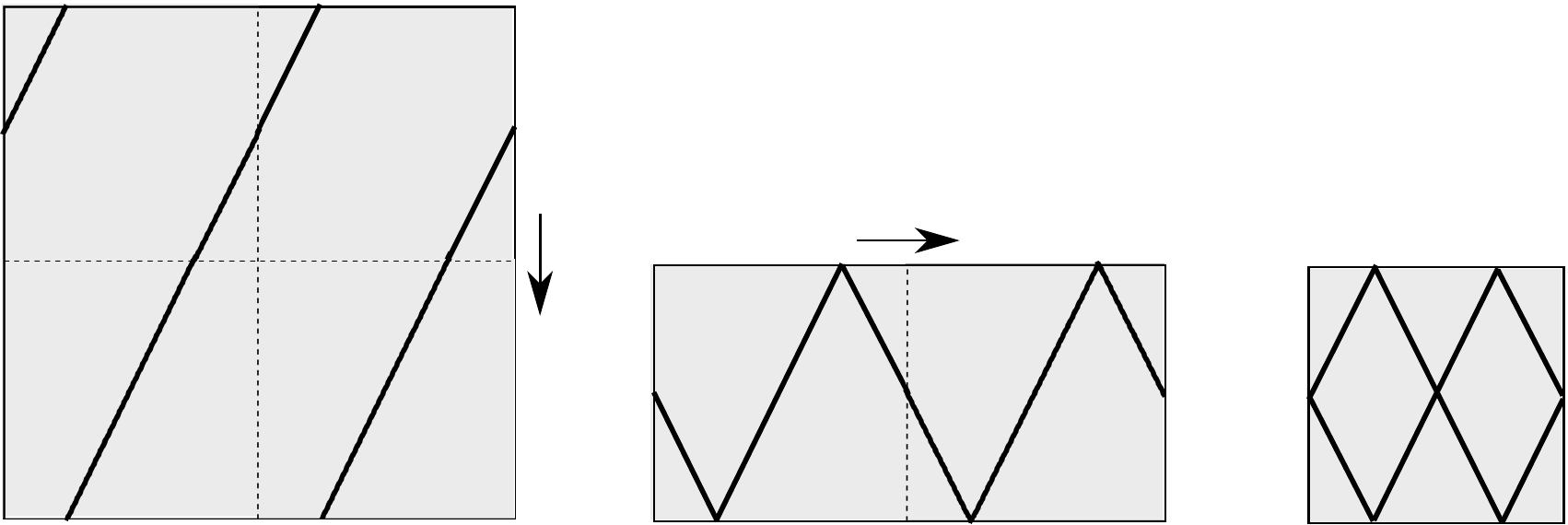} 
\caption{Folding the square torus back into the square billiard table \label{fold-table}\index{billiards}}
\end{center}
\end{figure}


We can unfold a straight line path on the square torus into a line on a piece of graph paper,  \index{unfolding}\index{square torus}\index{billiards} just as we unfolded a billiard trajectory: when the line hits the right edge of the torus, we unfold it to the right, and when it hits the top edge, we unfold it up. (You might wonder what we should do if the trajectory hits a vertex, but we can avoid this by slightly nudging the path so that it never hits a vertex.) In this way, we can see that lines on a piece of graph paper exactly correspond to straight line trajectories on the square torus. 

It is easy to go from a straight line on a piece of graph paper to a trajectory on the square torus: we can imagine that the graph paper is made of transparent squares, and stack them all on top of each other by translation (Figure \ref{line-torus}). The trajectory will now all be on one square, and that is the trajectory on the square torus. \index{square torus}

\begin{figure}[!h]
\begin{center}
\includegraphics[width=150pt]{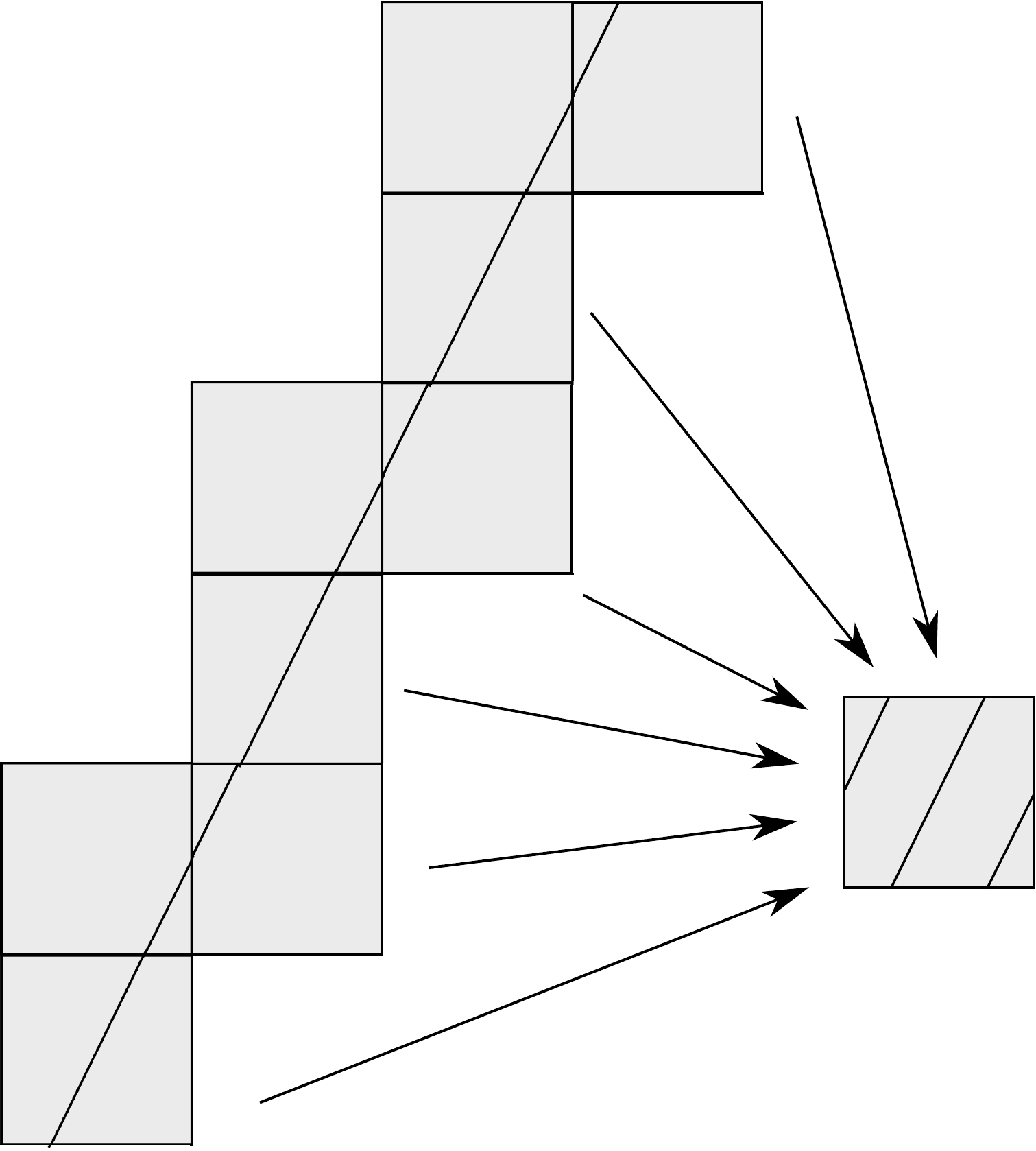} 
\caption{Stacking the linear trajectory into the square torus by translation \label{line-torus}}
\end{center}
\end{figure}

To imagine a trajectory on the flat, square torus as a trajectory on the curved, three-dimensional view of the torus, we can imagine a bug walking in a straight line on both surfaces. If the bug's path on the square torus is horizontal, the corresponding path on the 3D torus wraps around like an equator and comes back to where it started (path $A$ in Figure \ref{curved-torus}c). If the bug's path on the square torus is vertical, the corresponding path on the 3D torus passes through the hole and comes back to where it started (path $B$ in Figure \ref{curved-torus}c). If the bug's path is diagonal, it will wrap around and around the torus in a spiral, returning to its starting point if and only if the slope of the corresponding path on the square torus is rational  (Exercise \ref{periodic-torus}.)\index{square torus}\index{torus}

\begin{exercise} \label{periodic-torus}\index{square torus}
Show that a trajectory on the square torus is periodic if and only if its slope is rational. (You may use Exercise \ref{periodic-table}, the same result for the billiard table.)
\end{exercise}

\begin{exercise}\index{billiards}
Consider again the billiard table in the shape of an infinite sector from Exercise \ref{sector-table}, and show that any billiard on such a table makes finitely many bounces. (Hint: Unfold a trajectory on this table, by repeatedly reflecting the sector across the edges. The unfolding will look something like Figure \ref{unfold-octagon}.)
\end{exercise}

\newpage
\section{Cutting sequences} \label{cuttingseqs}

Given a straight line path on the square torus, we would like to be able to record where it goes. We do this with a \emph{cutting sequence}: \index{cutting sequence} When the path crosses the top/bottom edge, we record an $A$, and when it crosses the left/right edge, we record a $B$. This gives us an infinite sequence, which is the \emph{cutting sequence} corresponding to the trajectory. We assume that the trajectory extends ``backwards'' as well, so it is a bi-infinite sequence.\index{square torus}\index{cutting sequence}

\begin{example} \label{babba}
The trajectory in Figure \ref{csexample} has cutting sequence $\overline{ABABB}:$\footnote{We use the overline to indicate infinite repetition, just as $1/3=0.333\ldots = 0.\overline{3}.$ Here it is infinite in both directions: $\ldots ABABBABABB\ldots$} Let's start somewhere, say at the bottom-left intersection with edge $A$. Since it starts on edge $A$, we write down an $A$. Next, it crosses $B$ (at the midpoint), so now our cutting sequence is $AB\ldots$. Now it emerges from the left side $B$ and intersects $A$ at the top right, so our cutting sequence is $ABA\ldots$. It emerges from the right side of the bottom edge $A$ and shortly hits the right edge $B$, so our cutting sequence is $ABAB\ldots$. It emerges from the bottom of edge $B$ and hits the right edge $B$ at the top, so our cutting sequence is $ABABB\ldots$. It emerges from the left edge $B$ and hits edge $A$ at the left side, so we are back where we started. Thus the cutting sequence associated to this trajectory is the periodic sequence $\overline{ABABB}$. \index{square torus}\index{cutting sequence}

\end{example}

\begin{figure}[!h]
\begin{center}
\includegraphics[width=100pt]{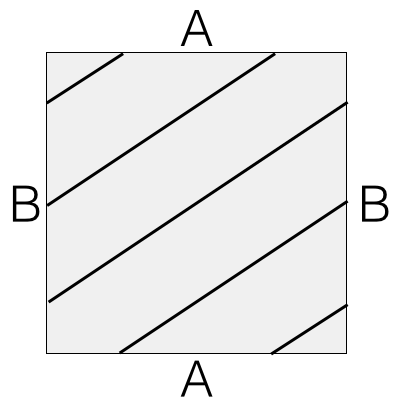} 
\caption{A trajectory with cutting sequence $\overline{ABABB}$ \label{csexample}}
\end{center}
\end{figure}

If we had started somewhere else $-$ say, the midpoint of $B$ $-$ we would write down the cutting sequence $\overline{BABBA}$, which is a cyclic permutation of $\overline{ABABB}$, and  since the sequence is bi-infinite, it is the same cutting sequence.

If we nudge the trajectory up or down a little bit, the corresponding cutting sequence does not change (assuming that we don't hit a vertex). \index{cutting sequence}  In fact, the cutting sequence corresponding to a trajectory on the square torus depends \emph{only} on its slope, as we will see (Algorithm \ref{cftocs} and Example \ref{constructcuttseq}).\footnote{For some surfaces, the cutting sequence depends on the slope and also on the location; see Example \ref{different-period}.}\index{square torus}\index{cutting sequence}

For example, the trajectory in Figure \ref{csexample} with cutting sequence $\overline{ABABB}$ has slope $2/3$. Let's see why this is the case: an $A$ in the cutting sequence means that we have gone up $1$ unit, and a $B$ in the cutting sequence means that we have gone to the right $1$ unit. So 
\begin{equation*}
\text{slope} = \frac{\text{rise}}{\text{run}} = \frac{\# \text{ of As}}{\# \text{ of Bs}} \text{ in the cutting sequence.}
\end{equation*}

Here is a strategy for drawing a trajectory with slope $p/q$: We'll make marks along edges $A$ and $B$, and then connect them up with segments. A trajectory with slope $p/q$ intersects edge $A$ $p$ times and edge $B$ $q$ times (Exercise \ref{pplusq}), and these intersections are equally spaced. 

If $p$ is odd, make a mark at the midpoint of edge $A$, and then space out the remaining $(p-1)/2$ marks on either side evenly, except that the space between the leftmost and rightmost marks and the corners is half as big (the space is actually the same size, but it is split between the left and the right). For example, for $p=5$, you will mark $1/10, 3/10, 1/2, 7/10$ and $9/10$ of the way along edge $A$.\index{square torus}

If $p$ is even, again make equally-spaced marks along edge $A$, with the leftmost and rightmost marks being half the distance from the corners. For example, for $p=2$, you will mark $1/4$ and $3/4$ of the way along edge $A$ (Figure \ref{csmarks}).

Now do the same for $q$ and edge $B$. Finally, connect the marks with line segments of slope $p/q$, which is easiest if you start in a corner and simply match them up. You will find that the line segments are parallel and equally spaced!

\begin{figure}[!h]
\begin{center}
\includegraphics[width=100pt]{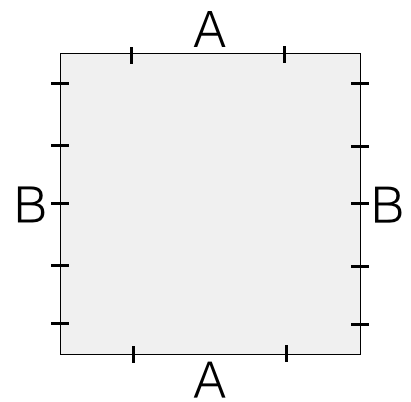} 
\caption{Marks for a trajectory with slope $2/5$ or $-2/5$ \label{csmarks}}
\end{center}
\end{figure}

You might wonder what the cutting sequence is for a trajectory that hits a vertex. We can avoid this problem by nudging the trajectory so that it does not hit the vertex, and crosses edge $A$ or $B$ nearby instead. If we choose for the trajectory to hit the vertex, the trajectory stops there, so the corresponding trajectory is a ray instead of a line, and the corresponding cutting sequence is only infinite in one direction.\index{square torus}

\begin{exercise}
Draw a trajectory on the square torus with slope $3/4$, and do the same for two other slopes of your choice. For each trajectory, write down the corresponding cutting sequence (use your picture).
\end{exercise}

\begin{exercise}
In Figure \ref{csmarks}, we put $2$ marks on edge $A$ and $5$ marks on edge $B$ and connected up the marks to create a trajectory with slope $2/5$. What if we did the same procedure for $4$ marks on edge $A$ and $10$ marks on edge $B$?
\end{exercise}

\begin{exercise} \label{pplusq}\index{cutting sequence}\index{square torus}
\begin{enumerate}
\item[(a)] Show that a trajectory on the square torus of slope $p/q$ (in lowest terms) crosses edge $A$ $p$ times and edge $B$ $q$ times, and thus has period $p+q$.
\item[(b)] Suppose that a given cutting sequence has period $n$. Are there any values of $n$ for which you can determine the cutting sequence (perhaps up to some symmetry) from this information?
\end{enumerate}
\end{exercise}

An active area of research is to describe all possible cutting sequences on a given surface. On the square torus, that question is: ``Which infinite sequences of $A$s and $B$s are cutting sequences corresponding to a trajectory?'' \index{square torus}

Let's answer an easier question: How can you tell that a given infinite sequence of $A$s and $B$s is \emph{not} a cutting sequence? 

Here's one way to tell that a given sequence is \emph{not} a cutting sequence.

\begin{proposition}\label{slope-contradiction}
If an infinite sequence of $A$s and $B$s has two $A$s in a row somewhere and also has two $B$s in a row somewhere, then it is \emph{not} a cutting sequence corresponding to a trajectory on the square torus.
\end{proposition}\index{cutting sequence}\index{square torus}

\begin{exercise}
Prove Proposition \ref{slope-contradiction}.
\end{exercise}

\begin{corollary}\index{square torus}
A given cutting sequence on the square torus has blocks of multiple $A$s separated by single $B$s, or blocks of multiple $B$s separated by single $A$s, but not both.
\end{corollary}

Sequences of $A$s and $B$s that are cutting sequences on the square torus are a classical object of study and are well understood. The non-periodic ones are called \emph{Sturmian sequences}. We will be able to completely describe all of the cutting sequences on the square torus (Theorem \ref{sturmian}) once we have developed a few more tools.  \index{Sturmian sequence}\index{cutting sequence}\index{square torus}

\section{Revisiting billiards on the square table}\label{revisit}

Our initial motivation for studying the square torus was to understand billiard paths on the square table, which we unfolded in Section \ref{st} to obtain the square torus. We will briefly return to billiards, to use our knowledge about trajectories on the square torus to prove things about billiards on the square billiard table. Remember that we can fold up the square torus into the square billiard table (Figures \ref{fold-table} and \ref{cs-billiard}), so the square torus that corresponds to the billiard table is twice as big (has four times the area).\index{billiards}

\begin{proposition} \label{torus-to-billiard}\index{square torus}\index{billiards}
Consider a periodic trajectory on the square torus, and a billiard path in the same direction on the square billiard table.
\begin{enumerate}[(a)] 
\item If the cutting sequence corresponding to the torus trajectory has period $n$, then the billiard path has period $2n$. 
\item If the cutting sequence corresponding to the trajectory is $\overline{w}$, then the sequence of edges that the billiard ball hits in one period is $ww$.
\end{enumerate}
\end{proposition}\index{cutting sequence}

\begin{proof}\index{square torus}
\begin{enumerate}[(a)]
\item We can transform a trajectory on the square torus into a billiard path on the square table by folding the torus vertically and horizontally (the left side of Figure \ref{cs-billiard}). Thus the path on the billiard table hits an edge when the corresponding trajectory on the square torus intersects edge $A$ or $B$, \emph{and} when the trajectory crosses one of the horizontal or vertical ``fold lines.'' So the billiard sequence (the sequence of edges that the billiard path hits) is a cutting sequence in the square grid, where there are lines at each vertical and horizontal integer, and also lines at the half-integers in both directions (the right side of Figure \ref{cs-billiard}). Every time the torus trajectory goes from a horizontal line to the next horizontal line, the billiard trajectory intersects an extra horizontal fold line between. The fold lines are halfway between the grid lines, so the billiard trajectory intersects twice as many horizontal lines as the torus trajectory. The argument for the vertical lines is the same, so the period of the billiard is twice the period of the torus trajectory.

\begin{figure}[!h]
\begin{center}
\includegraphics[width=300pt]{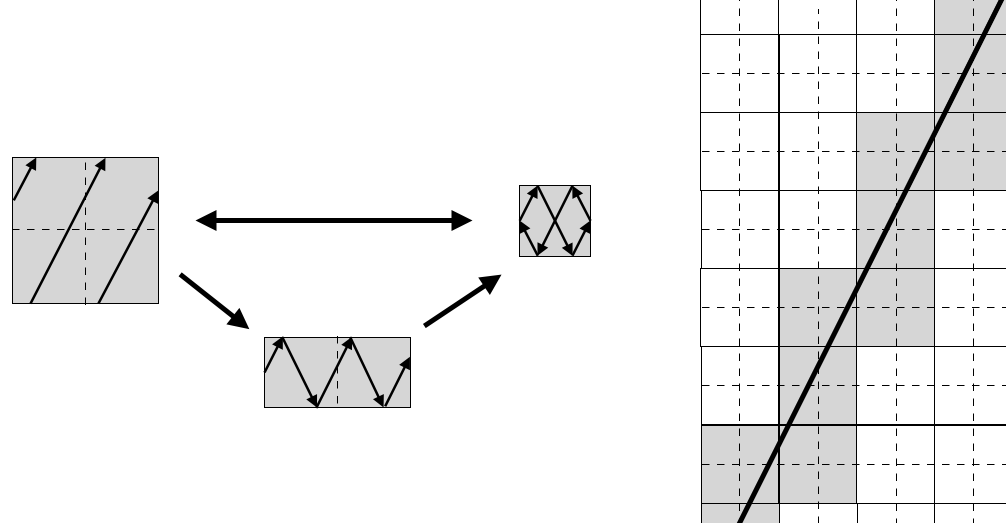} 
\caption{Translating between a trajectory on the square torus, and a billiard path on the square table \label{cs-billiard}}
\end{center}
\end{figure}

\item In fact, the torus trajectory on the integer grid is a dilated copy of the billiard path on the half-integer grid, so in the length of one period of the torus trajectory, the billiard path traverses the same sequence of edges, twice. So if the cutting sequence corresponding to the torus trajectory over one period is $w$, the cutting sequence corresponding to the billiard path of the same length is $ww$. The orientation of the billiard path is reversed after hitting the sequence of edges $w$, so it takes two cycles $ww$ to return to the starting point in the same direction and complete a period.\index{billiards}\index{cutting sequence}\index{square torus}
\end{enumerate}
\end{proof}

Now we can revisit the last three exercises of Section \ref{squaretorus}:

\vspace{.4em}\noindent{\em (Exercise \ref{periodeight}):}
How many billiard paths of period $8$ are there on the square table?

We can rephrase this question as: ``How many cutting sequences of period $4$ are there on the square torus?'' Up to symmetry, there is only one: $\overline{ABBB}$ or $\overline{BAAA}$, which are the same trajectory up to rotation. ($\overline{BABB}$, etc. is just a cyclic permutation of $\overline{ABBB}$, so it corresponds to the same trajectory.) We show the correspondence in Figure \ref{period8}.\index{cutting sequence}\index{square torus}

\begin{figure}[!h]
\begin{center}
\includegraphics[height=80pt]{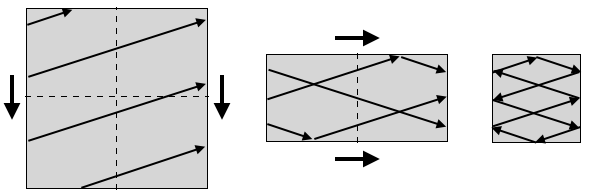} 
\caption{The trajectory $\overline{ABBB}$ can be folded into a billiard path of period $8$, which hits the sequence of edges $\overline{ABBBABBB}$. \label{period8}}
\end{center}
\end{figure}

\vspace{.4em}\noindent{\em (Exercise \ref{periodten}):}
Construct two different periodic billiard paths of period $10$ on a square table. \index{billiards}

Now we can rephrase this question as: ``Construct two different cutting sequences of period $5$ on the square torus.'' We can see that in fact there are only two possible cutting sequences: $\overline{ABBBB}$ and $\overline{ABABB}$, or the same with $A$ and $B$ reversed. ($\overline{AABBB}$ is not a valid cutting sequence, by Proposition \ref{slope-contradiction}.) Figure \ref{period10}a shows the trajectory $\overline{ABBBB}$ on the square torus, which can be folded into a period-$10$ billiard path. Figure \ref{period10}b shows the period-$10$ billiard path $\overline{ABABBABABB}$, which can be unfolded to a trajectory on the square torus (Exercise \ref{per10ex}).\index{cutting sequence}\index{square torus}

\begin{figure}[!h]
\begin{center}
\includegraphics[height=80pt]{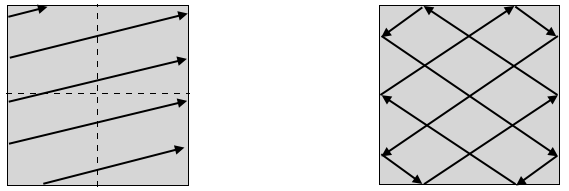} 
\caption{(a) The trajectory $\overline{ABBBB}$ and (b) the billiard path that hits the sequence of edges $\overline{ABABBABABB}$. \label{period10}}
\end{center}
\end{figure}

\begin{exercise} \label{per10ex}\index{square torus}\index{billiards}
Show how to fold the trajectory on the square torus in Figure \ref{period10}a into a billiard path of period $10$, and show how to unfold the billiard path in Figure \ref{period10}b into the trajectory $\overline{ABABB}$ on the square torus.
\end{exercise}

\vspace{.4em}\noindent{\em (Exercise \ref{oddperiod}):}
Construct a periodic billiard path on a square table with an odd period, or show that it is not possible to do so. 
\index{billiards}

By Proposition \ref{torus-to-billiard}a, every billiard path has even period. \\

In fact, the easiest way to go from a periodic trajectory on the square torus, to the billiard path with the same slope, is to simply draw in the negatively-sloped paths right on the square torus picture: Each time the trajectory hits an edge, make it bounce off and connect it to the appropriate intersection on a different edge (and then at the end, scale the picture to make it half as big). Similarly, the easiest way to unfold the billiard trajectory into the trajectory on the square torus is to erase the paths of negative slope (and and the end, make it twice as big).

\newpage
\section{Symmetries of the square torus} \index{symmetry}\label{sq-tor-symm}\index{square torus}

Let's step back a bit and consider symmetries of the square torus. 
We will say that a \emph{symmetry} of the square torus is a transformation that takes nearby points to nearby points, and doesn't overlap or leave any gaps.\footnote{The precise term for this ``symmetry'' is  \emph{automorphism}, a bijective action that takes the surface to itself.}

We will only allow symmetries that take the (single) vertex of the torus to itself, and later on symmetries that take all vertices of a given surface to other vertices of the surface.  \index{vertex}

{\bf Reflection:} Reflecting across vertical, horizontal or diagonal lines of symmetry is a symmetry of the square torus that fixes the vertex. On the curved torus, vertical and horizontal reflections are horizontal and vertical reflections, respectively; the diagonal reflections are more difficult to visualize. 

{\bf Rotation:} Rotating the square torus by a multiple of $\pi/2$ is a symmetry that fixes the vertex. This is  not easy to visualize on the curved torus. \index{symmetry}\index{square torus}

{\bf Twist/shear:} \index{shear} Imagine that we cut the curved torus along a meridian (edge $B$, say), give it a full twist, and then glue it back together (the top part of Figure \ref{shear-twist}). This is a symmetry of the surface. On the square torus, this is a \emph{shear}: We apply the matrix $\upshear$, and then reassemble the resulting parallelogram back into a square (bottom part of Figure \ref{shear-twist}). Similarly, cutting along edge $A$ and doing a full twist corresponds to shearing the square torus by the matrix $\shear$.\index{square torus}

\begin{figure}[!h]
\begin{center}
\includegraphics[width=100pt]{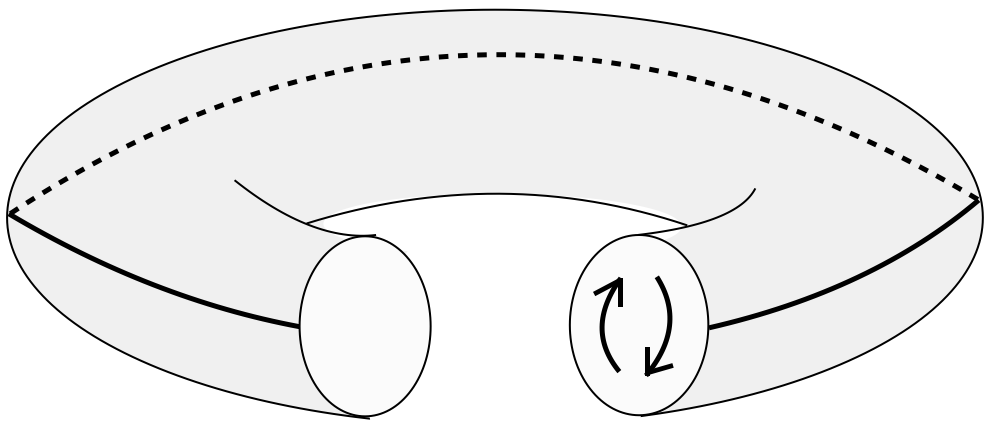} \ \ 
\includegraphics[width=100pt]{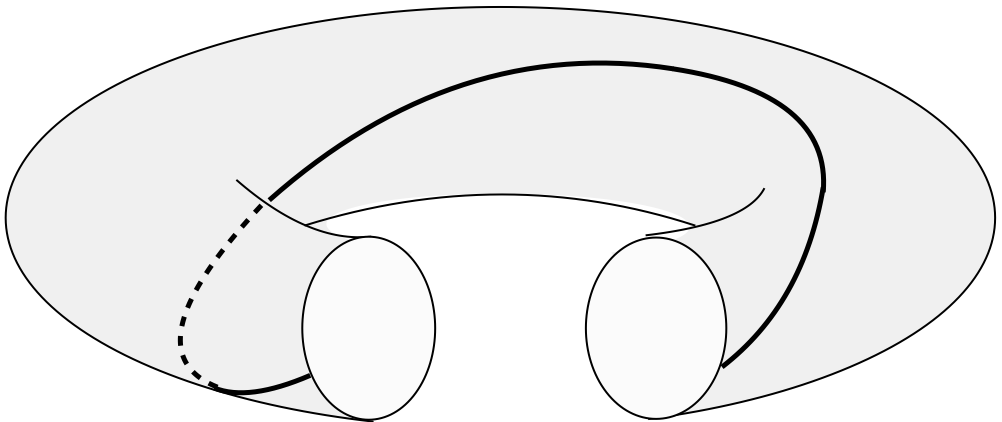} \ \ 
\includegraphics[width=100pt]{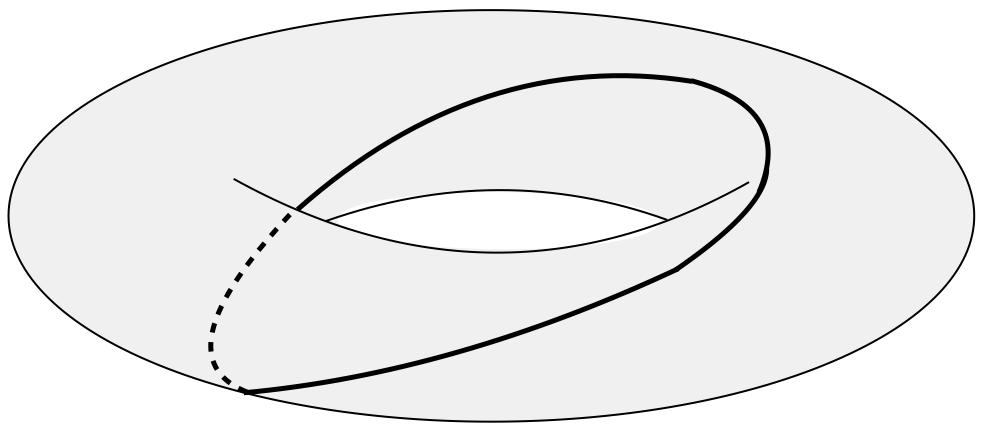} \\
\includegraphics[width=50pt]{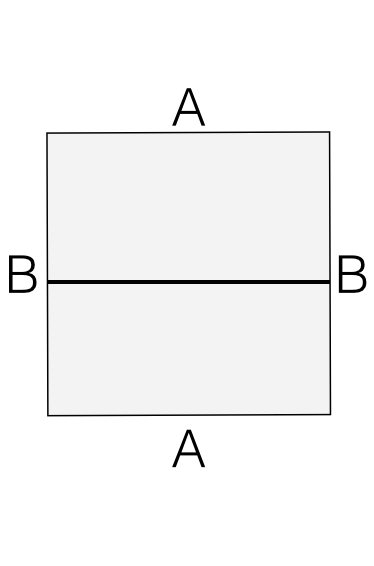} \ \ \ \ \ \ \ \ \ \ \ \ \ \ \ \ \ \ \ \ 
\includegraphics[width=50pt]{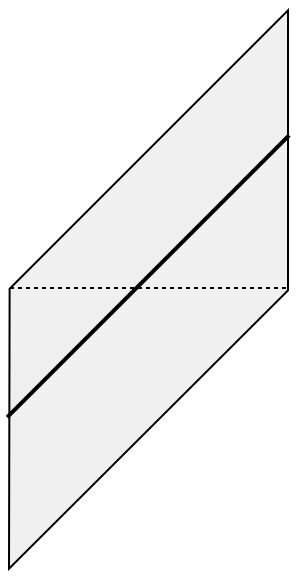} \ \ \ \ \ \ \ \ \ \ \ \ \ \ \ \ \ 
\includegraphics[width=50pt]{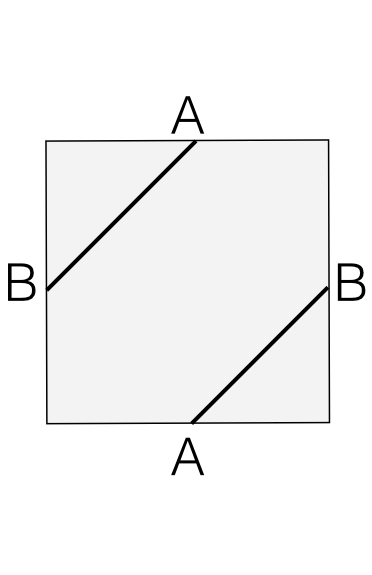} 
\caption{Twisting the 3D torus, shearing the square torus  \label{shear-twist}}
\end{center}
\end{figure}


Given a trajectory on the square torus, we want to know what happens to that trajectory under a symmetry of the surface. We will answer this question by comparing their cutting sequences: the cutting sequence $c(\tau)$ corresponding to the original trajectory $\tau$, and the cutting sequence $c(\tau')$ corresponding to the transformed trajectory $\tau'$.  The effects of rotations and reflections are easy to describe: \index{cutting sequence}\index{square torus}

{\bf Reflection:} Horizontal or vertical reflection sends edge $A$ to $A$ and edge $B$ to $B$. $\tau'$ is different from $\tau$ (Figure \ref{auto-torus}b)  but $c(\tau)$ and $c(\tau')$ are the same. Reflection across the diagonal sends $A$ to $B$ and $B$ to $A$, so $c(\tau')$ is $c(\tau)$ with the $A$s and $B$s reversed (Figure \ref{auto-torus}c).

{\bf Rotation:} Rotation by $\pi/2$ sends $A$ to $B$ and $B$ to $A$, so $c(\tau')$ is $c(\tau)$ with the $A$s and $B$s reversed (Figure \ref{auto-torus}d). Thus, the effect on the cutting sequence is the same as reflection across a diagonal. 
Rotation by $\pi$ preserves the trajectory, since a line is symmetric under rotation by $\pi$. \index{square torus}

\begin{figure}[!h]
\begin{center}
\includegraphics[width=340pt]{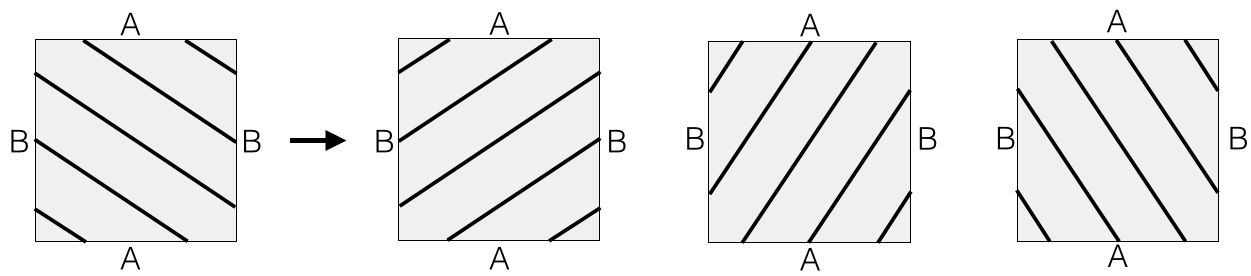} 
\caption{(a) A trajectory of slope $-2/3$ on the square torus, and the result of a (b) horizontal or vertical reflection  (c) reflection across the positive or negative diagonal and (d) $\pi/2$ clockwise or counter-clockwise rotation. \index{symmetry} \label{auto-torus}}
\end{center}
\end{figure}

In Theorem \ref{sturmian} and other future results, we will use the action of reversing $A$s and $B$s, and we will assume that it corresponds to the diagonal flip, as this takes positive slopes to positive slopes.

{\bf Shear:} This is the most interesting case, the one we will study in detail.\footnote{Our exposition here follows the introduction in Smillie and Ulcigrai's paper \cite{SU}.\index{Ulcigrai, Corinna}\index{Smillie, John}} We will examine the effect of the shear $\nupshear$, because we can reduce every other shearing symmetry to repeated applications of this one (Proposition \ref{reducingprop}). 
\index{Smillie, John} \index{Ulcigrai, Corinna} \index{shear} \index{cutting sequence}\index{shear} \index{square torus}

First, we will do an example of shearing via $\nupshear$, and then we will find a general rule for the action.

\begin{example}\index{cutting sequence}
We begin with the cutting sequence $\overline{BAABA}$ (Figure \ref{shear-torus}). We shear it via $\nupshear$, which transforms the square into a parallelogram, and then we reassemble the two triangles back into a square torus, while respecting the edge identifications. The new cutting sequence is $\overline{BAB}$. 
\end{example}\index{square torus}

\begin{figure}[!h]
\captionsetup{singlelinecheck=off}
\begin{center}
\includegraphics[height=160pt]{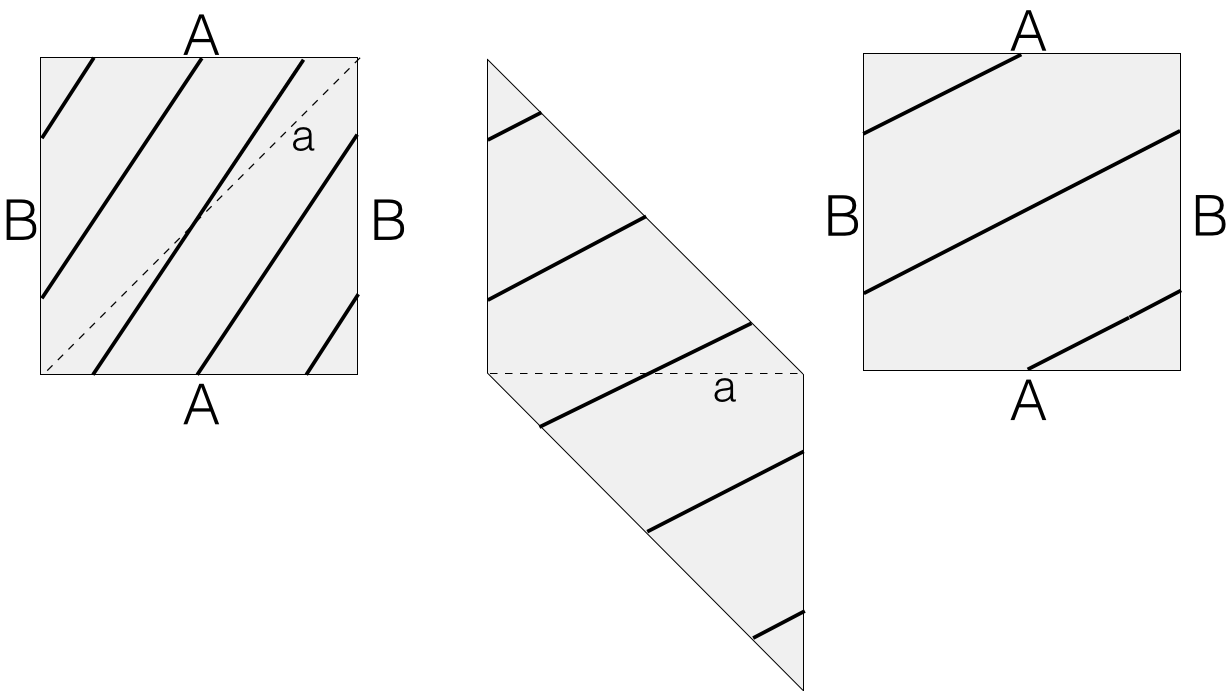} 
\caption{(a) The trajectory with cutting sequence $\overline{BAABA}$ (b) after shearing via $\protect\nupshear\protect$  (c) reassembled back into a square with cutting sequence $\overline{BAB}$. \label{shear-torus}}
\end{center}
\end{figure}

We ask: Is there a rule we can apply to the sequence $c(\tau)=\overline{BAABA}$ to obtain $c(\tau')=\overline{BAB}$, without drawing the shear geometrically? Of course, this one example would fit many different rules. Let's see \emph{why} $\overline{BAABA}$ became $\overline{BAB}$:

The shear $\nupshear$ is vertical, so it does not affect edge $B$: both of the crossings of edge $B$ survived the shear. Where did the $A$ in $\overline{BAB}$ come from? Let's trace it back. We draw it in as a dashed line in each of the pictures, and ``un-shear'' it back (going right to left in Figure \ref{shear-torus}) into the original square, where we can see that it is the positive diagonal. So we see that we get an $A$ in the transformed cutting sequence only when the trajectory crosses this positive diagonal. \index{cutting sequence}


If we restrict our attention to trajectories with slope greater than $1$ (as in this example), we can see that a trajectory crosses the dashed edge when it goes from edge $A$ to edge $A$. So we can ``augment'' our original cutting sequence with crossings of the dashed edge $A$, which we'll call $a$: 
\begin{equation*}
\overline{BAABA} \xrightarrow{\text{augment with $a$}} \overline{BAaABA} \xrightarrow{\text{remove $A$}} \overline{BaB} \xrightarrow{\text{change $a$ to $A$}} \overline{BAB}.
\end{equation*}
If our original sequence is more complicated, we can do a similar process:
\begin{equation*}
\overline{BAAAABAAA} \to \overline{BAaAaAaABAaAaA} \to \overline{BaaaBaa} \to \overline{BAAABAA}.
\end{equation*}
We can now state this result:

\begin{theorem} \label{kosl} \index{cutting sequence}\index{square torus}
Given a trajectory $\tau$ on the square torus with slope greater than $1$, and its corresponding cutting sequence $c(\tau)$, let $\tau'$ be result of applying $\nupshear$ to $\tau$. To obtain $c(\tau')$ from $c(\tau)$, shorten each string of $A$s by $1$.
\end{theorem}


\begin{exercise}\label{pfofcs}
Write down a proof of Theorem \ref{kosl}.
\end{exercise}

\begin{exercise}\index{shear}\index{square torus}
Determine the effect of the shear $\stt 10{-1}1$ on the trajectory of slope $5/2$ with corresponding cutting sequence $\overline{BAABAAA}$, both by shearing the square geometrically, and by applying Theorem \ref{kosl} to the cutting sequence. Check that the results agree.
\end{exercise}

In Proposition \ref{slope-contradiction}, we showed that no valid cutting sequence has $AA$ somewhere and $BB$ somewhere else. In Theorem \ref{kosl}, we showed that, given a valid cutting sequence with multiple $A$s separated by single $B$s, if we delete one $A$ from each block, the resulting sequence is still a valid cutting sequence. If we switch $A$s and $B$s, this is a diagonal flip, so the resulting sequence is still a cutting sequence. Now we are ready to characterize all possible cutting sequences on the square torus: 

\begin{theorem}\label{sturmian}
Given an infinite sequence of $A$s and $B$s, iterate the following process:
\begin{enumerate}
\item If it has $AA$ somewhere and also $BB$ somewhere, reject it; it is not a valid cutting sequence.
\item If it is $\ldots BBBABBB\ldots$ or $\ldots AAABAAA\ldots$, reject it; it is not a valid cutting sequence.
\item If it has multiple $A$s separated by single $B$s, delete an $A$ from each block.
\item If it has multiple $B$s separated by single $A$s, reverse $A$s and $B$s.
\end{enumerate}
If we can perform this process on the sequence forever, and the result is never rejected, then the original sequence is a valid cutting sequence, as are all the sequences following. If the result is eventually rejected, then neither the original sequence, nor any of the sequences following, are valid cutting sequences on the square torus.
\end{theorem}\index{square torus}

In the space of bi-infinite sequences of $A$s and $B$s, part of the space is valid cutting sequences, and part of it is sequences of $A$s and $B$s that are \emph{not} cutting sequences. The sequences $\ldots BBBABBB\ldots$ and $\ldots AAABAAA\ldots$, and their predecessors in the chain described above, are on the ``boundary'' between these two parts of the space.\index{cutting sequence}

For a proof of Theorem \ref{sturmian} and further reading, see \cite{smith}, \cite{morse}, \cite{series} and \cite{arnoux}. For an example of this process, see Example \ref{reducing-alltheway}. \\

Another reason to study cutting sequences on the square torus  is that they have very low complexity:  \index{Sturmian sequence}\index{complexity}\index{square torus}

\begin{definition}
The  \emph{complexity function} $p(n)$ on a sequence is the number of different ``words'' of length $n$ in the sequence.
\end{definition}

One way to think about complexity is that there is a ``window'' $n$ letters wide that you slide along the sequence, and you count how many different things appear in the window. 

The highest possible complexity for a sequence of $A$s and $B$s is $p(n)=2^n$, because for each of the $n$ positions in the window, you have 2 choices ($A$ or $B$). Aperiodic cutting sequences on the square torus (Sturmian sequences) have complexity $p(n)=n+1$. Periodic cutting sequences on the square torus have complexity $p(n)=n+1$ for $n<p$ and complexity $p(n)=n$ for $n\geq p$. \index{Sturmian sequence} 

\begin{exercise}
Confirm that the cutting sequence $\overline{ABABB}$ has complexity $p(n)=n+1$ for $n=1,2,3,4$ and complexity $p(n)=n$ for $n\geq 5$. \index{complexity}\index{cutting sequence}
\end{exercise}

The square torus has two edges ($A$ and $B$). Later, we will make translation surfaces with many edges. For an aperiodic cutting sequence on a translation surface with $k$ edges, the complexity  is $(k-1)n+1$. For further reading, see \cite{ferenczi}.\index{square torus}

\section{Continued fractions} \index{continued fraction}\label{contfrac}

We will see that there is a beautiful connection between twists and shears, and the \emph{continued fraction expansion} of the slope of a given trajectory. We will briefly pause our discussion of lines on the torus to introduce continued fractions.

The \emph{continued fraction expansion} 
gives an expanded expression of a given number. To obtain the continued fraction expansion for a number, we do the following method, which is best explained via an example:

\begin{example} \label{cont-frac-ex}
We compute the continued fraction expansion of $15/11$. 

\begin{equation*}\footnotesize
  \frac {15}{11} =1+\frac 4{11}=1+ \frac1{11/4}=
 1+  \cfrac{1}{1
          + 7/4} 
          =
 1+  \cfrac{1}{2
          + 3/4} 
          =
 1+ \cfrac{1}{2
          + \cfrac{1}{4/3 } }           
          =
{\bf 1}+ \cfrac{1}{{\bf 2}
          + \cfrac{1}{{\bf 1}
          + \cfrac{1}{{\bf 3}} } } .                   
\end{equation*}
Since all the numerators are $1$, we can denote the continued fraction expansion more compactly by recording only the bolded numbers. We write $15/11=[1,2,1,3]$. 
\end{example}

We can explain this procedure as an algorithm:

\begin{algorithm}
We begin with the entire number as our ``remainder,'' and iterate the following procedure.
\begin{enumerate}
\item If the remainder is more than $1$, subtract $1$.
\item If the remainder is between $0$ and $1$, take the reciprocal.
\item If the remainder is $0$, stop.
\end{enumerate}
We keep track of the new expression of the number via a \emph{continued fraction}, which we build as we go, as in Example \ref{cont-frac-ex}.
\end{algorithm}

The continued fraction expansion of a rational number terminates, and the continued fraction expansion of an irrational number is infinite. \index{continued fraction}

Geometrically, the continued fraction algorithm for a number $x$ is (Figure \ref{contfrac}):
\begin{enumerate}
\item Begin with a $1\times x$ rectangle (or $p\times q$ if $x=p/q$).
\item Cut off the largest possible square, as many times as possible. Count how many squares you cut off; this is $a_1$.
\item With the remaining rectangle, cut off the largest possible squares; the number of these is $a_2$.
\item Continue until there is no remaining rectangle. The continued fraction expansion of $x$ is then $[a_1,a_2,\ldots]$.
\end{enumerate}

\begin{figure}[!h]
\begin{center}
\includegraphics[width=150pt]{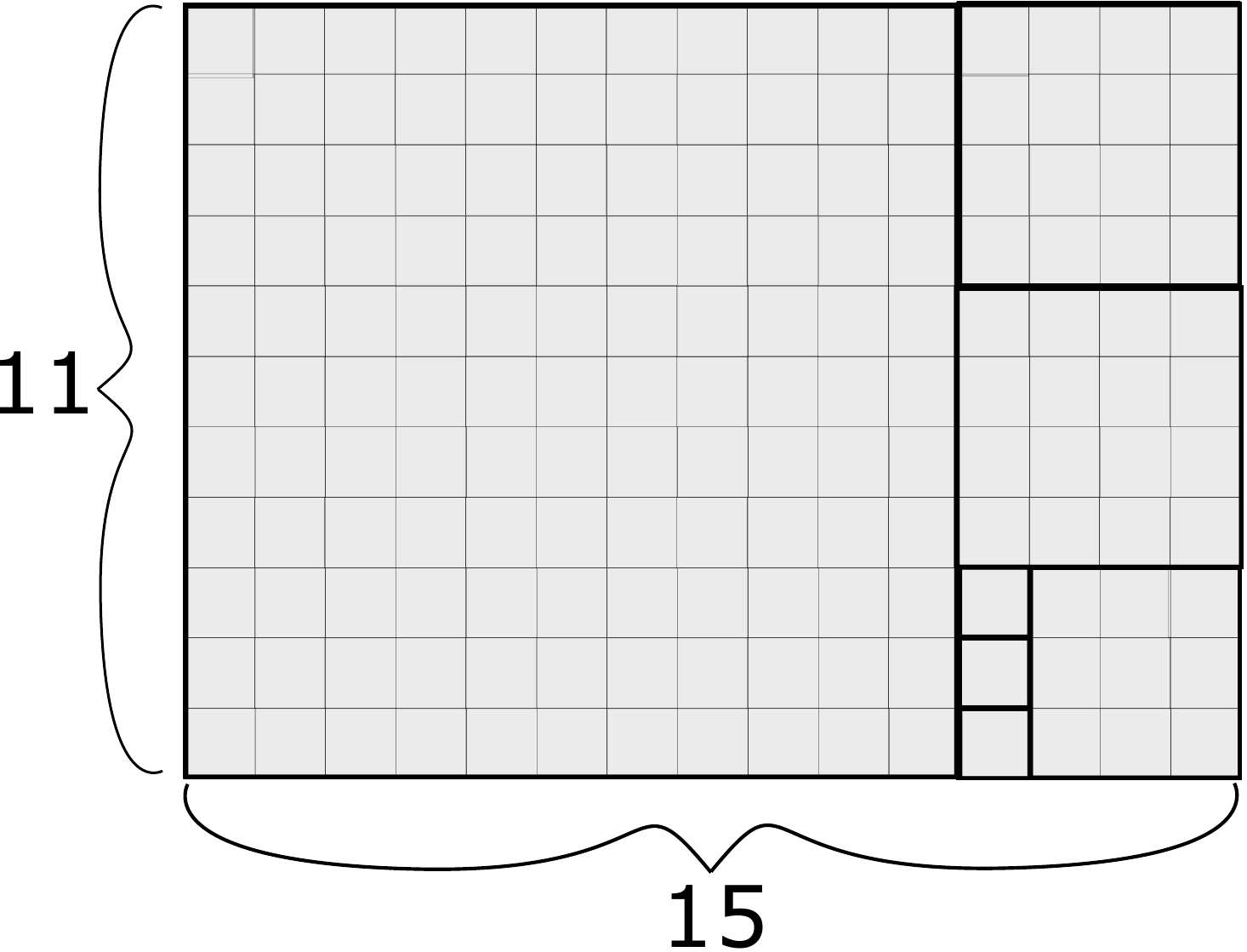} \ \ \ \ \ \ \ \ \ 
\includegraphics[width=170pt]{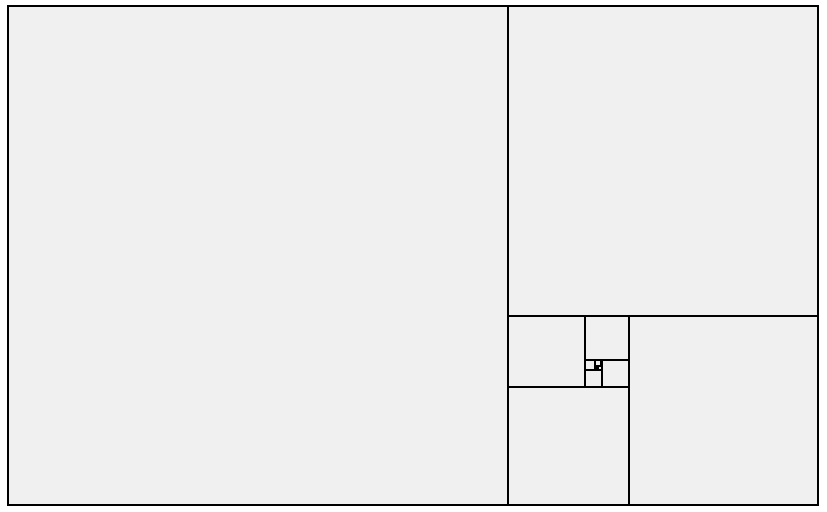}
\caption{A geometric explanation for the continued fraction expansion. (a) For a $15\times 11$ rectangle, the number of squares of each descending size is $1,2,1,3$. (b) For a $\phi \times 1$ rectangle, the number of squares of each descending size is $1,1,1,\ldots$ \label{contfrac}\index{golden ratio}}
\end{center}
\end{figure}

\newpage
  The golden ratio $\phi= (1+\sqrt{5})/2$ \index{golden ratio}  has a particularly elegant continued fraction expansion (Figure \ref{contfrac}b). It satisfies the equation $\phi=1+1/\phi$, so we get:

\begin{equation*}
\phi  = 
1 + \cfrac{1}{1
          + \cfrac{1}{1
          + \cfrac{1}{1 + \cfrac{1}{\ddots} } } } =[1,1,1,1,\ldots].
\end{equation*}

We will see a beautiful connection between the square torus and continued fraction expansions in the next section.\index{continued fraction}

\begin{exercise} \label{findcfe}
Find the continued fraction expansions of $5/7$.  (You will use it in Exercise \ref{usefivesevenths}.)
\end{exercise}

\begin{exercise} \label{grcfe}
Find the continued fraction expansions of $3/2, 5/3, 8/5$, and $13/8$. Describe any patterns you notice, and explain why they occur.
\end{exercise}

\begin{exercise}
Find the first few steps of the continued fraction expansion of $\pi$. Explain why the common approximation $22/7$ is a good choice. Then find the best fraction to use, if you want a fractional approximation for $\pi$ using integers of three digits or less.
\end{exercise}

\begin{exercise}
Find the continued fraction expansion of $\sqrt{2}-1$. Then solve the equation $x=\frac1{2+x}$ and explain how these are related. 
\end{exercise}

\newpage
\section{Continued fractions and cutting sequences}\index{continued fraction}\index{cutting sequence}\label{correspondence}

In this section, we show how to find the cutting sequence corresponding to a trajectory of a given slope, and conversely how to find the continued fraction expansion of the slope of the trajectory from its cutting sequence. To see how it works, let's look at an example. \index{square torus}

\begin{example}\label{reducing-alltheway}
We reduce the trajectory of slope $3/2$ all the way to a horizontal line in Figure \ref{reduce-end}. \index{shear}

\begin{figure}[!h]
\begin{center}
\includegraphics[width=77pt]{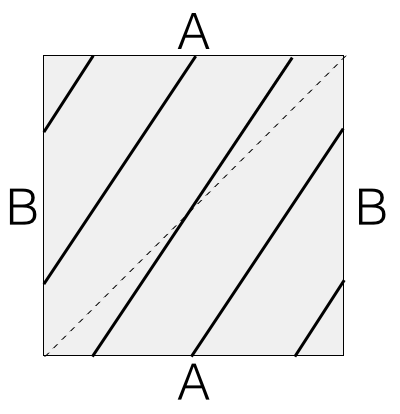} $\stackrel{\text{shear}}{\rightarrow}$
\includegraphics[width=67pt]{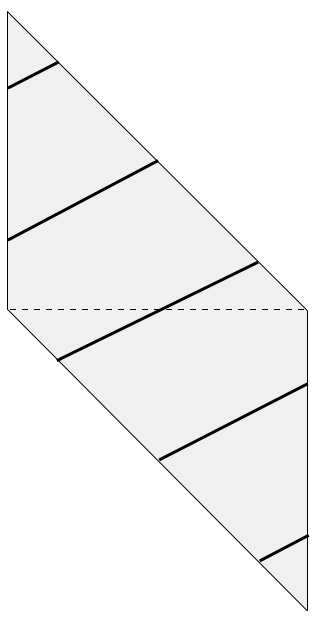} $\stackrel{\text{reassemble}}{\rightarrow}$ 
\includegraphics[width=77pt]{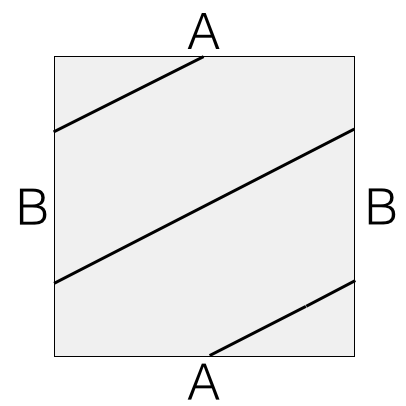} $\stackrel{\text{flip}}{\rightarrow}$ 
\includegraphics[width=77pt]{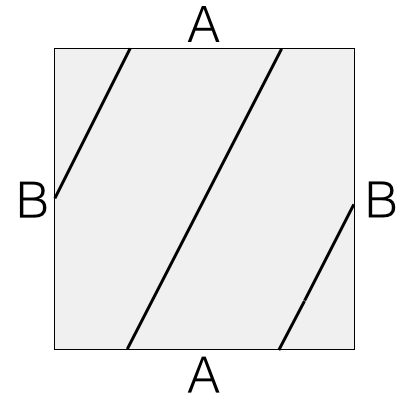} $\stackrel{\text{shear}}{\rightarrow}$
\includegraphics[width=67pt]{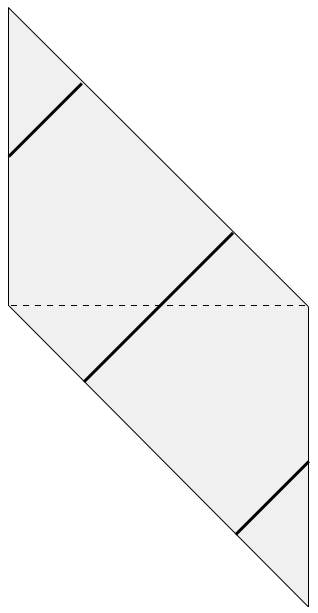} $\stackrel{\text{reassemble}}{\rightarrow}$ 
\includegraphics[width=77pt]{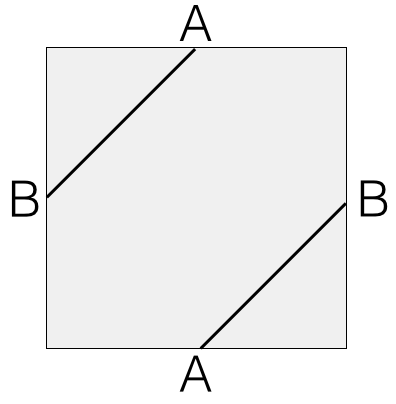} $\stackrel{\text{shear}}{\rightarrow}$ 
\includegraphics[width=67pt]{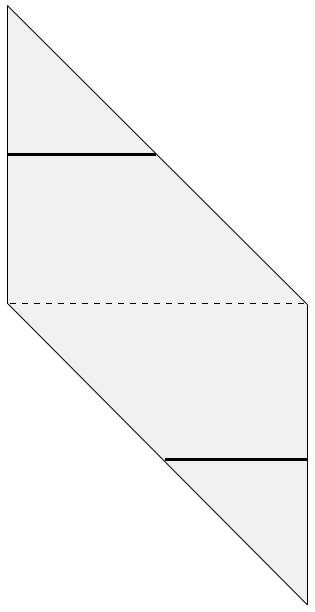} $\stackrel{\text{reassemble}}{\rightarrow}$ 
\includegraphics[width=77pt]{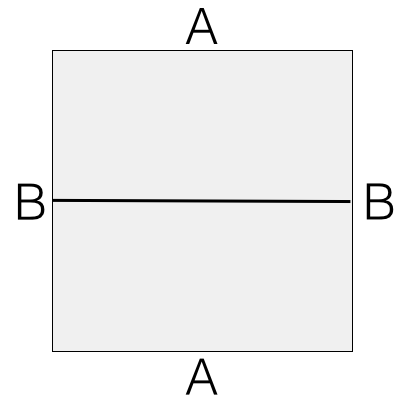} \ \ 
\caption{We apply Algorithm \ref{alg-torus} to a trajectory of slope $3/2$. \label{reduce-end}\index{shear}}
\end{center}
\end{figure}

\end{example}
%
%

We can describe this process with an algorithm:  \index{square torus}

\begin{algorithm} \label{alg-torus}
We iterate this process until it tells us to stop:
\begin{enumerate}
\item If the slope is greater than or equal to $1$, apply $\nupshear$.
\item If the slope is between $0$ and $1$, apply $\stt 0110$.
\item If the slope is $0$, stop.
\end{enumerate}
\end{algorithm}

We can note down the steps we took: shear, flip, shear, shear. We ended with a slope of $0$. The shear subtracts $1$ from the slope, and the flip inverts the slope, so working backwards, the initial slope was:\index{shear}

$$0 \to 1+0 \to 1+1+0 \to \frac{1}{1+1+0} \to 1+ \frac{1}{1+1+0} = 3/2.$$

Furthermore, $1+\frac{1}{2}$ is the continued fraction expansion for $3/2$ (Exercise \ref{grcfe}).

In this section, we'll explore how this process allows us to translate between a cutting sequence, and the continued fraction expansion of the slope of the associated trajectory. As we did in Section \ref{sq-tor-symm}, we will use the cutting sequence corresponding to a trajectory, rather than drawing out the trajectory geometrically like Example \ref{reducing-alltheway} every time.

\vspace{1em}\noindent {\bf Cutting sequence to continued fraction expansion:}\index{continued fraction}\index{cutting sequence}\index{square torus}

\begin{algorithm} \label{reducecs} Given a cutting sequence, perform the following procedure:
\begin{enumerate}
\item If the sequence has multiple $A$s separated by single $B$s, decrease the length of each string of $B$s by $1$.
\item If the sequence has multiple $B$s separated by single $A$s, reverse $A$s and $B$s (change all the $B$s to $A$s and all the $A$s to $B$s). 
\item If the sequence is an infinite string of $B$s, stop.
\end{enumerate}

Since Step $1$ strictly decreases the period, and Step $2$ does not change the period, a cutting sequence with finite period (corresponding to a trajectory with a rational slope; see Exercise \ref{periodic-torus}) will always end with a period of $1$, which is a string of $B$s. If the trajectory has irrational slope, the cutting sequence algorithm continues forever, just like the continued fraction algorithm of its slope.
\end{algorithm}

\begin{example} \label{crossoutstuff}\index{cutting sequence}
We start with the periodic cutting sequence $\overline{BAABABAABABA}$ and perform Algorithm \ref{reducecs} on it.

\begin{align*}
\text{(original sequence)}& {\color{white}\to} \overline{BAABABAABABA} \\
\text{remove an $A$ between every two $B$s} &\to \overline{BABBABB} \\
\text{reverse $A$s and $B$s} &\to \overline{ABAABAA} \\
\text{remove an $A$ between every two $B$s} &\to \overline{ABABA}  \\
\text{remove an $A$ between every two $B$s} &\to \overline{ABB} \\
\text{reverse $A$s and $B$s} &\to \overline{BAA} \\
\text{remove an $A$ between every two $B$s} &\to \overline{BA} \\
\text{remove an $A$ between every two $B$s} &\to \overline{B} \\
\text{The sequence is an infinite string of $B$s, so we stop.}
\end{align*}
We can note down the steps we took: shear, flip, shear, shear, flip, shear, shear.
\end{example}

Notice the following: \index{square torus}
\begin{enumerate}
\item The shear subtracts $1$ from the slope (Exercise \ref{minusone}a).
\item The flip inverts the slope (Exercise \ref{minusone}b).
\item The trajectory corresponding to an infinite string of $B$s has a slope of $0$.
\end{enumerate}

So to obtain the continued fraction expansion of the slope, follow the procedure and note down what the steps are, and then, starting from $0$, we can recreate the continued fraction expansion of the slope by working backwards: \index{continued fraction}

\begin{example}\label{crossoutstuffcont}
For our periodic sequence in Example \ref{crossoutstuff} above, we have 

\begin{center}
shear $\to$ flip $\to$ shear  $\to$ shear  $\to$ flip $\to$ shear $\to$ shear $\to 0$.
\end{center}

So, reading backwards and translating, our instructions are:

\begin{center}
0 $\to$ add 1 $\to$ add 1 $\to$ invert $\to$ add 1 $\to$ add 1 $\to$ invert $\to$ add 1.
\end{center}

So our continued fraction is \index{continued fraction}

\begin{equation*} \label{flipshearcf}
0 \to 1 \to 2 \to \frac 12 \to 1+\frac 12 \to 2+\frac 12 \to \cfrac{1}{2+\cfrac 12} \to 1+\cfrac{1}{2+\cfrac 12}.                   
\end{equation*}
Each of these $8$ fractions is the slope of the trajectory corresponding to the cutting sequences in Example \ref{crossoutstuff}, starting with $\overline{B}$ corresponding to slope $0$ and going up. \index{cutting sequence}
\end{example}

\newpage
\noindent {\bf Continued fraction expansion to cutting sequence:} \index{continued fraction}\index{cutting sequence}

We do the process in Algorithm \ref{reducecs} in reverse: 


\begin{algorithm} \label{cftocs}\index{continued fraction}\index{square torus}
Let the continued fraction expansion of the slope be $[a_1, a_2, \ldots, a_k]$, and do the following:
\begin{itemize}
\item Start with a string of $B$s. (The corresponding trajectory has slope $0$.)
\item Insert $a_k$ $A$s between each pair of $B$s. (The corresponding trajectory now has slope $a_k$.)
\item Reverse $A$s and $B$s. (The corresponding trajectory now has slope $\frac{1}{a_k}$.)
\item Insert $a_{k-1}$ $A$s between each pair of $B$s. (The corresponding trajectory now has slope $a_{k-1}+\cfrac{1}{a_k}$.)
\item Reverse $A$s and $B$s. (The corresponding trajectory now has slope $\cfrac{1}{a_{k-1}+\cfrac{1}{a_k}}$.)
\item Continue this process for $k$ steps, ending by inserting $a_1$ $A$s between each pair of $B$s. This yields the cutting sequence corresponding to the fractional slope $a_1+\cfrac 1{\ddots+\cfrac 1 {a_k}}$.
\end{itemize}
\end{algorithm}

\begin{example} \label{constructcuttseq} \index{continued fraction}\index{cutting sequence}\index{square torus}
We construct the cutting sequence corresponding to a trajectory with slope $7/4 = [1,1,3]$. Below, we show only one period of the periodic sequence. (Remember that it repeats, so when we insert an $A$ between every two $B$s, it may not look like the $A$ is between two $B$s if it is at the beginning or end of the period.)
\begin{align*}
  & &\text{slope} \\
 \text{(original sequence)} & {\color{white}\to \  } \overline{B} & 0 \\
\text{insert $3$ $A$s between each pair of $B$s}  &\to \overline{BAAA} &3 \\
\text{reverse $A$s and $B$s} &\to \overline{ABBB} &1/3 \\
\text{insert an $A$ between each pair of $B$s} &\to \overline{AABABAB} &4/3\\
\text{reverse $A$s and $B$s} &\to \overline{BBABABA} &3/4 \\
\text{nsert an $A$ between each pair of $B$s} & \to \overline{BABAABAABAA} &7/4
\end{align*}
To check these answers, we can count that, for example, $\overline{BABAABAABAA}$ does indeed have $7$ $A$s and $4$ $B$s.
\end{example}

Given a cutting sequence, we can determine the continued fraction expansion of the slope by doing this process in reverse:

\begin{proposition} Consider a trajectory $\tau$ with corresponding cutting sequence $c(\tau)$. When performing Algorithm \ref{reducecs} on $c(\tau)$, let the number of times we decrease the length of each string of $A$s of the original sequence (step 1) be $a_1$, and then after the flip (step 2) let the number of times we decrease the length of each string of $A$s in the resulting sequence (step 1) be $a_2$, and so on until the last step we cross out an $A$ between each pair of $B$s $a_k$ times. \index{square torus}

Then the slope of $\tau$ has continued fraction expansion $[a_k, a_{k-1}, \ldots, a_1]$.
\end{proposition}

\begin{proof}
This follows from the construction in Algorithm \ref{cftocs}, with the indexing on the $a_i$s reversed.
\end{proof}

\begin{exercise}\label{minusone}\index{shear}\index{square torus}
\begin{enumerate}[(a)]
\item Show that applying the shear $\stt 10{-1}1$ to the square torus decreases the slope of a trajectory by $1$.
\item Show that applying the flip $\stt 0110$ to the square torus inverts the slope of a trajectory.
\end{enumerate}
\end{exercise}

\begin{exercise} \index{continued fraction}\index{cutting sequence}\index{square torus}
Find the continued fraction expansion of the slope of the trajectory corresponding to the cutting sequence $\overline{BABBABBABB}$ using only the cutting sequence, as in Examples $\ref{crossoutstuff}-\ref{crossoutstuffcont}$.
\end{exercise}

\begin{exercise}\index{cutting sequence}\label{usefivesevenths}\index{square torus}
Find the cutting sequence corresponding to a trajectory with slope $5/7$, as in Example \ref{constructcuttseq}. You may use the continued fraction expansion of $5/7$ that you found in Exercise \ref{findcfe}. 
\end{exercise}

\section{Every shear can be understood via basic shears}\index{shear}

We gave a simple rule (Theorem \ref{kosl}) for the effect of the shear $\stt 10{-1}1$ on the cutting sequence corresponding to a trajectory whose slope is greater than $1$. We chose to study $\stt 10{-1}1$ because its effect on the slope of the trajectory is easier to state (``decrease slope by $1$'') than our other three options.\footnote{See Exercise \ref{ex-simplest} for the effects of the other basic shears.} We also chose this shear because it is directly related to the continued fraction expansion of the slope, as we explored in the previous section, and because this shear makes the associated cutting sequence simpler (decreases the period of a periodic sequence). Now that we understand this one shear, we can use our result to determine the effect of other shears on trajectories. First, we determine the effects of the two basic shears.\index{square torus}

Recall that the effect of $\nupshear$ on a trajectory on the square torus with slope greater than $1$, with respect to the associated cutting sequence, is to delete an $A$ from each string of $A$s (Theorem \ref{kosl}). The effects of the shears $\shear$ and $\upshear$, which we call ``basic shears,'' are similar:

\begin{proposition}\label{basic-effects}\index{cutting sequence}\index{shear}\index{square torus}
The effects of applying the basic shears $\shear$ and $\upshear$ to a linear trajectory on the square torus, with respect to the effect on the associated cutting sequence, are:
\begin{enumerate}[(a)]
\item $\upshear$: Lengthen every string of $A$s by $1$.
\item $\shear$ Lengthen every string of $B$s by $1$.
\end{enumerate}
Note that a ``string of $A$s" or a ``string of $B$s'' can have length $0$.
\end{proposition}

\begin{proof}
\begin{enumerate}[(a)]
\item $\upshear$ is the inverse of $\stt 10{-1}1$, so the action is also the inverse.
\item $\shear$ is a horizontal action instead of a vertical action, so the roles of $A$ and $B$ are reversed.
\end{enumerate}

For completeness, we demonstrate (b) explicitly: Since $\shear$ is a horizontal shear, the intersections on edge $A$ are preserved, and we must trace back the new edge $B$ to see where it comes from (Figure \ref{favorite-shear}). We see that we get a new $B$ exactly when the original trajectory crosses the negative diagonal. Every segment with positive slope crosses the negative diagonal, so we get a $b$ between every pair of letters. When we cross out the $B$s and change the $b$s to $B$s, the effect is to lengthen each string of $B$s by $1$.
\end{proof}

\begin{example} \index{cutting sequence}
We can see the action of $\shear$ on an example trajectory (Figure \ref{favorite-shear}):

\begin{figure}[!h]
\begin{center}
\includegraphics[height=80pt]{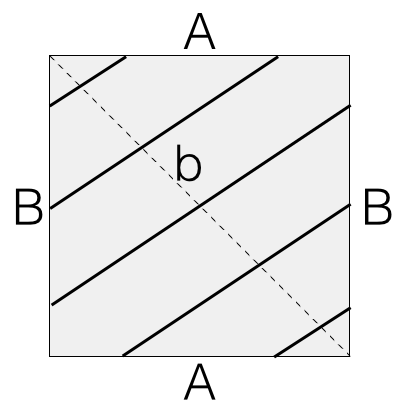} \ \ \ \ \
\includegraphics[height=70pt]{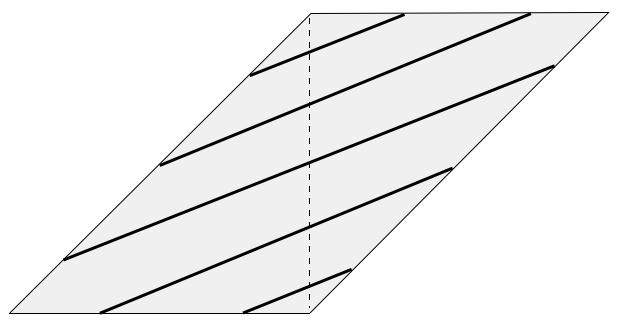} \ \ \ \ \
\includegraphics[height=80pt]{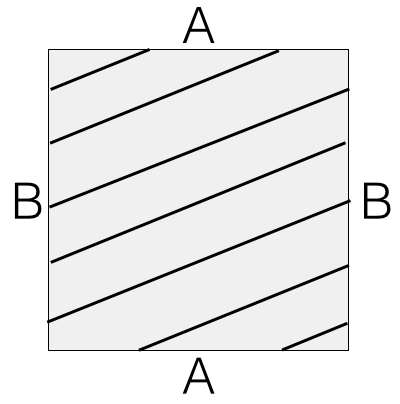} 
\caption{We horizontally shear the trajectory with cutting sequence $\overline{ABABB}$, obtaining the trajectory with cutting sequence $ \overline{ABBABBB}$ \label{favorite-shear}}
\end{center}
\end{figure}

The combinatorial action on the cutting sequence is:
\begin{equation*}
\overline{ABABB} \to \overline{AbBbAbBbBb} \to \overline{AbbAbbb} \to \overline{ABBABBB},
\end{equation*}
so indeed we have lengthened each string of $B$s by $1$.

\end{example}

Now, we can reduce every shear to a composition of these two:


\begin{proposition} \label{reducingprop}\index{shear}\index{determinant}
Every $2\times 2$ matrix with nonnegative integer entries and determinant $1$ is a product of powers of the basic shears $\shear$ and $\upshear$.
\end{proposition}

We only consider matrices with integer entries and determinant $1$ because they are the symmetries of the square torus: determinant $1$ ensures that they preserve area, which is essential if we want to reassemble the sheared torus back into a square with no overlaps and no gaps, and integer entries ensure that vertices of the square are sent to vertices of the square grid.\index{symmetry}

\begin{example} \label{composed-shear}
$\left[\begin{smallmatrix} 3&7 \\ 2&5 \end{smallmatrix}\right] = \shear \upshear^2 \shear^2$.
\end{example}

To prove Proposition \ref{reducingprop}, we need a simple lemma.

\begin{lemma} \label{reducinglemma}\index{determinant}
For a matrix $\stt abcd$ with determinant $1$ and $a,b,c,d\geq 0$, either \mbox{$a\leq c$} and \mbox{$b\leq d$}, or \mbox{$a\geq c$} and \mbox{$b\geq d$}.
\end{lemma}

\begin{proof}
First, notice that if $a=c$ and $b=d$, then the determinant is $0$, which is impossible. Also, if $a=c$ and $b>d$, or if $a<c$ and $b=d$, then the determinant is negative, which is impossible.

Next, notice that if $a> c$ and $b=d$, then the matrix must be of the form $\stt {c+1}1c1$, which satisfies the statement.
Similarly, if $a=c$ and $b<d$, then the matrix must be of the form $\stt 1b1{b+1}$, which satisfies the statement.

Now if $a>c$ and $b<d$, then $$ad-bc=1 \implies cd-bc <1 \implies cd-dc<0,$$ a contradiction.

Finally, if $a<c$ and $b>d$, then $$ad-bc=1 \implies ad-ba<1 \implies ab-ba<0,$$ a contradiction.
%
%
%
%
%
\end{proof}

\begin{proof}[Proof of Proposition \ref{reducingprop}]\index{determinant}
First, consider a matrix $\stt abcd$ with $a,b,c,d\geq 0$. By Lemma \ref{reducinglemma}, either  \mbox{$a\leq c$} and \mbox{$b\leq d$}, or \mbox{$a\geq c$} and \mbox{$b\geq d$}.

If $a\geq c$ and $b\geq d$, then we apply $\shear^{-1} = \stt 1{-1}01$. The resulting matrix $\stt {a-c}{b-d}cd$ also has determinant $1$.

Similarly, if $a\leq c$ and $b\leq d$, we apply $\upshear^{-1} = \stt 10{-1}1$.  The resulting matrix  $\stt ab{-a+c}{-b+d}$ also has determinant $1$.


We repeat this process, applying $\shear^{-1}$ or $\upshear^{-1}$ depending on whether $a$ and $b$ are greater than or less than $c$ and $d$. Since all the entries are nonnegative, and at least two of them are decreased at every step, eventually at least one entry will be $0$. 
If one of the entries is $0$, then
the matrix is either $\stt 1 n 01 = \shear^n$, or it is $\stt 10n1 = \upshear^n$.
Since applications of $\shear^{-1}$ and $\upshear^{-1}$ to our original matrix eventually yields a power of $\shear$ or $\upshear$, the original matrix must be a composition of $\shear$ and $\upshear$.
%
\end{proof}

\begin{example}\index{cutting sequence}\index{shear}\index{square torus}
Suppose that we have the trajectory of slope $1$, with cutting sequence $\overline{AB}$, on the square torus. We shear the torus via the matrix \mbox{$\stt 3725=\shear \upshear^2 \shear^2$}. Reading from right to left and applying Proposition \ref{basic-effects}, we compute the resulting cutting sequence:

\begin{align*}
  & &\text{slope} \\
 \text{(original sequence)} & {\color{white}\to \  } \overline{AB} & 1 \\
\text{lengthen every string of $B$s by $1$} &\to \overline{ABB} &1/2 \\
\text{lengthen every string of $B$s by $1$} &\to \overline {ABBB} &1/3 \\
\text{lengthen every string of $A$s by $1$} &\to \overline{AABABAB} &4/3 \\
\text{lengthen every string of $A$s by $1$} &\to \overline {AAABAABAAB} &7/3 \\
\text{lengthen every string of $B$s by $1$} &\to \overline{ABABABBABABBABABB} &7/10 \\
\end{align*}

When we shear the square torus by $\nupshear$ and reassemble the pieces back into a square, the result is a square with a cut through the negative diagonal (Figure \ref{shear-torus}). The analogous action by a non-basic shear looks a little more complicated:

\begin{figure}[!h]
\begin{center}
\includegraphics[width=230pt]{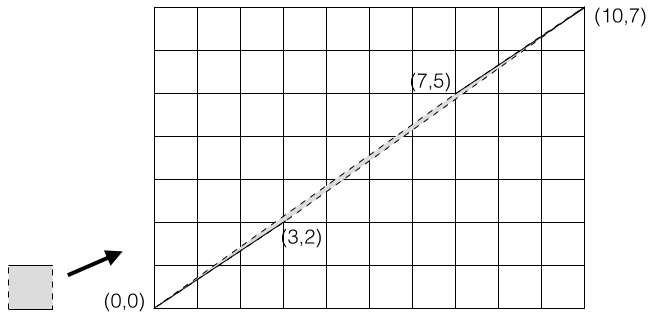} \ \ \ 
\includegraphics[width=110pt]{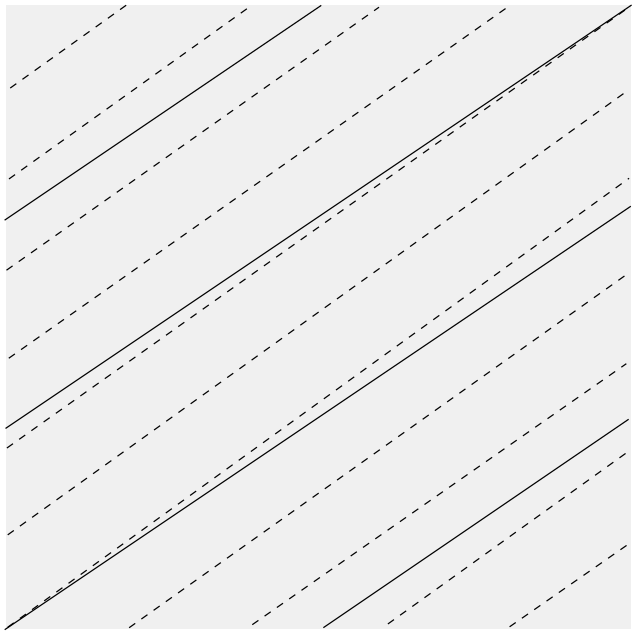} 
\caption{Shearing the square torus via the matrix $\protect\stt 3725\protect$, and reassembling the sheared parallelogram back into a square (enlarged to show detail) \label{bigshear}\index{shear}}
\end{center}
\end{figure}
\end{example}

\begin{exercise}\label{ex-simplest}
For each part below, show that the indicated matrix has the indicated effect on the rational slope of a given trajectory. You may use Exercise \ref{minusone} and Proposition \ref{basic-effects}.
\begin{enumerate}
\item $\stt 10{-1}1:\frac pq \mapsto \frac{p-q}q$.
\item $\upshear: \frac pq \mapsto \frac{p+q}q$.
\item $\shear: \frac pq \mapsto \frac p{p+q}$.
\item $\stt 1{-1}01: \frac pq \mapsto \frac p{p-q}$.
\end{enumerate}
\end{exercise}

\begin{exercise}
Make up a $2\times 2$ matrix with nonnegative integer entries and determinant $1$. Then write it as a composition of $\shear$ and $\upshear$, as in Example \ref{composed-shear}. You may want to use the strategy in the proof of Proposition \ref{reducingprop}.
\end{exercise}

\section{Polygon identification surfaces}


In Section \ref{squaretorus}, we identified the top and bottom edges, and the left and right edges, of a square, and obtained a surface: the square torus. Similarly, we can identify opposite parallel edges of a parallelogram, or a hexagon, or an octagon, or two regular pentagons, or many other polygons, to form a surface.


We determined that when we identify opposite parallel edges of a square, the resulting surface is a torus (Figure \ref{curved-torus}). By the same construction, identifying opposite parallel edges of a parallelogram also results in the torus (Figure \ref{par-torus}).  \index{parallelogram} \index{torus}

\begin{figure}[!h]
\begin{center}
\includegraphics[width=60pt]{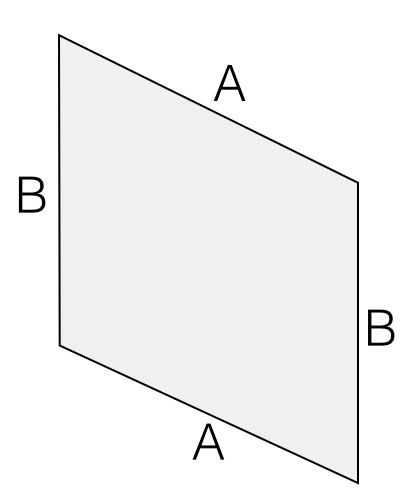} \ \ \ \ \
\includegraphics[width=120pt]{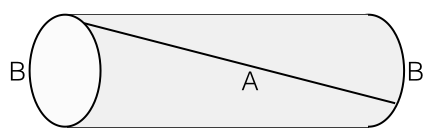}\ \ \ \ \
\includegraphics[width=123pt]{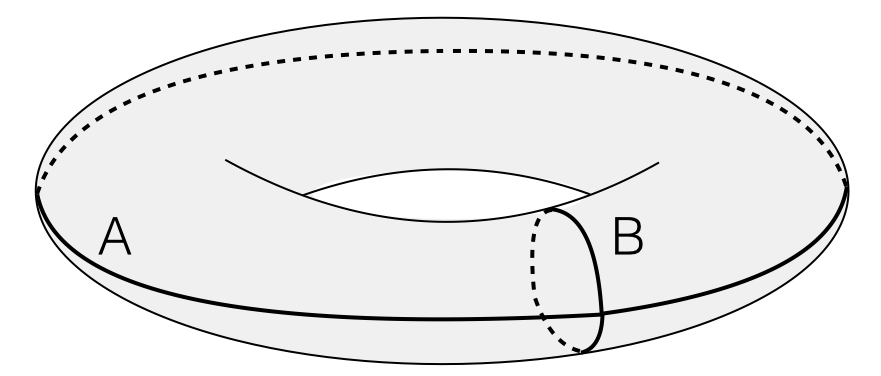}
\caption{Wrapping a parallelogram into a torus \label{par-torus}}
\end{center}
\end{figure}

It turns out that the surface we get by identifying opposite parallel edges of a hexagon is also the torus.  \index{hexagon} Any hexagon with three pairs of opposite parallel edges, identified cyclically, tiles the plane (Figure \ref{hex-tiling}a). Choose a point in the hexagon, and mark that point in each hexagon (Figure \ref{hex-tiling}b). Then we can connect these points to form parallelograms. When we cut out just one such parallelogram, opposite parallel edges are at corresponding locations in the original hexagons, so opposite parallel edges of the parallelogram are identified, so the object is a torus (Figure \ref{hex-tiling}c), and the hexagons are just a different way of cutting up the same surface.

\begin{figure}[!h]
\begin{center}
\includegraphics[width=80pt]{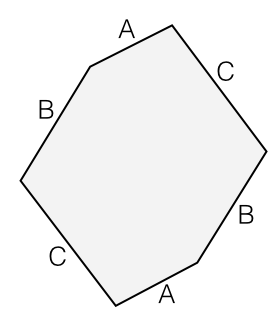} \ \ \ \ \
\includegraphics[width=140pt]{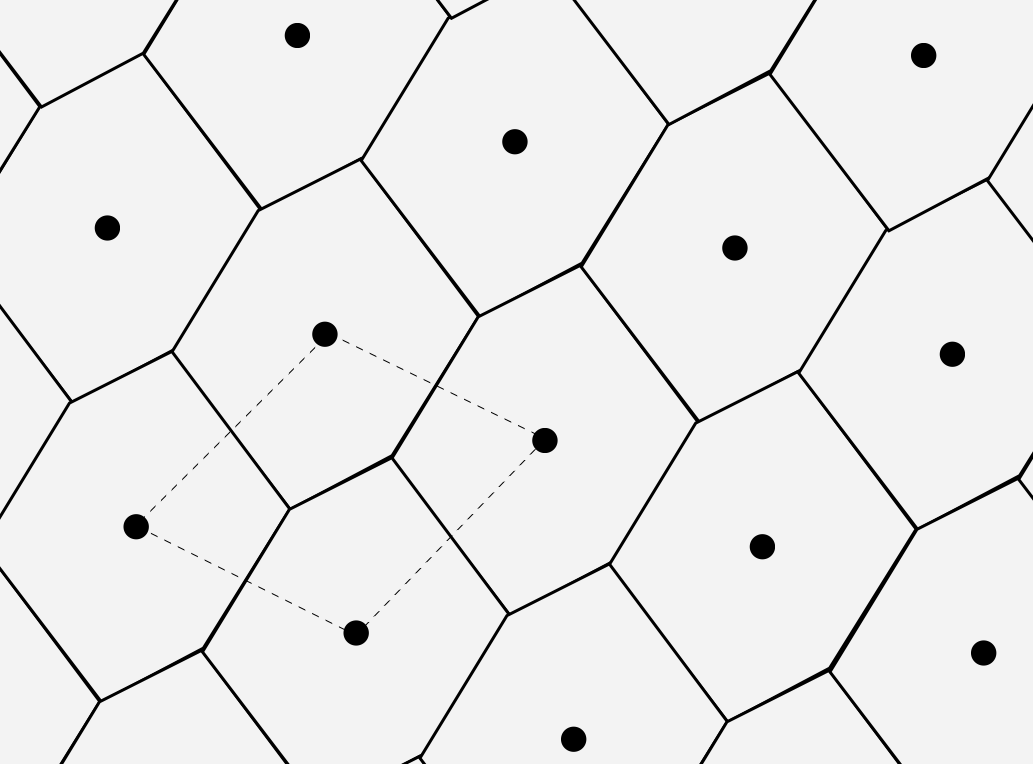}\ \ \ \ \
\includegraphics[width=90pt]{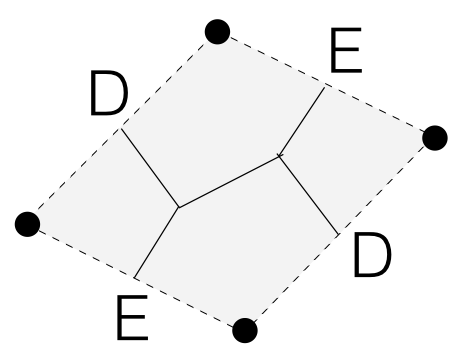}
\caption{Tiling the plane with hexagons \label{hex-tiling}}
\end{center}
\end{figure}

\begin{exercise}\index{hexagon}
In Figure \ref{hex-tiling}, we tiled the plane with a hexagon that has three pairs of opposite parallel edges. Our ``random'' hexagon happened to be convex. Does a non-convex hexagon with three pairs of opposite parallel edges still tile the plane?
\end{exercise}

Identifying \emph{opposite parallel edges} means gluing together two edges that are parallel and \emph{oppositely oriented}. An \emph{oriented edge} is an edge with an arrow pointing out of the edge, perpendicular to the edge and out of the polygon. We identify oppositely oriented edges, because then when you translate the identified edges to the same place, the surfaces are next to each other, rather than on top of each other. This ensures that a trajectory passing into one edge comes out the identified edge going the same direction. 

Figure \ref{not-opp-par} shows a surface where two of the edge identifications ($B$ and $C$) are of opposite parallel edges, and two of the edge identifications ($A$ and $D$) are of non-opposite parallel edges. (The shading on the vertices indicates which corners match up.) The trajectory ends up going two different directions. Edges $B$ and $C$ are identified by translation, while if we want to put edges $A$ and $D$ next to each other, we have to turn one of the squares $180^{\circ}$ or flip one square over onto the other, so the surface in Figure \ref{not-opp-par} is not a translation surface. You can see that if you constructed this surface with paper and tape, it would look like a pillowcase, so to a small creature living on the surface, the vertices look different from the other points on the surface.

\begin{figure}[!h]
\begin{center}
\includegraphics[width=200pt]{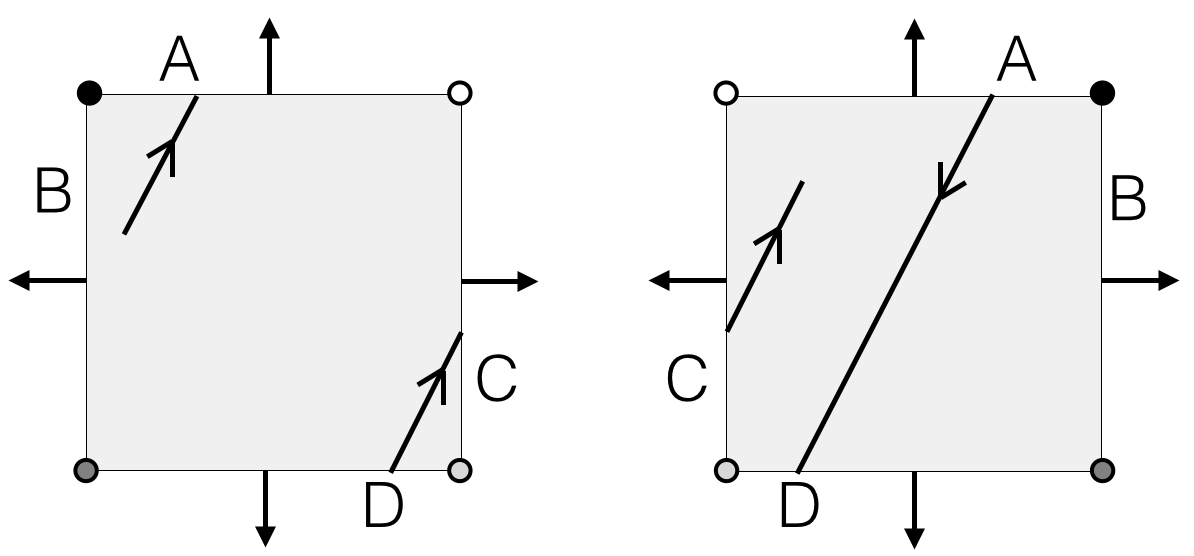} 
\caption{A surface with parallel edge identifications that are \emph{not} opposite. \label{not-opp-par}}
\end{center}
\end{figure}

We can create a surface by identifying opposite parallel edges of a single polygon, as we have been doing, and we can do the same with two polygons, or with any number of polygons. Figure \ref{multipoly} shows some examples. \index{double pentagon} \index{octagon}

\begin{figure}[!h]
\begin{center}
\includegraphics[width=90pt]{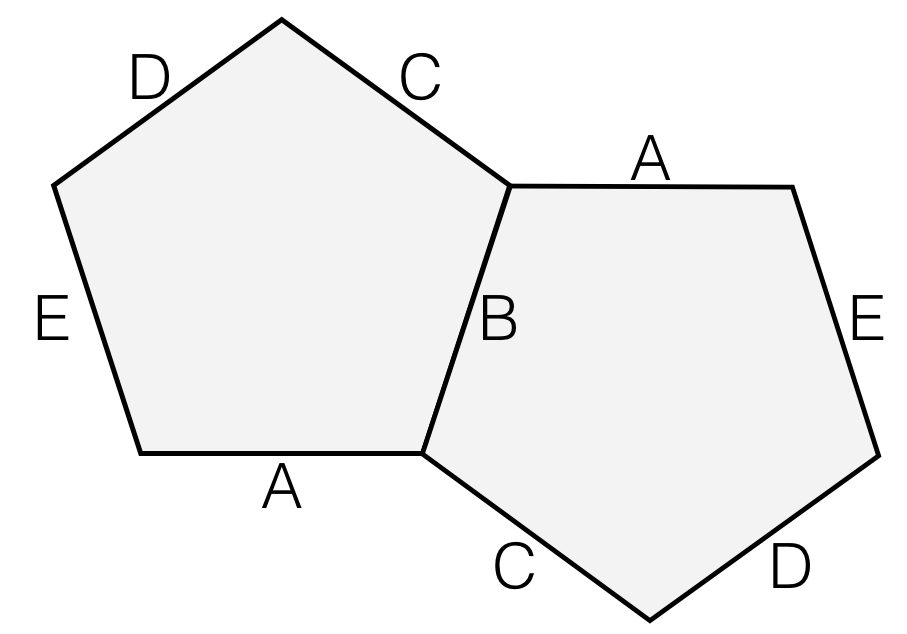} \ \ \ \
\includegraphics[width=60pt]{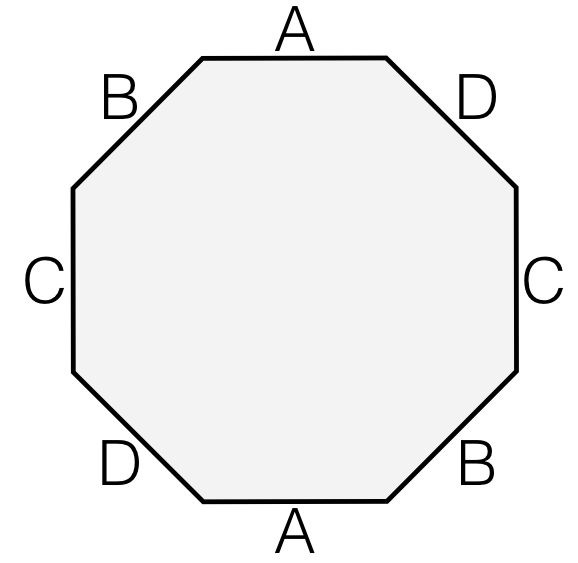} \ \ \  \ \ \ \ \ 
\includegraphics[width=150pt]{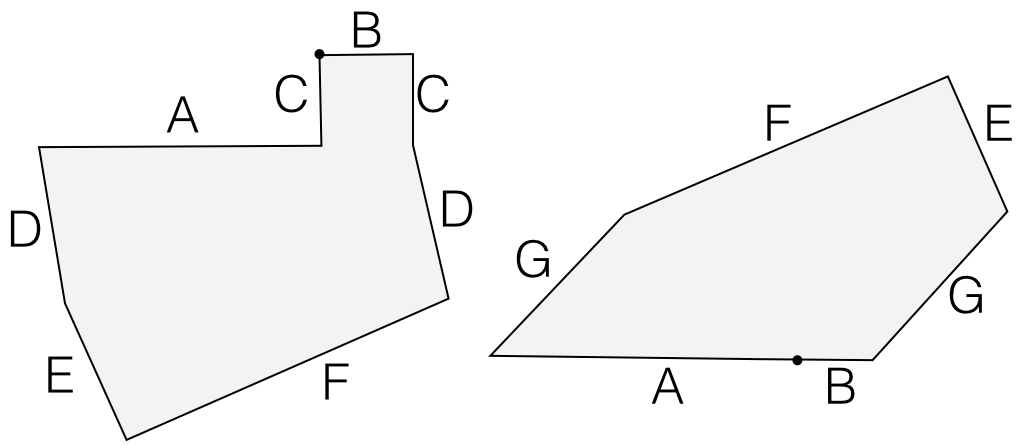}
\caption{Edges marked with the same letter are identified by translation. \label{multipoly}}
\end{center}
\end{figure}

The double pentagon and octagon surfaces in Figure \ref{multipoly} have a lot of symmetry, and we will explore them in Section \ref{regularpolygons}. The ``random'' surface on the right side of Figure \ref{multipoly} is just an arbitrary surface with opposite parallel edges identified.

We can ask the same questions about polygon identification surfaces in general that we explored in the special case of the square torus, such as which directions are periodic. For the square torus, the periodic directions are exactly those with rational slope. For other surfaces, the criterion for periodicity is different. We call a trajectory \emph{periodic} if it repeats, which is equivalent to it visiting  corresponding points in two congruent polygons in the unfolding. 

\begin{example}
In Figure \ref{per-directions}a, slope $2/3$ is a periodic direction because it passes through two points (marked with black dots) that are in the same position in two different squares. If we connect the centers of consecutive squares that the trajectory passes through, we end up with three arrows in direction $[1,0]$ and two in direction $[0,1]$, so the trajectory is in the direction $[3,2]$.

In Figure \ref{per-directions}b, \index{double pentagon} the trajectory is periodic, because it passes through two points (marked with black dots) that are in the same position in two different pentagons. If we connect the centers of consecutive pentagons that the trajectory passes through, we end up with the vector
$$2[\cos(2\pi/5),\sin(2\pi/5)]+[\cos(-2\pi/5),\sin(-2\pi/5)]+[\cos(\pi/10),\sin(\pi/10)].$$

In Figure \ref{per-directions}c, the trajectory is periodic,\index{octagon} because it passes through two points (marked with black dots) that are in the same position in two different octagons. If we connect the centers of consecutive octagons that the trajectory passes through, we end up with the vector
$$2[1,0]+2\left[\sqrt{2}/2,\sqrt{2}/2\right].$$
\end{example}

\begin{figure}[!h]
\begin{center}
\includegraphics[height=65pt]{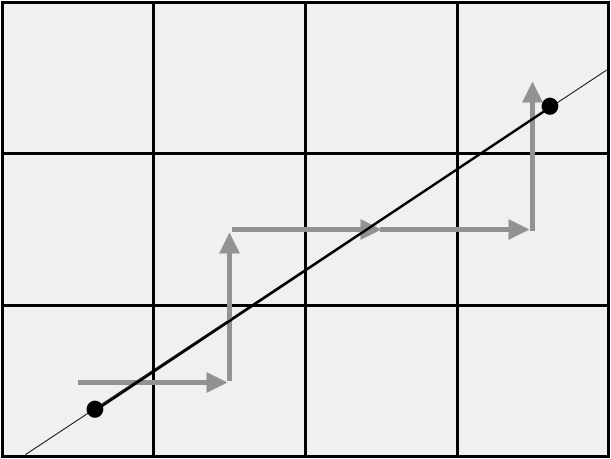} \ \ \  \
\includegraphics[height=65pt]{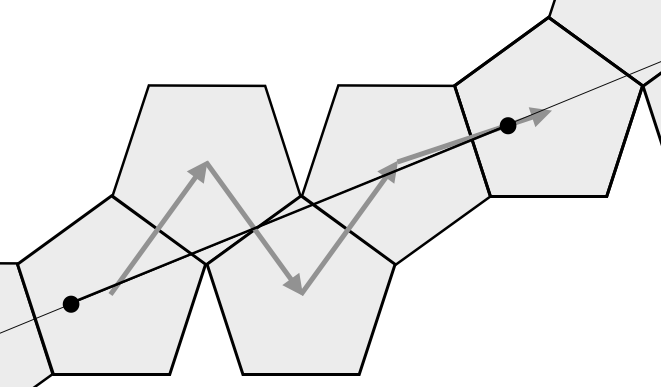} \ \ \  \
\includegraphics[height=65pt]{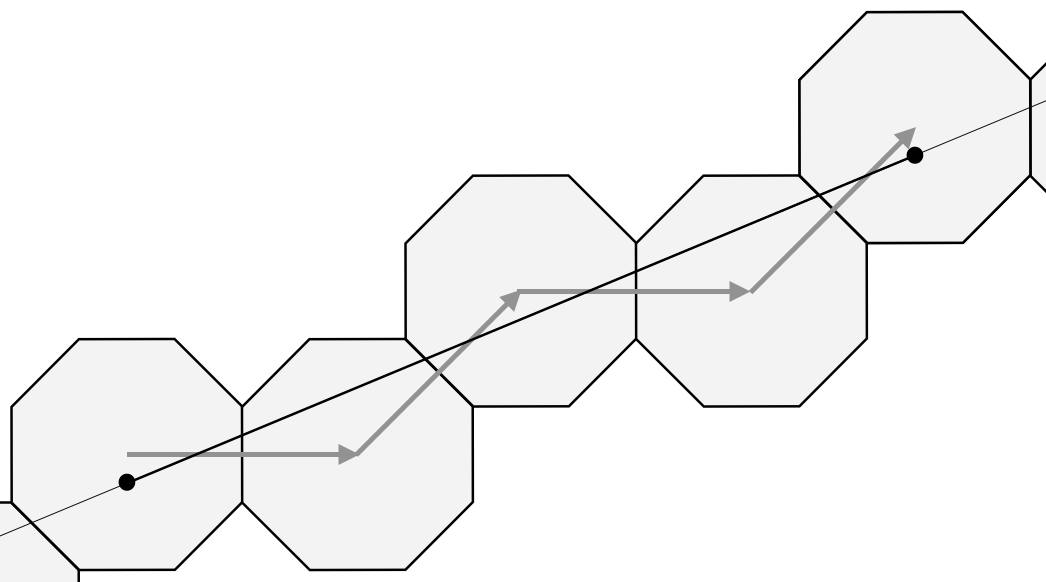}
\caption{Periodic directions in the square, double pentagon and octagon \label{per-directions}}
\end{center}
\end{figure}

\begin{proposition}
For a surface made by identifying opposite parallel edges of congruent regular polygons, the rational directions are the directions that are sums of unit vectors in directions perpendicular to each edge. 
\end{proposition}

%

\newpage
\section{Vertices and the Euler characteristic} \index{Euler characteristic}\index{vertex}\label{vert-eul-char}

Once we've made a surface, the Euler characteristic gives us a way of easily determining what kind of surface we obtain, without needing to come up with a clever trick like cutting up and reassembling  hexagons into parallelograms.

\begin{definition} \label{eulerchar}
Given a surface $S$ made by identifying edges of polygons, with $V$ vertices, $E$ edges, and $F$ faces, its \emph{Euler characteristic} is $\chi(S) = V-E+F$.
\end{definition}

\begin{example} \label{ecex}\index{tetrahedron}\index{cube}
A tetrahedron has $4$ vertices, $6$ edges and $4$ faces, so $$\chi(\text{tetrahedron})=4-6+4=2.$$
A cube has $8$ vertices, $12$ edges and $6$ faces, so $$\chi(\text{cube})=8-12+6=2.$$
The surface created by identifying opposite parallel edges of a square has one vertex (see Section \ref{st}), two edges ($A$ and $B$) and one face, so $$\chi(\text{square torus})=1-2+1=0.$$
\end{example}

\begin{exercise}\index{Euler characteristic}
Find the Euler characteristic of the octahedron, of a triangular prism, and of another surface of your choice.
\end{exercise}

We can use the Euler characteristic to determine the genus (number of holes) of a surface.\index{genus}

\begin{theorem}\index{Euler characteristic}\label{eulerthm}
A surface $S$ with genus $g$ has Euler characteristic $\chi(S)=2-2g$.
\end{theorem}

\begin{example} \label{ecgenus} We determine the genus of the surfaces from Example \ref{ecex}:\index{tetrahedron}\index{cube}\index{genus}
$$\chi(\text{tetrahedron})=\chi(\text{cube})=2 \wimplies 2-2g=2 \wimplies g=0,$$ so a tetrahedron and a cube both have genus $0$ (no holes). $$\chi(\text{square torus})=0 \wimplies 2-2g=2 \wimplies g=1,$$ so the torus has genus $1$ (one hole).  All of these results are consistent with our previous knowledge of tetrahedra, cubes and tori.
\end{example}


Given a polygon, or collection of polygons, whose edges have been identified to form a surface, it is easy to count $E$ and $F$: $E$ is half the total number of edges of the polygons, since they are pairwise-identified, and $F$ is the number of polygons you have. To determine the number of vertices, we have to do a little bit more work.\\

\newpage
\noindent {\bf Finding the different vertices of a surface} 
\begin{example}
It's easiest to explain how to identify the vertices of a surface with an example (Figure \ref{vertex-chasing-torus}). Let's do the square torus: first, mark any vertex (say, the top left). We want to see which other vertices are the same as this one. The marked vertex is at the left end of edge $A$, so we also mark the left end of the bottom edge $A$. We can see that the top and bottom ends of edge $B$ on the left are now both marked, so we mark the top and bottom ends of edge $B$ on the right, as well. Now all of the vertices are marked, so the surface has just one vertex. (We saw this via a different method in Section \ref{st}, by noticing that the four corners of the square all come together.)
\end{example}

\begin{figure}[!h]
\begin{center}
\includegraphics[width=100pt]{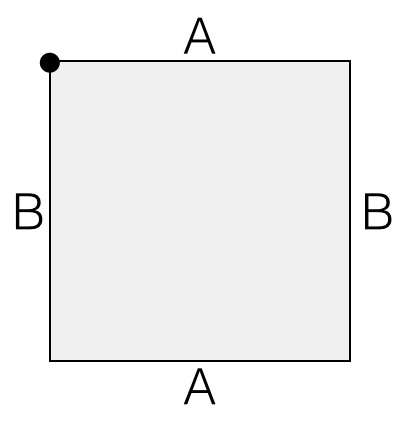} \ \ \ \ \
\includegraphics[width=100pt]{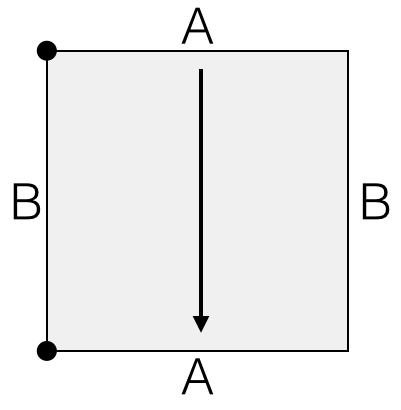}\ \ \ \ \
\includegraphics[width=103pt]{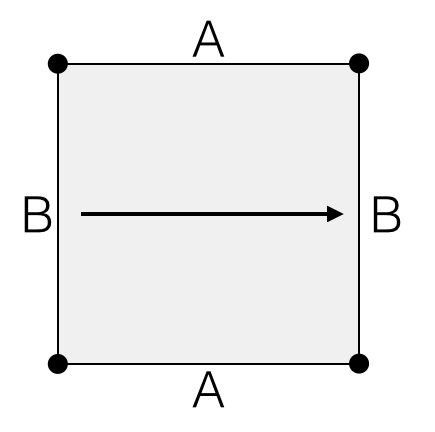}
\caption{The square torus has one vertex. \label{vertex-chasing-torus}}
\end{center}
\end{figure}

\begin{example} \index{hexagon}\index{vertex}
We determine how many vertices a hexagon with opposite parallel edges identified has (Figure \ref{vertex-chasing-hex}): First, mark any vertex (say, the left end of the top edge $A$ again). This vertex is also at the left end of the bottom edge $A$, so we mark that vertex as well. Now this marked vertex is at the right of edge $C$ on the left side of the hexagon, so we mark the right end of edge $C$ on the left side of the hexagon as well. Now this marked vertex is at the left end of edge $B$ on the right side of the hexagon, so we mark the left end of edge $B$ on the left side of the hexagon $-$ and we find ourselves back where we started. So these three vertices of the hexagon are the same vertex in the surface.

\begin{figure}[!h]
\begin{center}
\includegraphics[width=80pt]{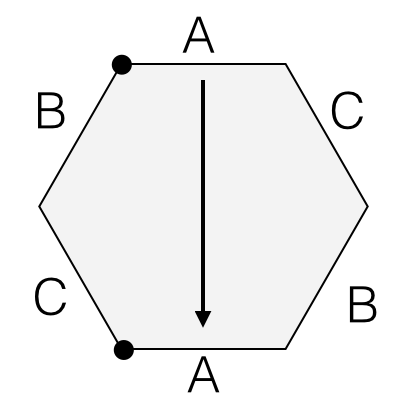} \
\includegraphics[width=80pt]{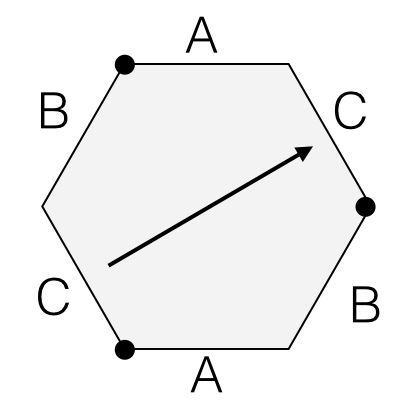}\
\includegraphics[width=80pt]{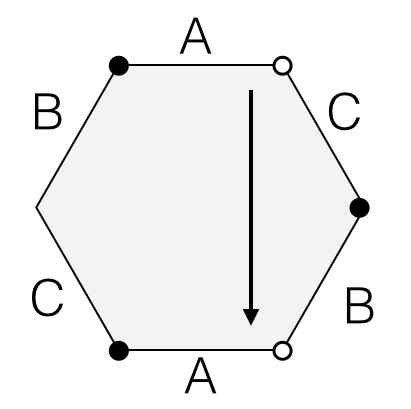} \
\includegraphics[width=80pt]{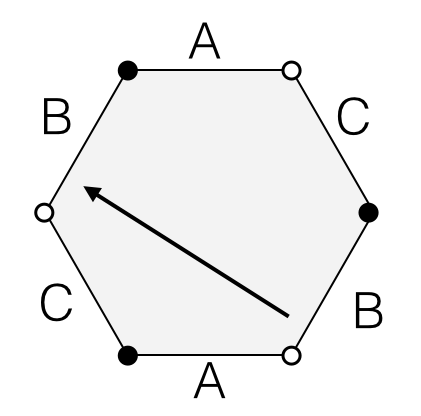}
\caption{The hexagonal torus has two vertices. \label{vertex-chasing-hex}}
\end{center}
\end{figure}

Now we choose any un-marked vertex on the hexagon (say, the right end of the top edge $A$) and mark it with a different color (white in Figure \ref{vertex-chasing-hex}). This vertex is also at the right end of the bottom edge $A$, so we mark that vertex white as well. Now the right end of the bottom edge $A$ is also the left end of edge $B$, so we mark the left end of vertex $B$ on the left side of the hexagon. This is the left vertex of edge $C$, so we mark the left vertex of edge $C$ on the right side of the hexagon white as well $-$ and we are again back where we started.

We have discovered that the hexagon surface has two vertices, each corresponding to three vertices on the flat hexagon.
\end{example}

\begin{example} \index{hexagon}\index{Euler characteristic}\index{genus}
Now we can calculate the Euler characteristic of the hexagon surface: It has two vertices, three edges and one face, so its Euler characteristic is $$\chi(\text{hexagonal torus})=2-3+1=0.$$ As  in Theorem \ref{eulerthm} and Example \ref{ecgenus}, an Euler characteristic of $0$ means that the surface has genus $1$, so the hexagon surface is indeed the torus.
\end{example}

\noindent {\bf Determining the angle around a vertex} 
\begin{example}
We can determine the angle around a vertex by ``walking around'' it. \index{vertex, angle around} To do this, first choose a vertex (say, the top left vertex of the hexagon, between edges $A$ and $B$, which we marked as black) and draw a counter-clockwise arrow around the vertex. In our example, this arrow goes from the top end of edge $B$ to the left end of edge $A$ (Figure \ref{vertex-walking-hex}a). Now we find where that arrow ``comes out'' on the identified edge $A$ at the bottom of the hexagon, and keep going: now the arrow goes from the left end of the bottom edge $A$ to the bottom end of the left edge $C$. We keep going at the bottom of the right edge $C$ and draw an arrow to the top end of right edge $B$. We find the identified point on the top end of left edge $B$, and see that this is where we started! So the angle around the black vertex is $3\cdot 2\pi/3=2\pi$.

By the same method, we can see that the angle around the white vertex is also $2\pi$ (Figure \ref{vertex-walking-hex}b).
\end{example}

\begin{figure}[!h]
\begin{center}
\includegraphics[height=100pt]{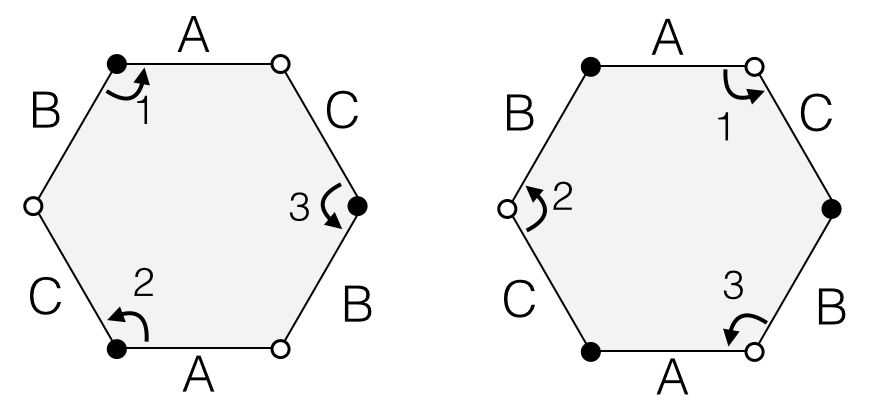} 
\caption{The angle around each vertex of the hexagonal torus is $2\pi$. \label{vertex-walking-hex}\index{hexagon}}
\end{center}
\end{figure}

Since the black and white vertices each have $2\pi$ of angle around them, all the corners of the surface come together in a flat plane, as we previously saw in the picture of the hexagons tiling the plane in Figure \ref{hex-tiling}.

\begin{definition} \index{flat surface} \index{cone point}
A surface is called \emph{flat} if it looks like the flat plane everywhere, except possibly at finitely many \emph{cone points} (vertices), where the vertex angle is a multiple of $2\pi$.
\end{definition}


\begin{exercise} \index{vertex, angle around} \label{flat-proof}\index{flat surface}
Prove that if a surface is created by identifying opposite parallel edges of a collection of polygons, then it is \emph{flat}: every vertex of that surface has a vertex angle that is a multiple of $2\pi$.
\end{exercise}

\begin{exercise} \label{not-flat}
Show that the surface in Figure \ref{not-opp-par} is not flat, i.e. it has a vertex with an angle that is not a multiple of $2\pi$. Find its genus. Are you surprised?
\end{exercise}


\begin{exercise}\index{vertex, angle around}\index{Euler characteristic}\index{genus}
Find the number of vertices, and the angle around each vertex, and the Euler characteristic, and the genus, of 
each of the surfaces in Figure \ref{multipoly}.
\end{exercise}


\begin{exercise} \index{vertex, angle around}
For each of the following, construct a surface, made from polygons with opposite parallel edges identified, with the given property, or explain why it is not possible to do so:
\begin{enumerate}[(a)]
\item One of the vertices has  angle $6\pi$ around it.
\item One of the vertices has angle $\pi$ around it.
\item Two of the vertices have different angles around them.
\end{enumerate}
\end{exercise}

\begin{exercise}\index{cone point}\index{vertex, angle around}
What does it look like to have $6\pi$ of angle at a vertex? Cut slits in three sheets of paper, and tape the edges together as in Figure \ref{six-pi}. The vertex angle at the white point is now that of three planes, which is $6\pi$.
\end{exercise}

\begin{figure}[!h]
\begin{center}
\includegraphics[height=90pt]{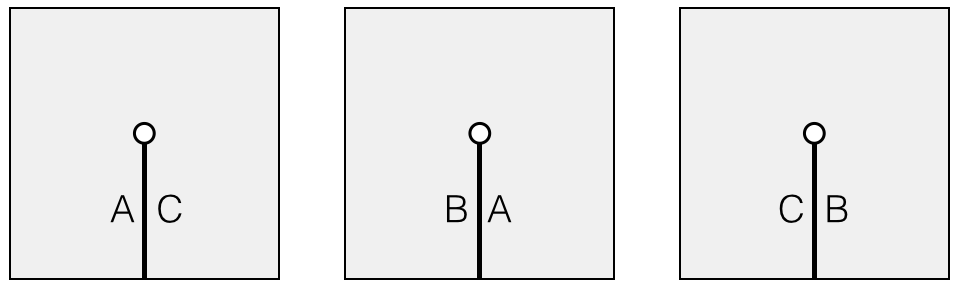} 
\caption{How to construct a vertex angle of $6\pi$ out of paper and tape. \label{six-pi}}
\end{center}
\end{figure}

\begin{exercise} Determine how many vertices the surface created by identifying opposite parallel edges of a $2n$-gon has. Determine how many vertices the surface created by identifying opposite parallel edges of two $n$-gons, one of which is a $180^{\circ}$ rotation of the other, has. (See Figure \ref{multipoly}.) Your answer may depend on $n$.
\end{exercise}

Many of the ideas that we explored for the square torus also apply to other surfaces made from polygons.  In Section \ref{st}, we sheared the square torus into a parallelogram and then reassembled the pieces, which was a twist of the surface. We were able to do this because the square is extremely symmetric. In Section \ref{letssheareverything}, we will see that we can do this same shearing and twisting on many other surfaces, which we introduce in Sections \ref{squaretiledsurfaces}, \ref{regularpolygons}, \ref{wardsurfaces}, and \ref{bmsurfaces}.


 \section{Cylinders}\index{cylinder}\label{cylinders} \label{letssheareverything}\index{shear}
 
 In Section \ref{st}, we sheared the square torus by the matrix $\shear$, which transformed it into a parallelogram, and then we reassembled the pieces back into a square (Figure \ref{shear-twist}). That action was a twist of the torus surface. We can apply the same kind of transformation to many other surfaces.\index{torus}\index{shear}
 
 \begin{example} \label{shearltable}\index{L-shaped table}
 Consider the \lst made of three squares, with edge identifications as shown in Figure \ref{shearltablefig}. We shear it by the matrix $\stt 1201$ and then reassemble (by translation) the pieces back into the $\mathsf{L}$ shape.
 
 \begin{figure}[!h]
 \begin{center}
\includegraphics[width=320pt]{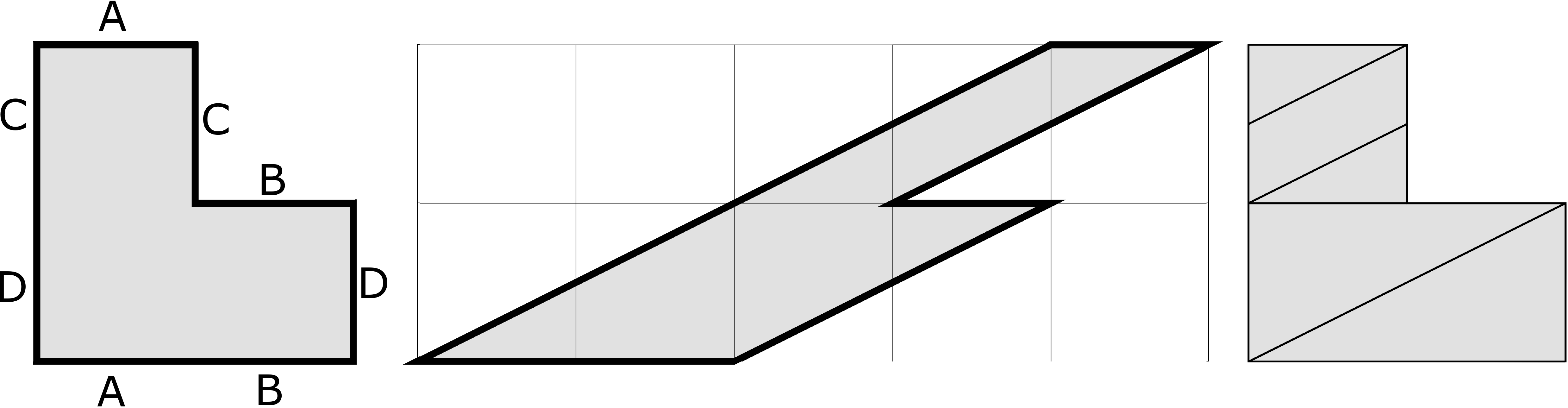}
\caption{We shear, we reassemble, we conquer \label{shearltablefig}}
\end{center}
\end{figure}
  \end{example}
  
  \begin{exercise}
  Show how to reassemble (by translation) the pieces of the sheared \lst in the middle of Figure \ref{shearltablefig} into the \lst on the right. 
  \end{exercise}
 
  Notice that the $2\times 1$ rectangle at the bottom has been twisted once, like the square in Figure \ref{shear-twist}. The $1\times 1$ square at the top has been twisted twice. This is because the square is half as wide as the rectangle. Let's develop some language to describe this:\index{shear}

  \begin{definition} \label{def:modulus}\label{def:cylinder}\index{cylinder}\index{modulus}
For a surface made from congruent copies of a regular polygon, a \emph{cylinder direction} is a direction of any trajectory that goes from a vertex to another vertex.
  

We can partition such a surface into \emph{cylinders}, whose boundaries are in the cylinder direction and which have no vertices on the interior. (To construct the cylinders, draw a line in the cylinder direction through each vertex of the surface, which cuts the surface up into strips, and then follow the edge identifications to see which strips are glued together. If a line divides two strips that are in the same cylinder, delete it.)

The \emph{modulus} of a cylinder is the ratio of its width to its height. The width and height are measured parallel to, and perpendicular to, the cylinder direction.
  \end{definition}
  
For the square torus, there is just one horizontal cylinder, with modulus $1$.
    
The \lst in Example \ref{shearltable} has two horizontal cylinders, one of modulus $1$ and the other of modulus $2$.

For the square torus, and for any square-tiled surface, cylinder directions are those with rational slope.

\begin{example}\index{octagon}\index{cylinder}\index{modulus}\label{oct-mod}
The regular octagon surface (Figure \ref{octcyl}) has two horizontal cylinders: The white rectangular center of the octagon, and grey parallelogram composed of the top and bottom trapezoids. We can calculate their moduli:

The rectangle has width $1+\sqrt{2}$ and height $1$, so its modulus is $1+\sqrt{2}$.

The parallelogram has width $2+\sqrt{2}$ and height $1/\sqrt{2}$, so its modulus is $\frac{2+\sqrt{2}}{\sqrt{2}} = 2(1+\sqrt{2})$.
 
One modulus is twice the other. This turns out to be the case for any  regular even-gon surface, as we will see later (Proposition \ref{evenmod}).
\end{example}

  \begin{figure}[!h]
  \begin{center}
\includegraphics[height=100pt]{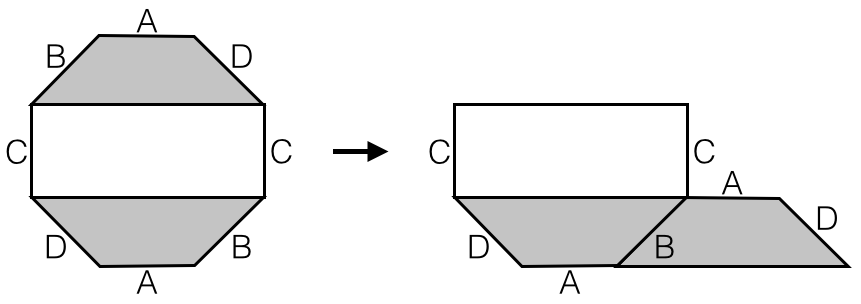}  
\caption{A cylinder decomposition for the regular octagon surface \label{octcyl}\index{octagon}}
\end{center}
\end{figure}

 \begin{exercise}\label{hexagon}\index{hexagon}\index{cylinder}\index{modulus}\index{cutting sequence}
 Decompose the regular hexagon into horizontal cylinders, and find the modulus of each. Do so for each of the orientations shown in Figure \ref{two-hexes}.
 
  \begin{figure}[!h]
  \begin{center}
\includegraphics[width=200pt]{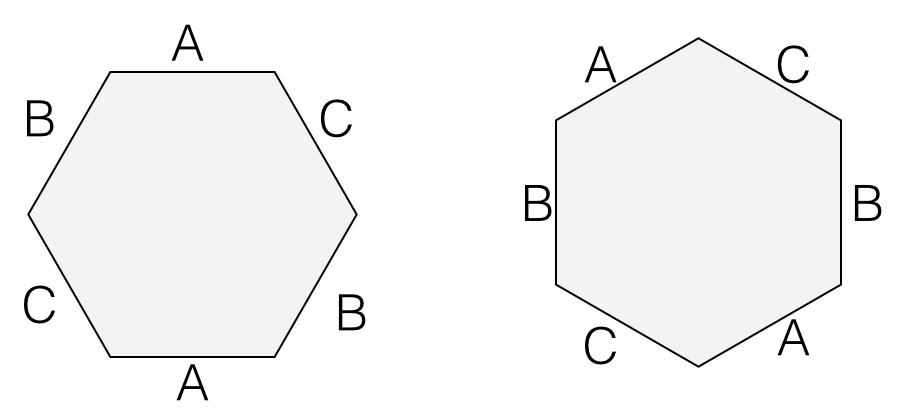}
\caption{Two different orientations of the regular hexagon surface \label{two-hexes}}
\end{center}
\end{figure}
 \end{exercise}

 Cylinders are in the direction of any trajectory that connects two vertices of the surface. Cylinder decompositions of the double pentagon in four different directions are shown in Figure \ref{pent-cyl}. \index{cylinder}
 
   \begin{figure}[!h]
  \begin{center}
\includegraphics[width=340pt]{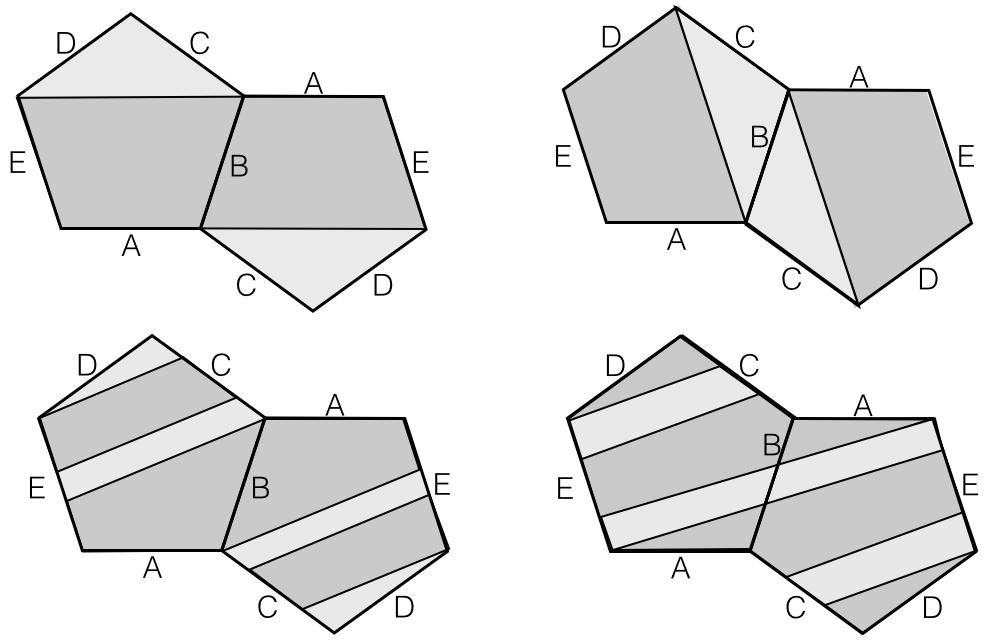} 
\caption{Four cylinder decompositions of the double pentagon surface \label{pent-cyl}}
\end{center}
\end{figure}

In Section \ref{correspondence}, we explored in great detail how to translate between the slope of a trajectory on the square torus, and its corresponding cutting sequence. This is possible because the cutting sequence corresponding to a trajectory on the square torus depends \emph{only} on the slope. \index{cutting sequence}

In fact, there is a different cutting sequence in each cylinder in the cylinder decomposition in a given direction. So surfaces with multiple cylinders have multiple cutting sequences corresponding to trajectories in the same direction.\index{cutting sequence}

\begin{example}\label{different-period}\index{cutting sequence}
For the regular octagon surface in Figure \ref{octcyl}, the horizontal trajectory in the white cylinder has cutting sequence $\overline{C}$, and in the grey cylinder has cutting sequence $\overline{BD}$.

For the double pentagon surface in Figure \ref{pent-cyl}, (a) cutting sequences in the horizontal direction are $\overline{BE}$ and $\overline{CD}$, (b) in the direction in (b) are $\overline{AD}$ and $\overline{BC}$, in the direction in (c) are $\overline{ABECEB}$ and $\overline{CDCE}$ and in the direction in (d) are $\overline{BECE}$ and $\overline{ABECDCEB}$.\index{double pentagon}

Notice that for a given direction, the cutting sequences need not have the same period.
\end{example}

\begin{exercise}\index{double pentagon}\index{cylinder}\index{cutting sequence}
\begin{enumerate}[(a)]
\item Construct a vertical cylinder decomposition of the double pentagon surface. Find the cutting sequence corresponding to a vertical trajectory in each cylinder. \index{double pentagon}\index{cylinder}
\item The two cylinder decompositions in the top line of Figure \ref{pent-cyl}a are essentially the same, just in a different direction. Is the vertical cylinder decomposition you found the same as either of those in the bottom line of Figure \ref{pent-cyl}?
\end{enumerate}
\end{exercise}

Any symmetry of a flat surface must take vertices to vertices, so any shear we apply to a flat surface must twist each cylinder an integer number of times.\index{shear}\index{symmetry}

\begin{lemma}
Let $m\in\mathbf{N}$. The shear $\stt 1m01$ twists the square torus $m$ times.
\end{lemma}

\begin{proof}
The shear $\stt 1m01$ transforms the vertical unit edges of the square into edges with vector $[m,1]$ which thus cut across $m$ squares as they achieve a height of $1$, so the square is twisted $m$ times.
\begin{figure}[!h]
\begin{center}
\includegraphics[width=360pt]{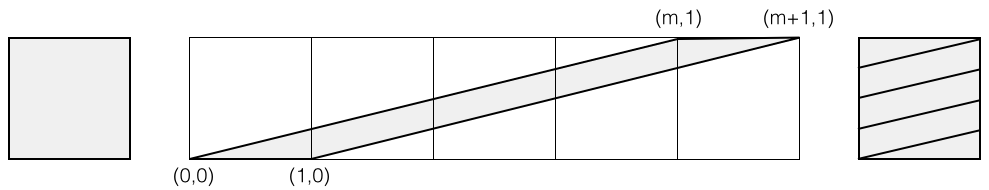}
\caption{The square torus is twisted $m$ times. Notice that when we cut up and reassemble the parallelogram back into the torus, we respect the edge identifications, as required. \label{m-twist}}
\end{center}
\end{figure}
\end{proof}
 
\begin{corollary} \label{howmanytwists}\index{cylinder}\index{modulus}\index{shear}
Suppose that a given surface has cylinders $c_1, c_2, \ldots, c_k$ with moduli $m_1, m_2, \ldots, m_k$, respectively. Let $M$ be an integer multiple of each $m_i$. Then the shear $\stt 1M01$ twists the cylinder $c_i$ $M/m_i$ times.
\end{corollary}

\begin{corollary} \label{whatcanweshear}\index{symmetry}
A given surface has can be sheared in a given direction, in such a way that vertices are taken to vertices, nearby points go to nearby points, and there are no overlaps or gaps, if and only if the moduli of its cylinders in that direction are rationally related.
\end{corollary}

\begin{definition}\index{rationally related}
A set of numbers is \emph{rationally related} if all of the numbers are rational multiples of each other.
\end{definition}

We have seen a few examples of surfaces with rationally-related moduli. In the next sections, we will introduce many families of such surfaces.

\newpage
 \section{Square-tiled surfaces} \index{square-tiled surface} \label{squaretiledsurfaces}
One of the simplest ways to create a surface is to start with some squares, and identify opposite parallel edges to create a surface. Some examples are below.

\begin{example}\index{L-shaped table}
Besides the square torus and rectangles, the \lst made from three squares (Figure \ref{shearltablefig}) is the simplest example of a square-tiled surface. 
\end{example}

\begin{example}
The \emph{escalator} \index{escalator} is shown in Figure \ref{sqsurfs} (a), with $4$ levels.  (\cite{animal}, Example 7)  
\end{example}

\begin{example}
The \emph{Eierlegende Wollmilchsau} is shown in Figure \ref{sqsurfs} (b). 
This surface takes its name from the mythical German creature ``egg-laying wool-milk-sow'' that provides everything one might need. This is because this surface has several nice properties, and has served as a counterexample on several occasions (see \cite{animal}, Example 2). \index{Eierlegende Wollmilchsau}
\end{example}

 \begin{figure}[!h]
 \begin{center}
\includegraphics[height=130pt]{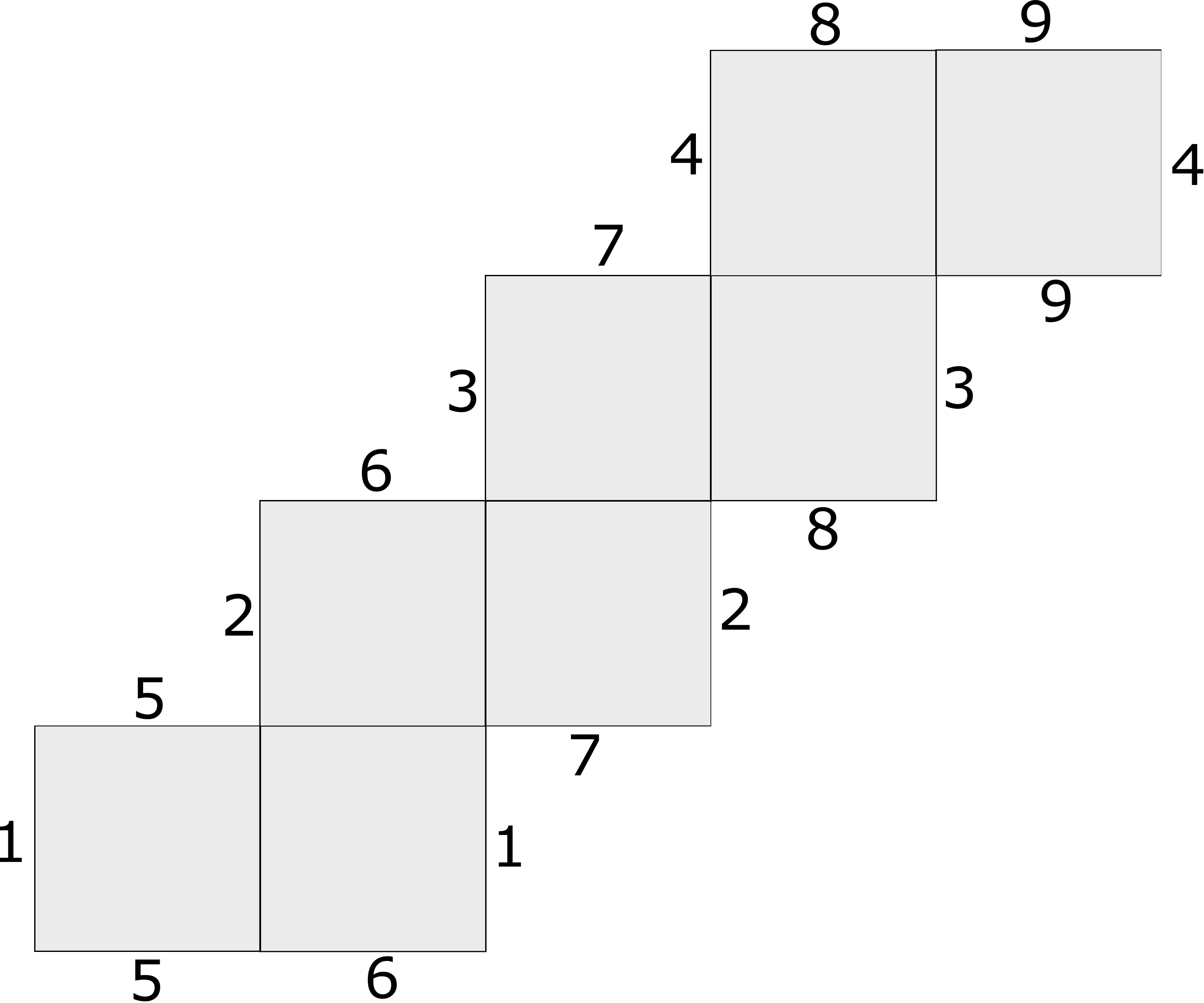} \ \ \ 
\includegraphics[width=140pt]{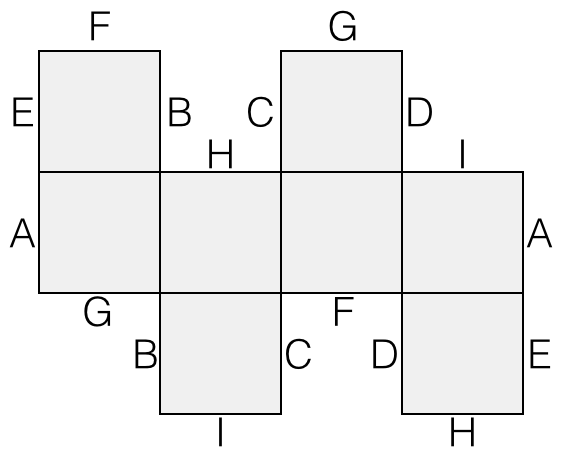}
\caption{(a) The escalator, and (b) the Eierlegende Wollmilchsau. In each surface, edges with the same label are identified. \label{sqsurfs}}
\end{center}
\end{figure}


\begin{exercise} \label{opppar}\index{vertex, angle around}\index{square-tiled surface}
Show that for any square-tiled surface with opposite parallel edges identified, the angle at each vertex is an integer multiple of $2\pi$.
\end{exercise}

\begin{exercise} For each surface in Figure \ref{sqsurfs}:\index{Eierlegende Wollmilchsau}\index{vertex, angle around}
\begin{enumerate}[(a)]
\item Determine how many vertices the surface has, and find the angle around each.
\item Shade each horizontal cylinder a different color, and find the modulus of each.
\end{enumerate}
\end{exercise}

\begin{exercise} For the Eierlegende Wollmilchsau:\index{cylinder}\index{modulus}
\begin{enumerate}[(a)]
\item Shade each vertical cylinder a different color, and find the modulus of each.
\item Redraw the surface so that at least one vertical cylinder is arranged vertically. (You will have to give names to some of the interior edges.)
\end{enumerate}
\end{exercise}

The converse of Exercise \ref{opppar} is: ``If every vertex angle of a given square-tiled surface is a multiple of $2\pi$, then the surface is created by identifying opposite parallel edges.'' This is not true: Figure \ref{crossfig} gives a counterexample, where edges are identified and every vertex angle is a multiple of $2\pi$, but some of the edge identifications are not of opposite parallel edges.

\begin{figure}[!h]
\begin{center}
\includegraphics[height=100pt]{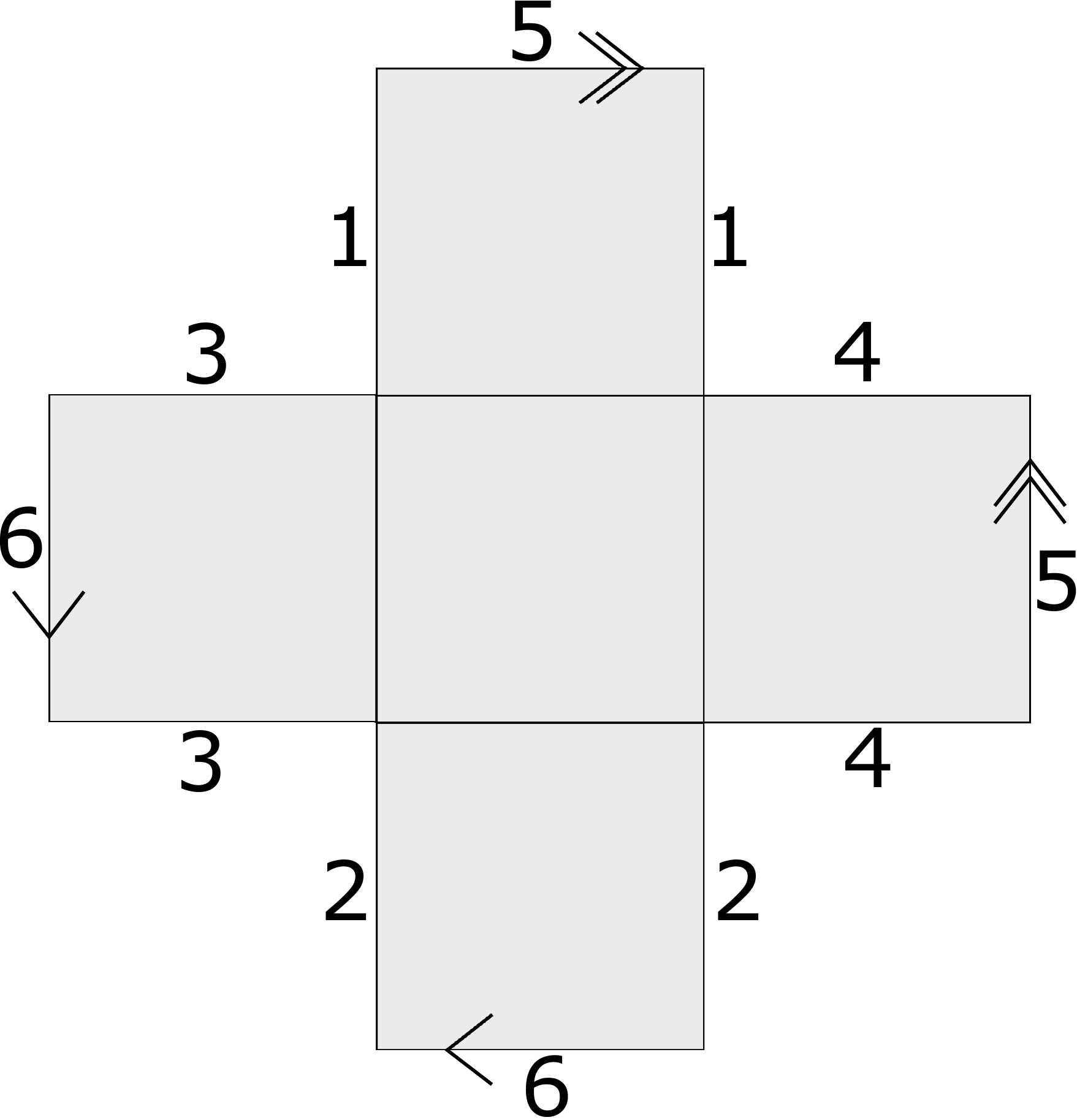}
\caption{A flat surface where some  identifications are \emph{not} opposite and parallel. \label{crossfig}}
\end{center}
\end{figure}

\begin{exercise}\index{square-tiled surface}
Check that every vertex of the surface in Figure \ref{crossfig} has an angle that is a multiple of $2\pi$. \index{vertex, angle around}
\end{exercise}

\begin{exercise}\index{shear}\index{cylinder}\index{square-tiled surface}\index{modulus}
For the square-tiled surface in Figure \ref{cylcalcs}:
\begin{enumerate}[(a)]
 \item Shade each horizontal cylinder a different color, and  find the modulus of each.
  \item Find the values of $M$ for which the horizontal shear $\stt 1M01$ is a symmetry of the surface.
 \item Shade each vertical cylinder a different color, and  find the modulus of each.
  \item Find the values of $M$ for which the vertical shear $\stt 10M1$ is a symmetry of the surface.
 \end{enumerate}
\end{exercise}

\begin{figure}[!h]
\begin{center}
\includegraphics[height=100pt]{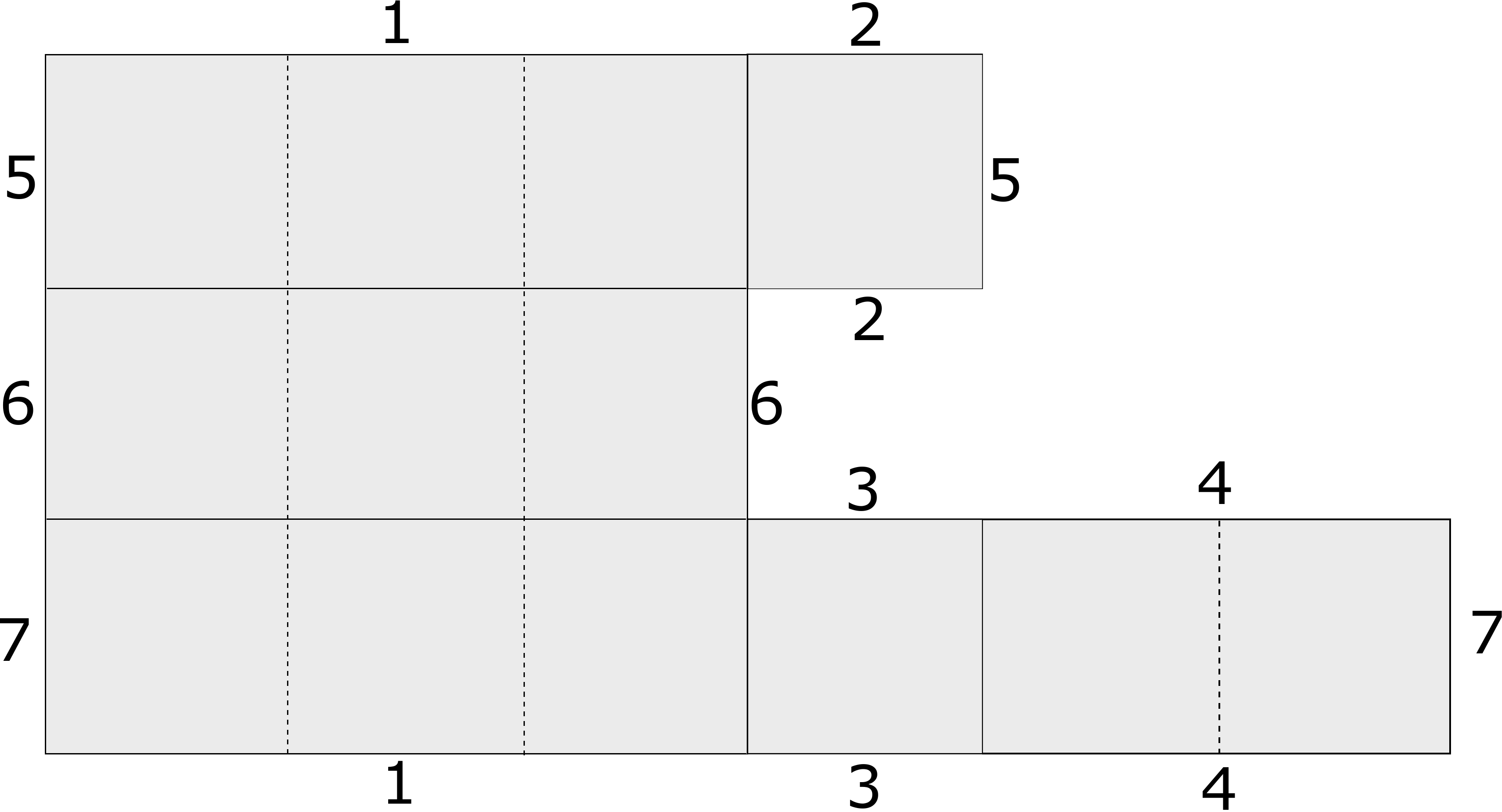}
\caption{Opposite parallel edges with the same label are identified. Dashed lines divide the figure into squares. \label{cylcalcs}}
\end{center}
\end{figure}

We can also shear ``rectangle-tiled surfaces,'' taking vertices to vertices as above, as long as all of the rectangles have the same modulus. The best-studied such surfaces are $\mathsf{L}$-shaped tables.


\begin{example}\index{L-shaped table}\index{golden L}\index{golden ratio}\label{gl-ex}
A particularly beautiful \lst is the ``Golden $\mathsf{L}$'' (Figure \ref{ltableexamples}). In this surface, each cylinder has modulus $\phi$, where $\phi$ is the golden ratio $\frac 12 (1+\sqrt{5})$. Since $\phi=1+1/\phi$, the $\mathsf{L}$ is a square with congruent golden-ratio rectangles glued to the top and side. For an exploration of this surface, see \cite{goldenell}.
\end{example}

 \begin{figure}[!h]
 \begin{center}
\includegraphics[height=110pt]{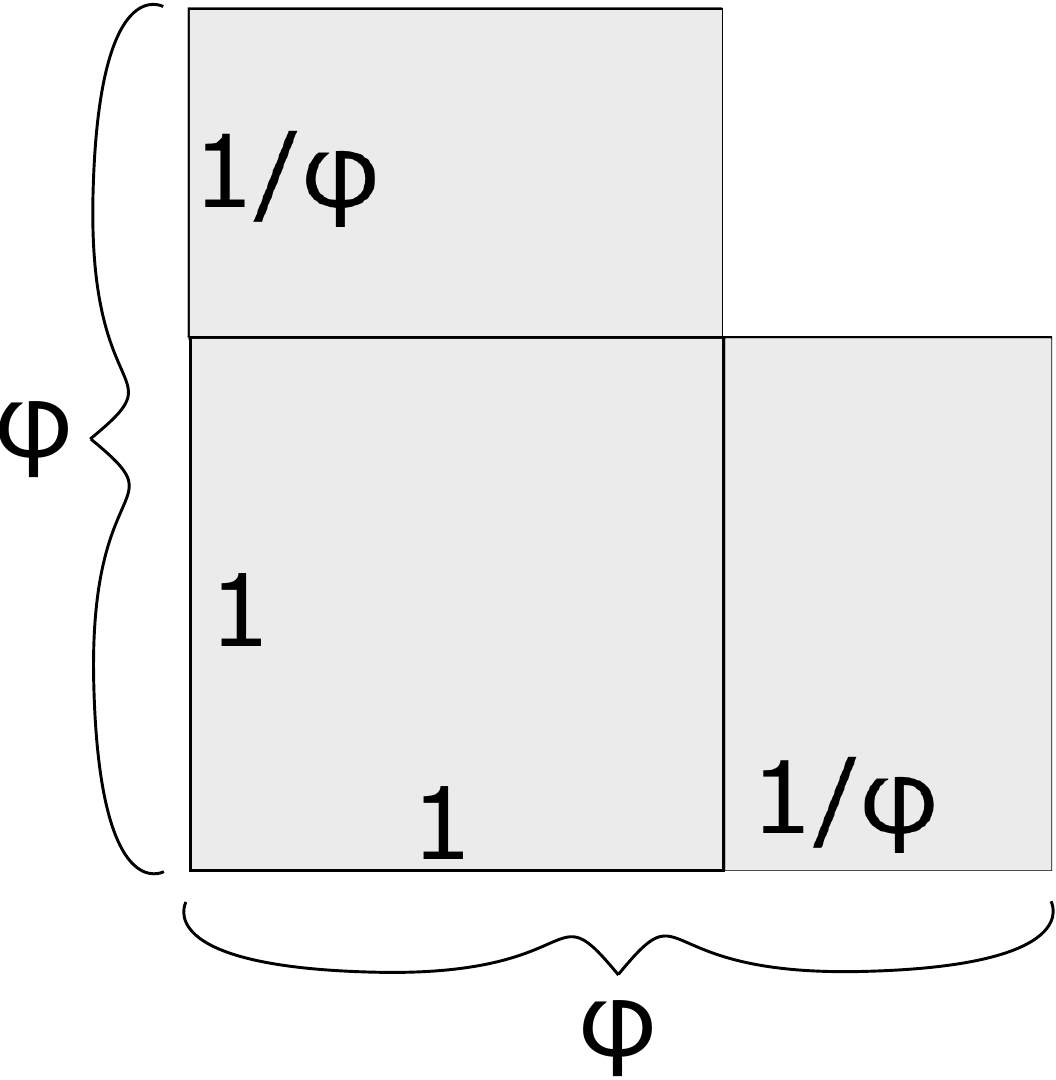}
\caption{The Golden $\mathsf{L}$, whose cylinders have modulus $\phi$ \label{ltableexamples}\index{golden L}\index{golden ratio}}
\end{center}
\end{figure}

\section{Regular polygons and the Modulus Miracle} \index{regular polygon surface} \label{regularpolygons}

Amazingly, many surfaces made from regular polygons can be sheared, cut up and reassembled back into the original surface in the same way that we have done with square-tiled surfaces. This property was first explored by Veech in \cite{Veech}. This property is unexpected and beautiful, so we call it the Modulus Miracle.\index{Veech, William}

\begin{definition} \index{regular polygon surface}
The \emph{double regular $n$-gon surface} is the surface made by identifying opposite parallel edges of two regular $n$-gons, one of which is the reflection of the other. We assume that each $n$-gon has a horizontal edge, and that all edges have unit length.

The \emph{regular $2n$-gon surface} is the surface made by identifying opposite parallel edges of a regular $2n$-gon. We assume that the $2n$-gon has a pair of horizontal edges, and that all edges have unit length. 
\end{definition}


In Example \ref{oct-mod}, we showed that the cylinders of the regular octagon surface have modulus $1+\sqrt{2}$ and $2(1+\sqrt{2})$. Because the cylinders are rationally related, we can shear the regular octagon surface (Corollary \ref{whatcanweshear}), and reassemble the pieces back into a regular octagon, while respecting the edge identifications (Figure \ref{octshearfig}). For an extensive study of this surface, see \cite{SU}.\index{octagon}

\begin{figure}[!h]
\begin{center}
\includegraphics[width=350pt]{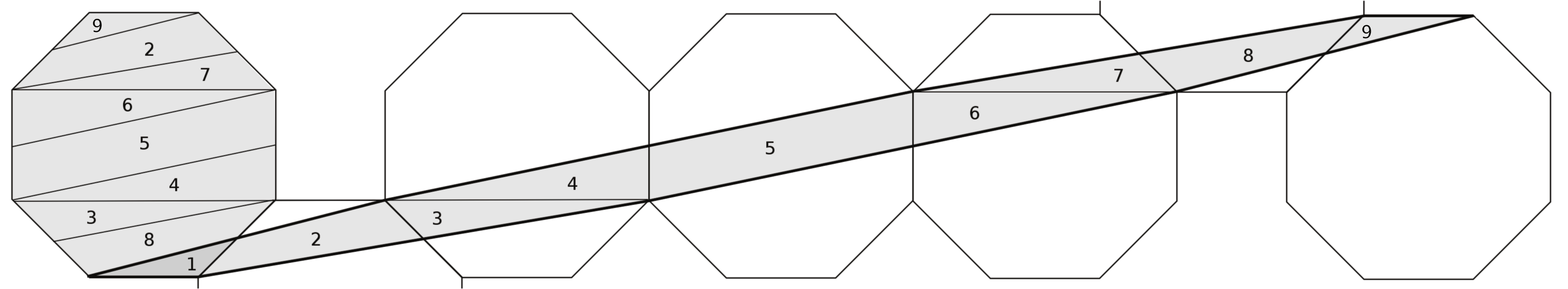}
\caption{We shear the octagon and reassemble the pieces by translation back into an octagon, respecting the edge identifications. \label{octshearfig}\index{shear}\index{octagon}}
\end{center}
\end{figure}

This kind of thing turns out to be true for all regular even polygon surfaces:\index{octagon}

\begin{proposition}[Modulus Miracle, single even polygon case]\label{evenmod} \index{cylinder}\index{modulus}\index{modulus miracle}  \index{regular polygon surface}
Let $n$ be even. Every horizontal cylinder of a regular $n$-gon surface has modulus $2\cot\pi/n$, except a central rectangular cylinder, which has modulus $\cot\pi/n$.\footnote{If $n=4k$, the surface has a central rectangular cylinder; if $n=4k+2$, it does not.}
\end{proposition}

\begin{proof}
This can be proven via a trigonometric calculation (\cite{davisthesis}, Lemma 5.4). It was known to Veech \cite{Veech} much earlier. \index{Veech, William}
\end{proof}

\begin{example}\label{dp}\index{cylinder}\index{modulus}\index{double pentagon}
The double regular pentagon surface has two cylinders. The central cylinder is the union of two trapezoids, and the other cylinder is the union of two triangles. We can calculate their moduli. Notice that the exterior angle at each vertex of a regular pentagon is $2\pi/5$:

The central (dark grey) cylinder has width $2+2\cos(2\pi/5)$ and height $\sin(2\pi/5)$, so its modulus is $$\frac{2+2\cos(2\pi/5)}{\sin(2\pi/5)}=2\cot\pi/5.$$
The other (light grey) cylinder has width $1+2\cos(2\pi/5)$ and height $\sin(4\pi/5)$, so its modulus is $$\frac{1+2\cos(2\pi/5)}{\sin(4\pi/5)}=2\cot\pi/5.$$
\end{example}

\begin{figure}[!h]
\begin{center}
\includegraphics[width=300pt]{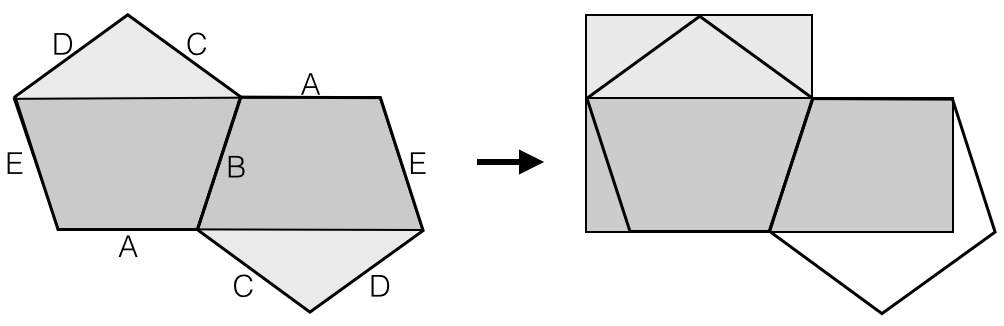}
\caption{The double pentagon surface has two cylinders. We can cut and reassemble the pieces of the cylinders (respecting the edge identifications) to draw them as an $\mathsf{L}$-shaped table. \label{golden-pent}\index{golden L}}
\end{center}
\end{figure}

That both of these simplify to $2\cot\pi/5$ is a trigonometric calculation. The result is true in general:

\begin{proposition}[Modulus Miracle, double polygon case] \label{oddmod}\index{modulus miracle}  \index{regular polygon surface}
Every horizontal cylinder of a double regular $n$-gon surface has modulus $2\cot\pi/n$.
\end{proposition}
\begin{proof}
Again, this can be proven by calculation (\cite{davisthesis}, Lemma 5.3). 
\end{proof}\index{modulus}

\begin{exercise}
\begin{enumerate}[(a)]
\item For the double pentagon surface, find the ratio of the width of the larger cylinder to the width of the smaller cylinder. Simplify your answer to a familiar form.
\item Explain why the double pentagon surface is just a cut, reassembled, and horizontally stretched version of the Golden $\mathsf{L}$ from Example \ref{gl-ex}. \index{golden L}\index{cylinder}\index{double pentagon}\index{golden ratio}
\end{enumerate}
\end{exercise}

Because the cylinders have the same modulus, we can shear the double pentagon surface, and reassemble the pieces back into a double pentagon, while respecting the edge identifications. For an extensive study of this surface, see \cite{davis} and \cite{dft}. For a video showing this shearing and reassembling, see \cite{dv12}. \index{double pentagon}

\begin{figure}[!h]
\begin{center}
\includegraphics[width=300pt]{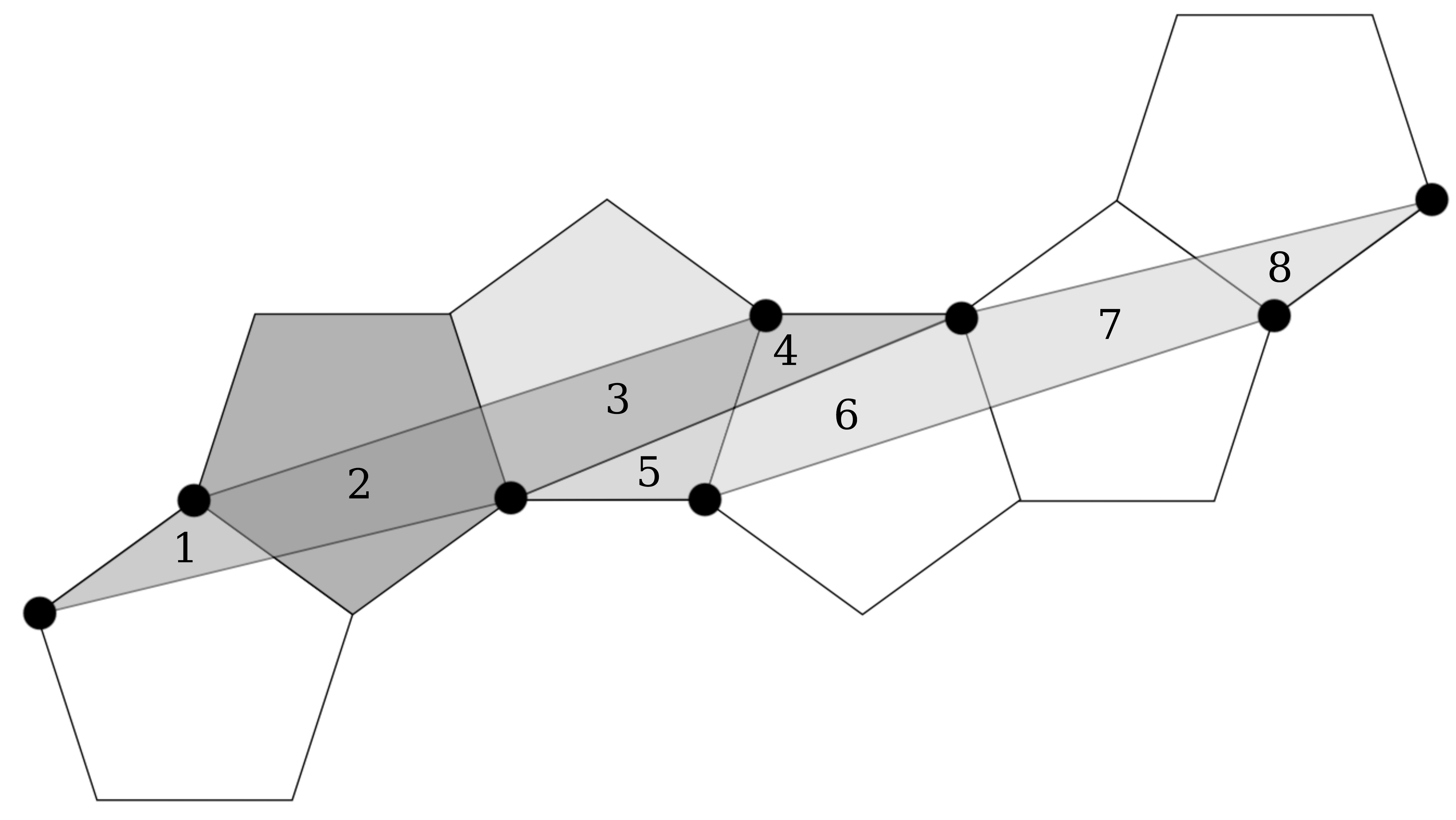}
\caption{We shear the double pentagon via the matrix $[1,2\cot \pi/n;0,1]$. \label{dpshear}}
\end{center}
\end{figure}

\begin{exercise}\index{double pentagon}\index{shear}
\begin{enumerate}[(a)]
\item Show how to reassemble (by translation, and respecting the edge identifications) the 8 pieces of the sheared double pentagon in Figure \ref{dpshear} back into two regular pentagons, as is done with the octagon in Figure \ref{octshearfig}.
\item Explain why the even-numbered pieces end up in one pentagon and the odd-numbered pieces in the other.
\end{enumerate}
\end{exercise}

When we sheared, cut and reassembled the square torus in Section \ref{sq-tor-symm}, we were able to give a rule for the effect of this action on a cutting sequence corresponding to a trajectory on the surface (Theorem \ref{kosl}). That rule was: \emph{Given a trajectory $\tau$ on the square torus with slope greater than $1$, and its corresponding cutting sequence $c(\tau)$, let $\tau'$ be result of applying $\nupshear$ to $\tau$. To obtain $c(\tau')$ from $c(\tau)$, shorten each string of $A$s by $1$.}\index{cutting sequence}

We can do the same for the double pentagon, octagon, and other regular polygon surfaces: \index{regular polygon surface}

\begin{theorem} \label{kosl-reg} \index{regular polygon surface}\index{sandwiched}\index{cutting sequence}
Given a trajectory $\tau$ on the double regular $n$-gon surface for $n$ odd, or on the regular $n$-gon surface for $n$ even, where the slope of $\tau$ is between  $0$ and $\pi/n$, and its corresponding cutting sequence $c(\tau)$, let $\tau'$ be result of applying $\stt {-1}{2\cot \pi/n}01$ to $\tau$. To obtain $c(\tau')$ from $c(\tau)$, keep only the \emph{sandwiched} letters.
\end{theorem}

The even case is proved in \cite{SU} and the odd case in \cite{davis}.

\begin{definition}\index{sandwiched}
A \emph{sandwiched} letter is one where the same letter both precedes and follows it. 
\end{definition}


\begin{example}\index{cutting sequence}\index{sandwiched}\index{double pentagon}
In Figure \ref{pent-cyl}d, we found cutting sequences $\overline{\mathbf{B}E\mathbf{C}E}$ and $\overline{\mathbf{A}BEC\mathbf{D}CEB}$ on the double pentagon surface. Applying the shear $\stt {-1}{2\cot \pi/5}01$ to each of these trajectories results in trajectories with corresponding cutting sequences $\overline{BC}$ and $\overline{AD}$, respectively (the sandwiched letters are bold). The original cutting sequences correspond to parallel trajectories, and you can check that the trajectories $\overline{BC}$ and $\overline{AD}$ are also parallel, and are in the cylinder direction of Figure \ref{pent-cyl}b. 

The matrix $\stt {-1}{2\cot \pi/5}01=\stt {1}{2\cot \pi/5}01 \stt {-1}001$ is a horizontal flip followed by a horizontal shear. Figure \ref{bece-shear} shows how the trajectory $\overline{BECE}$ is flipped, sheared and reassembled into the trajectory $\overline{BC}$.

\begin{figure}[!h]
\begin{center}
\includegraphics[width=320pt]{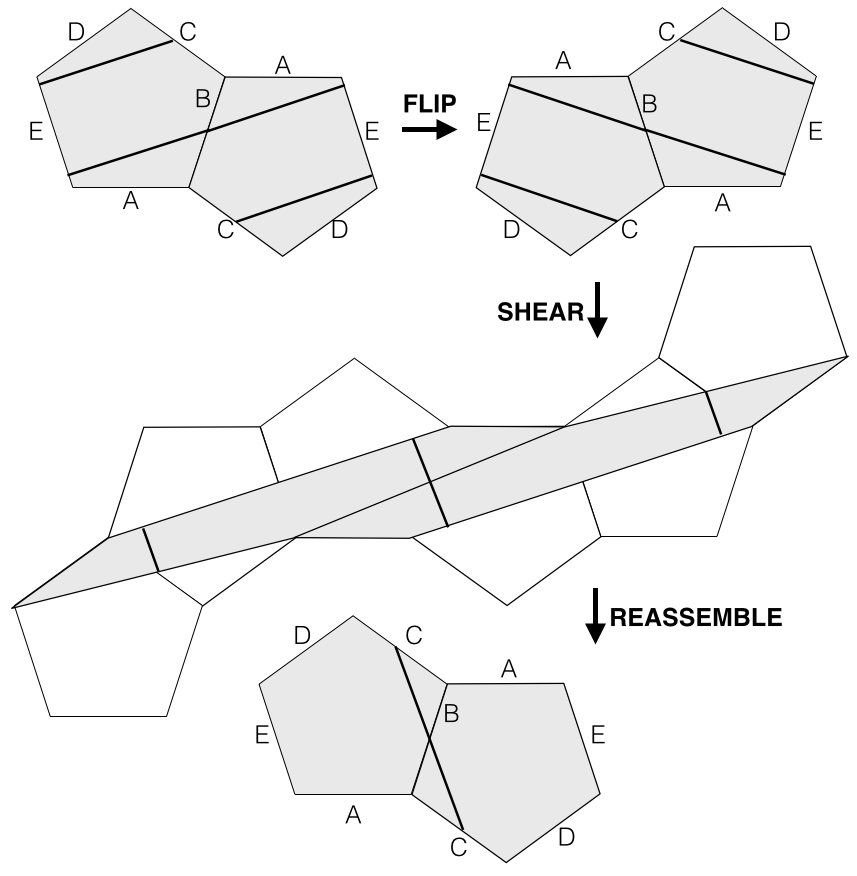}
\caption{We flip and shear the trajectory $\overline{BECE}$ on the double pentagon, and obtain the trajectory $\overline{BC}$. \label{bece-shear}}
\end{center}
\end{figure}

You might wonder why we are using $\stt {-1}{2\cot \pi/5}01$ instead of $\stt {1}{2\cot \pi/5}01$. We choose to describe the effect of the matrix $\stt {-1}{2\cot \pi/5}01$ because it induces the effect ``keep only sandwiched letters'' on the associated cutting sequence, while the matrix $\stt {1}{2\cot \pi/5}01$ induces the effect ``keep only sandwiched letters, and also permute the edge labels,'' where the edge labels are permuted based on the horizontal flip: $B$ and $E$ are reversed, and $C$ and $D$ are reversed.  We choose to describe the more elegant action. For more details, see \cite{davisthesis}.
\end{example}

\begin{exercise}
The rule in Theorem \ref{kosl} is ``shorten each string of $A$s by $1$,'' and the rule in Theorem \ref{kosl-reg} is ``keep only the sandwiched letters.''  For a cutting sequence on the square torus, are these equivalent? If not, can you reconcile them?
\end{exercise}

\newpage
To see the shearing and reassembling in Figure \ref{bece-shear} in live action, see the video: \url{https://vimeo.com/47049144} \cite{dv12}.

\section{Billiards on triangular tables}

Our original motivation for studying the square torus was that it was the unfolding of the square billiard table. In fact, we can view \emph{all} regular polygon surfaces as unfoldings of \emph{triangular} billiard tables.

\begin{example}\label{unfold-octagon}
We  unfold the $(\pi/2,\pi/8,3\pi/8)$ triangular billiard table until every edge is paired with a parallel, oppositely-oriented partner edge:

\begin{figure}[!h]
\begin{center}
\includegraphics[height=110pt]{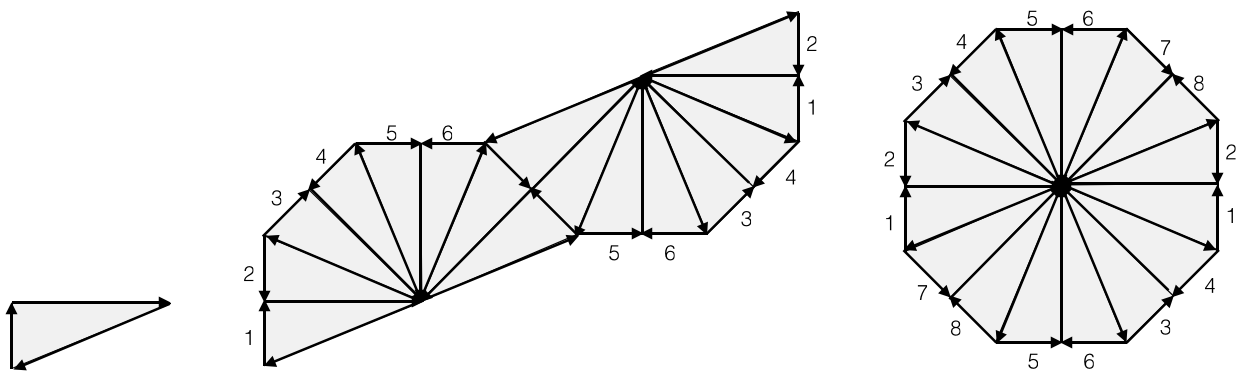}
\caption{We unfold the right triangle with vertex angle $\pi/8$ into the regular octagon surface. \label{unfold-octagon}}
\end{center}
\end{figure}

This gives us the regular octagon surface! So the regular octagon surface is the unfolding of the $(\pi/2,\pi/8,3\pi/8)$ triangle.
\end{example}

\begin{proposition}
A billiard path on the $(\pi/2,\pi/n,\pi-\pi/2-\pi/n)$ triangle corresponds to a trajectory on the regular $n$-gon surface for $n$ even and to a trajectory on the double regular $n$-gon surface for $n$ odd.
\end{proposition}

\begin{proof}
In the even case, the unfolding is exactly as in Example \ref{unfold-octagon}. In the odd case, we unfold the triangle into a single $n$-gon, but there are no pairs of parallel edges, so we must continue to unfold, and we end up with pairs of oppositely-oriented parallel copies of the same edge when we have the double $n$-gon surface.
\end{proof}

\begin{definition}
A \emph{rational polygon} is a polygon whose angles are all rational multiples of $\pi$.
\end{definition}

For a billiard table that is a rational polygon, the unfolding requires a finite number of copies of the table in order to end up with pairs of oppositely-oriented parallel copies of the same edge. Every such table has many periodic billiard paths; in fact, each has infinitely many periodic directions. For further reading about billiards on polygonal tables, see \cite{masur}  and \cite{kenyon}.

\section{Ward surfaces}\index{Ward surface}\index{Ward, Clayton} \label{wardsurfaces}

For some time, square-tiled surfaces and regular polygon surfaces were the only known examples of surfaces that have all of the symmetries we listed for the square torus: rotation, reflection, and the shear. (Such surfaces are known as \emph{Veech surfaces}, or \emph{lattice surfaces}.) Then Veech's student, Clayton Ward, discovered a larger family of such surfaces, now known as Ward surfaces \cite{Ward}.\index{Veech surface} We will give two different constructions of Ward surfaces.

One way to describe a Ward surface is as a regular $2n$-gon with two regular $n$-gons, where the odd-numbered edges of the $2n$-gon are glued to one of the $n$-gons, and the even-numbered edges of the $2n$-gon are glued to the other $n$-gon.

\begin{example}\label{oct2sq}\index{Ward surface}\index{octagon}
 For $n=4$, the Ward surface is an octagon and two squares (Figure \ref{m34}).

\begin{figure}[!h]
\begin{center}
\includegraphics[height=140pt]{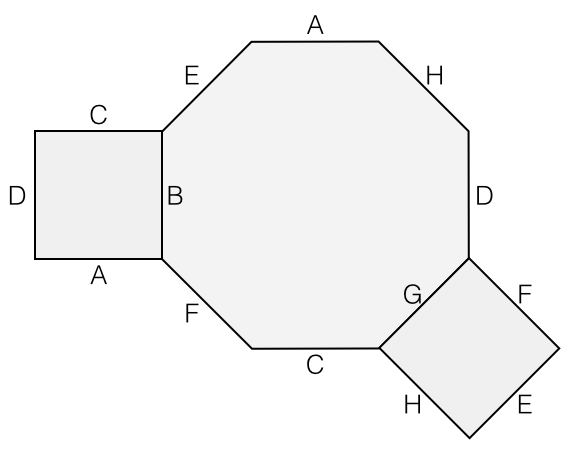}
\caption{The $n=4$ Ward surface: a regular octagon, half of whose edges are glued to the square, and the other half to the diamond \label{m34}}
\end{center}
\end{figure}

\end{example}

\begin{exercise}\index{cylinder}
Decompose the surface in Example \ref{oct2sq} into horizontal cylinders, and find the modulus of each. Simplify your answers to a form where you can compare them. Are they rationally related?
\end{exercise}


The other way to describe a Ward surface is as the unfolding of the $(\pi/n, \pi/2n, \pi-\pi/n-\pi/2n)$ triangle.

\begin{example}\label{oct2sq}\index{Ward surface}\index{octagon}
 For $n=4$, the Ward surface is the unfolding of the $(\pi/4, \pi/8, 5\pi/8)$ triangle (Figure \ref{wardsurfacefig}).

\begin{figure}[!h]
\begin{center}
\includegraphics[height=80pt]{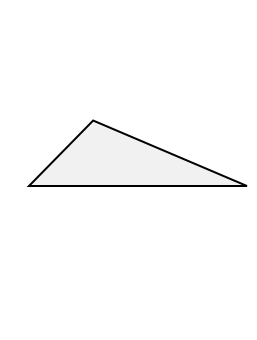} \ \  
\includegraphics[height=120pt]{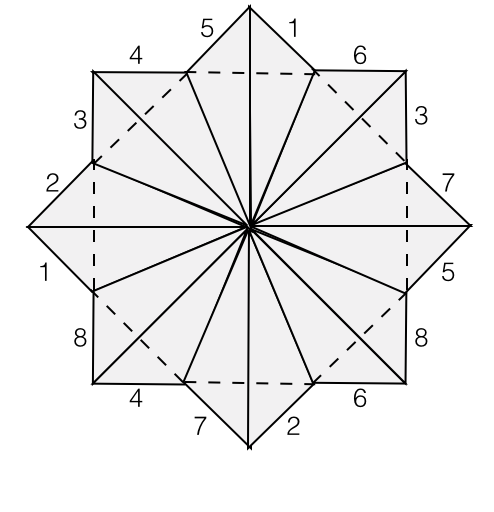} \ \ 
\includegraphics[height=120pt]{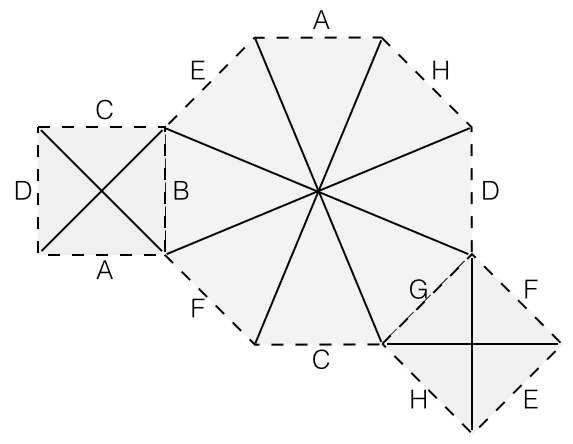}
\caption{The $(\pi/8, \pi/4, 5\pi/8)$ triangle, its unfolding into a Ward surface, and cutting and reassembling the surface into the presentation of Figure \ref{m34}. \label{wardsurfacefig}}
\end{center}
\end{figure}

\end{example}

For each integer $n\geq 3$, there is a Ward surface, consisting of a regular $2n$-gon and two regular $n$-gons, or equivalently the unfolding of the \mbox{$(\pi/n, \pi/2n, \pi-\pi/n-\pi/2n)$} triangle into a kind of ``sunburst'' figure.

\begin{exercise}
Draw the $n=5$ Ward surface, whichever presentation you choose. Label the edge identifications.
\end{exercise}

\section{Bouw-M\"oller surfaces}\index{Bouw-M\"oller surface}\index{Bouw, Irene} \index{M\"oller, Martin} \label{bmsurfaces}

In 2006, Irene Bouw and Martin M\"oller discovered a larger family of Veech surfaces, now called Bouw-M\"oller surfaces \cite{BM}. The regular polygon surfaces and the Ward surfaces are special cases of Bouw-M\"oller surfaces. Bouw and M\"oller gave an algebraic description of the surfaces, and later, Pat Hooper \index{Hooper, Pat} found a polygon decomposition for the surfaces, which we present here \cite{Hooper}.

For any $m\geq 2$, and any $n\geq 3$, the $(m,n)$ Bouw-M\"oller surface is created by identifying opposite parallel edges of $m$ semi-regular\footnote{A \emph{semi-regular polygon} is an equiangular polygon with an even number of sides. Edge lengths alternate between two different values, which may be equal and may be $0$.} $2n$-gons, each of whose edge lengths are carefully chosen so that the cylinders all have the same modulus. For the precise definition, see \cite{Hooper}, \textsection $4.2$. In lieu of giving the definition here, we give several examples.\index{cylinder}\index{modulus}

\begin{example}\index{double pentagon} \index{Bouw-M\"oller surface}
The $(2,n)$ Bouw-M\"oller surface is made from two regular $n$-gons, and it is the double regular $n$-gon. For example, the double pentagon in Example \ref{dp} is the $(2,5)$ Bouw-M\"oller surface.
\end{example}

\begin{example}\index{octagon} \index{Ward surface}
The $(3,n)$ Bouw-M\"oller surface is made from three polygons, the first and last of which are regular $n$-gons and the middle of which is a regular $2n$-gon, and it is the $n$ Ward surface. For example, the regular octagon and two squares in Example \ref{oct2sq} is the $(3,4)$ Bouw-M\"oller surface.
\end{example}

\begin{example}\label{bmex} \index{Bouw-M\"oller surface}
The $(6,5)$ Bouw-M\"oller surface is made of $6$ polygons, the first and last of which are regular $5$-gons and the middle $4$ of which are semi-regular $10$-gons.

\begin{figure}[!h]
\begin{center}
\includegraphics[width=350pt]{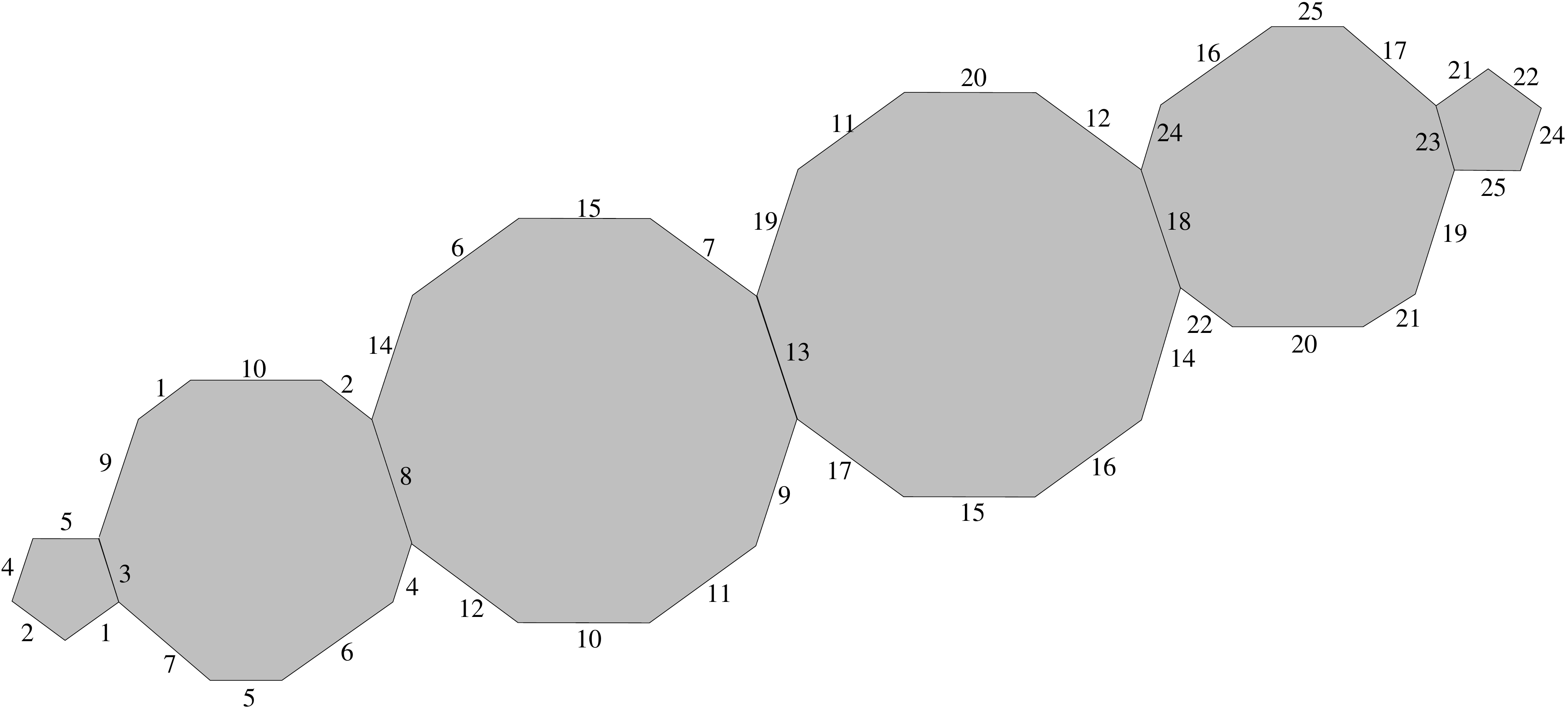} 
\caption{The $(6,5)$ Bouw-M\"oller surface. \label{m6n5}}
\end{center}
\end{figure}

\end{example}

\begin{exercise} \index{Bouw-M\"oller surface}\index{cylinder}\index{modulus}
For the surface in Figure \ref{m6n5}, shade each horizontal cylinder differently, as in Figure \ref{golden-pent}. Does it seem plausible that all of the cylinders have the same modulus?
\end{exercise}

%
%
%
%




\printindex


\begin{thebibliography}{}
\bibitem[A02]{arnoux} Pierre Arnoux. Sturmian sequences, pages 143-198. Lecture Notes in Mathematics, 1794. Springer, Berlin, 2002.
\bibitem[ACL15]{goldenell} Jayadev Athreya, Jon Chaika, Samuel Leli\`evre, The gap distribution of slopes on the golden L.
Contemporary Mathematics, volume 631, $47-62$, 2015.
\bibitem[BM06]{BM} Irene Bouw, Martin M\"oller: Teichm$\mathrm{\ddot{u}}$ller curves, triangle groups, and Lyapunov exponents, Annals of Mathematics (2006).
\bibitem[DFT11]{dft} Diana Davis, Dmitry Fuchs, Sergei Tabachnikov, Periodic trajectories in the regular pentagon, Moscow Math J., vol. 3 (2011).
\bibitem[DS12]{dv12} Diana Davis, Libby Stein et al: Cutting sequences on the double pentagon, explained through dance: \url{https://vimeo.com/47049144}, 2012.
\bibitem[D13]{davis}  Diana Davis: Cutting sequences, regular polygons, and the Veech group. Geometriae Dedicata, Volume 162, Issue 1 (2013), pp. $231-261$.
\bibitem[D14]{davisthesis} Diana Davis, Cutting sequences on translation surfaces, New York Journal of Mathematics, Volume 20 (2014) $399-429$.
\bibitem[D15]{snapshot} Diana Davis, Billiards and flat surfaces, Snapshots of modern mathematics from Oberwolfach, no. $1/2015$.
\bibitem[FZ08]{ferenczi} S\'ebastien Ferenczi, Luca Zamboni, Languages of $k$-interval exchange transformations, Bulletin of the London Mathematical Society 40 (2008), p. 705-714.
\bibitem[H13]{Hooper} W. Patrick Hooper, Grid graphs and lattice surfaces. Int. Math. Res. Not. IMRN (2013) no. 12, $2657-2698$.
\bibitem[M86]{masur} Howard Masur, Closed trajectories for quadratic differentials with an application to billiards. Duke Mathematics Journal 53 (1986), no. 2, 307--314.
\bibitem[KS00]{kenyon} Richard Kenyon and John Smillie, Billiards on rational triangles, Commentarii Mathematici Helvetici 75 (2000), 65--108.
\bibitem[M40]{morse} Marston Morse and Gustav A. Hedlund. Symbolic dynamics II. Sturmian trajectories. American Journal of Mathematics, $62:1-42$, 1940.
\bibitem[SW13]{animal} Jan-Christoph Schlage-Puchuta and Gabriela Weitze-Schmith\"usen, Finite translation surfaces with maximal number of translations, arxiv 2013.
\bibitem[GS06]{origamis} Gabriela Schmith\"usen, Origamis with non congruence Veech groups, Proceedings of Symposium on Transformation Groups, Yokohama (2006).
\bibitem[S85]{series} Caroline Series, The geometry of Markoff numbers. The Mathematical Intelligencer, 7(3):$20-29$, 1985.
\bibitem[SU11]{SU} John Smillie, Corinna Ulcigrai: Beyond Sturmian sequences: coding linear trajectories in the regular octagon, Proceedings of the London Mathematical Society, 102(2) (2011), pp. 291-340.
\bibitem[S77]{smith} H. J. S. Smith. Note on continued fractions. Messenger of Mathematics 2nd. series, $6:1-14$, 1877.
\bibitem[V89]{Veech} William Veech: Teichm$\mathrm{\ddot{u}}$ller curves in moduli space, Eisenstein series and an application to triangular billiards. Inventiones Mathematicae, 87 (1989), pp. 553-583.
\bibitem[W98]{Ward} Clayton Ward: Calculation of Fuchsian groups associated to billiards in a rational triangle, Ergodic Theory and Dynamical Systems, Vol. 18, Issue 04 (1998), pp. 1019-1042.



\end{thebibliography}
\end{document}